\documentclass[11pt]{amsart}
\usepackage{}
\usepackage{amssymb}

\usepackage{amsmath}
\usepackage{txfonts}

\input xy
\xyoption{all}

\usepackage{amsfonts}
\usepackage{mathrsfs}

\usepackage{latexsym}
\usepackage{graphicx}
\usepackage{amscd,amssymb,amsmath,amsbsy,amsthm}
\usepackage[all]{xy}
\usepackage[colorlinks,plainpages,backref,urlcolor=blue]{hyperref}
\usepackage{verbatim}

\usepackage{tikz}
\usetikzlibrary{arrows,calc}
\tikzset{
>=stealth',
help lines/.style={dashed, thick},
axis/.style={<->},
important line/.style={thick},
connection/.style={thick, dotted},
}

\newcommand {\Omit}[1]{}

\newcommand{\nc}{\newcommand}
\nc{\rnc}{\renewcommand}
\nc{\bb}[1]{{\mathbb #1}}
\nc{\bbA}{\bb{A}}\nc{\bbB}{\bb{B}}\nc{\bbC}{\bb{C}}\nc{\bbD}{\bb{D}}
\nc{\bbE}{\bb{E}}\nc{\bbF}{\bb{F}}\nc{\bbG}{\bb{G}}\nc{\bbH}{\bb{H}}
\nc{\bbI}{\bb{I}}\nc{\bbJ}{\bb{J}}\nc{\bbK}{\bb{K}}\nc{\bbL}{\bb{L}}
\nc{\bbM}{\bb{M}}\nc{\bbN}{\bb{N}}\nc{\bbO}{\bb{O}}\nc{\bbP}{\bb{P}}
\nc{\bbQ}{\bb{Q}}\nc{\bbR}{\bb{R}}\nc{\bbS}{\bb{S}}\nc{\bbT}{\bb{T}}
\nc{\bbU}{\bb{U}}\nc{\bbV}{\bb{V}}\nc{\bbW}{\bb{W}}\nc{\bbX}{\bb{X}}
\nc{\bbY}{\bb{Y}}\nc{\bbZ}{\bb{Z}}
\nc{\mbf}[1]{{\mathbf #1}}
\nc{\bfA}{\mbf{A}}\nc{\bfB}{\mbf{B}}\nc{\bfC}{\mbf{C}}\nc{\bfD}{\mbf{D}}
\nc{\bfE}{\mbf{E}}\nc{\bfF}{\mbf{F}}\nc{\bfG}{\mbf{G}}\nc{\bfH}{\mbf{H}}
\nc{\bfI}{\mbf{I}}\nc{\bfJ}{\mbf{J}}\nc{\bfK}{\mbf{K}}\nc{\bfL}{\mbf{L}}
\nc{\bfM}{\mbf{M}}\nc{\bfN}{\mbf{N}}\nc{\bfO}{\mbf{O}}\nc{\bfP}{\mbf{P}}
\nc{\bfQ}{\mbf{Q}}\nc{\bfR}{\mbf{R}}\nc{\bfS}{\mbf{S}}\nc{\bfT}{\mbf{T}}
\nc{\bfU}{\mbf{U}}\nc{\bfV}{\mbf{V}}\nc{\bfW}{\mbf{W}}\nc{\bfX}{\mbf{X}}
\nc{\bfY}{\mbf{Y}}\nc{\bfZ}{\mbf{Z}}
\nc{\bfa}{\mbf{a}}\nc{\bfb}{\mbf{b}}\nc{\bfc}{\mbf{c}}\nc{\bfd}{\mbf{d}}
\nc{\bfe}{\mbf{e}}\nc{\bff}{\mbf{f}}\nc{\bfg}{\mbf{g}}\nc{\bfh}{\mbf{h}}
\nc{\bfi}{\mbf{i}}\nc{\bfj}{\mbf{j}}\nc{\bfk}{\mbf{k}}\nc{\bfl}{\mbf{l}}
\nc{\bfm}{\mbf{m}}\nc{\bfn}{\mbf{n}}\nc{\bfo}{\mbf{o}}\nc{\bfp}{\mbf{p}}
\nc{\bfq}{\mbf{q}}\nc{\bfr}{\mbf{r}}\nc{\bfs}{\mbf{s}}\nc{\bft}{\mbf{t}}
\nc{\bfu}{\mbf{u}}\nc{\bfv}{\mbf{v}}\nc{\bfw}{\mbf{w}}\nc{\bfx}{\mbf{x}}
\nc{\bfy}{\mbf{y}}\nc{\bfz}{\mbf{z}}

\newcommand{\bS}{\mathbb{S}}
\newcommand{\G}{\mathbb{G}}
\newcommand{\op}{\text{op}}

\nc{\mcal}[1]{{\mathcal #1}}
\nc{\calA}{\mcal{A}}\nc{\calB}{\mcal{B}}\nc{\calC}{\mcal{C}}\nc{\calD}{\mcal{D}}
\nc{\calE}{\mcal{E}} \nc{\calF}{\mcal{F}}\nc{\calG}{\mcal{G}}\nc{\calH}{\mcal{H}}
\nc{\calI}{\mcal{I}}\nc{\calJ}{\mcal{J}}\nc{\calK}{\mcal{K}}\nc{\calL}{\mcal{L}}
\nc{\calM}{\mcal{M}}\nc{\calN}{\mcal{N}}\nc{\calO}{\mcal{O}}\nc{\calP}{\mcal{P}}
\nc{\calQ}{\mcal{Q}}\nc{\calR}{\mcal{R}}\nc{\calS}{\mcal{S}}\nc{\calT}{\mcal{T}}
\nc{\calU}{\mcal{U}}\nc{\calV}{\mcal{V}}\nc{\calW}{\mcal{W}}\nc{\calX}{\mcal{X}}
\nc{\calY}{\mcal{Y}}\nc{\calZ}{\mcal{Z}}
\nc{\fA}{\frak{A}}\nc{\fB}{\frak{B}}\nc{\fC}{\frak{C}} \nc{\fD}{\frak{D}}
\nc{\fE}{\frak{E}}\nc{\fF}{\frak{F}}\nc{\fG}{\frak{G}}\nc{\fH}{\frak{H}}
\nc{\fI}{\frak{I}}\nc{\fJ}{\frak{J}}\nc{\fK}{\frak{K}}\nc{\fL}{\frak{L}}
\nc{\fM}{\frak{M}}\nc{\fN}{\frak{N}}\nc{\fO}{\frak{O}}\nc{\fP}{\frak{P}}
\nc{\fQ}{\frak{Q}}\nc{\fR}{\frak{R}}\nc{\fS}{\frak{S}}\nc{\fT}{\frak{T}}
\nc{\fU}{\frak{U}}\nc{\fV}{\frak{V}}\nc{\fW}{\frak{W}}\nc{\fX}{\frak{X}}
\nc{\fY}{\frak{Y}}\nc{\fZ}{\frak{Z}}
\nc{\fa}{\frak{a}}\nc{\fb}{\frak{b}}\nc{\fc}{\frak{c}} \nc{\fd}{\frak{d}}
\nc{\fe}{\frak{e}}\nc{\fFf}{\frak{f}}\nc{\fg}{\frak{g}}\nc{\fh}{\frak{h}}
\nc{\fri}{\frak{i}}\nc{\fj}{\frak{j}}\nc{\fk}{\frak{k}}\nc{\fl}{\frak{l}}
\nc{\fm}{\frak{m}}\nc{\fn}{\frak{n}}\nc{\fo}{\frak{o}}\nc{\fp}{\frak{p}}
\nc{\fq}{\frak{q}}\nc{\fr}{\frak{r}}\nc{\fs}{\frak{s}}\nc{\ft}{\frak{t}}
\nc{\fu}{\frak{u}}\nc{\fv}{\frak{v}}\nc{\fw}{\frak{w}}\nc{\fx}{\frak{x}}
\nc{\fy}{\frak{y}}\nc{\fz}{\frak{z}}

\newtheorem{theorem}{Theorem}[section]
\newtheorem{lemma}[theorem]{Lemma}
\newtheorem{corollary}[theorem]{Corollary}
\newtheorem{prop}[theorem]{Proposition}

\newtheorem*{lemma*}{Lemma}

\theoremstyle{definition}
\newtheorem{definition}[theorem]{Definition}
\newtheorem{example}[theorem]{Example}
\newtheorem{remark}[theorem]{Remark}
\newtheorem{remarks}[theorem]{Remarks}

\newtheorem{thm}{Theorem}

 \DeclareMathOperator{\gr}{gr}
\DeclareMathOperator{\im}{im} 
 \DeclareMathOperator{\id}{id}
 \DeclareMathOperator{\Sym}{Sym}
 
 \DeclareMathOperator{\GL}{GL}
\DeclareMathOperator{\Hom}{{Hom}} 
\DeclareMathOperator{\Ext}{{Ext}}
\DeclareMathOperator{\Cone}{{Cone}}
\DeclareMathOperator{\sHom}{{\mathscr{H}om}}

\DeclareMathOperator{\Spec}{{Spec}} 
\DeclareMathOperator{\Aut}{Aut}
 \DeclareMathOperator{\End}{End}

\DeclareMathOperator{\Mod}{Mod\hbox{-}}
\DeclareMathOperator{\lMod}{\hbox{-}Mod}

\DeclareMathOperator{\tors}{tors}
\DeclareMathOperator{\holim}{holim}
\DeclareMathOperator{\hocolim}{hocolim}
\DeclareMathOperator{\cofib}{cofib}
\DeclareMathOperator{\EM}{EM}

\DeclareMathOperator{\MSL}{{MSL}}

\DeclareMathOperator{\Laz}{\bbL az}
\DeclareMathOperator{\MGL}{{MGL}}
\DeclareMathOperator{\SH}{SH}
\DeclareMathOperator{\Ell}{Ell}
\DeclareMathOperator{\SL}{SL}
\DeclareMathOperator{\Spc}{Spc}
\DeclareMathOperator{\Pic}{Pic}

\DeclareMathOperator{\Tow}{Tow}

\DeclareMathOperator{\Gr}{Gr}

\DeclareMathOperator{\Sm}{Sm}
\DeclareMathOperator{\CH}{CH}

\DeclareMathOperator{\BGL}{BGL}
\DeclareMathOperator{\BSL}{BSL}
\DeclareMathOperator{\colim}{colim}

\DeclareMathOperator{\Sch}{Sch}

\DeclareMathOperator{\BU}{BU}

\newcommand{\sO}{\mathcal{O}}
\newcommand{\sF}{\mathscr{F}}

\newcommand{\sH}{\mathcal{H}}
\newcommand{\sX}{\mathcal{X}}

\newcommand{\sV}{\mathscr{V}}
\newcommand{\sE}{\mathscr{E}}

\newcommand{\CC}{\mathbb{C}}

\DeclareMathOperator{\Char}{char}

\DeclareMathOperator{\topo}{top}
\DeclareMathOperator{\MSU}{MSU}
\DeclareMathOperator{\BSU}{BSU}
\DeclareMathOperator{\Th}{Th}
\DeclareMathOperator{\tho}{th}
\DeclareMathOperator{\Ho}{Ho}

\newcommand{\inj}{\hookrightarrow}

\def\angl#1{{\langle #1\rangle}}

\newcommand{\bideg}{^{*,*}}
\newcommand{\<}{\langle}
\renewcommand{\>}{\rangle}

\newcommand{\PP}{\bbP}
\renewcommand{\P}{\PP}
\newcommand{\ZZ}{\bbZ}
\newcommand{\QQ}{\bbQ}
\newcommand{\Z}{\bbZ}

\newcommand{\Q}{\bbQ}
\newcommand{\N}{\bbN}
\newcommand{\A}{\bbA}
\newcommand{\mS}{\bbS}
\newcommand{\F}{\mathbb{F}}

\DeclareMathOperator{\MU}{MU}
\DeclareMathOperator{\SU}{SU}
\DeclareMathOperator{\eff}{eff}

\DeclareMathOperator{\DM}{DM}
\DeclareMathOperator{\thom}{th}

\newcount\cols
{\catcode`,=\active\catcode`|=\active
 \gdef\Young(#1){\hbox{$\vcenter
 {\mathcode`,="8000\mathcode`|="8000
  \def,{\global\advance\cols by 1 &}%
  \def|{\cr
        \multispan{\the\cols}\hrulefill\cr
        &\global\cols=2 }%
  \offinterlineskip\everycr{}\tabskip=0pt
  \dimen0=\ht\strutbox \advance\dimen0 by \dp\strutbox
  \halign
   {\vrule height \ht\strutbox depth \dp\strutbox##
    &&\hbox to \dimen0{\hss$##$\hss}\vrule\cr
    \noalign{\hrule}&\global\cols=2 #1\crcr
    \multispan{\the\cols}\hrulefill\cr%
   }
 }$}}
}

\usepackage[T2A,T1]{fontenc}
\makeatletter
\def\easycyrsymbol#1{\mathord{\mathchoice
  {\mbox{\fontsize\tf@size\z@\usefont{T2A}{\rmdefault}{m}{n}#1}}
  {\mbox{\fontsize\tf@size\z@\usefont{T2A}{\rmdefault}{m}{n}#1}}
  {\mbox{\fontsize\sf@size\z@\usefont{T2A}{\rmdefault}{m}{n}#1}}
  {\mbox{\fontsize\ssf@size\z@\usefont{T2A}{\rmdefault}{m}{n}#1}}
}}
\makeatother
\newcommand{\yu}{\easycyrsymbol{\cyryu}}

\title[Algebraic elliptic cohomology and flops \uppercase\expandafter{\romannumeral2}]
{Algebraic elliptic cohomology and flops \uppercase\expandafter{\romannumeral2}: $\SL$-cobordism}

\date{ \today}

\author[M.~Levine]{Marc~Levine}
\address{Universit\"at Duisburg-Essen,
Fakult\"at Mathematik, Campus Essen, 45117 Essen, Germany}
\email{marc.levine@uni-due.de}

\author[Y.~Yang]{Yaping~Yang}
\address{School of Mathematics and Statistics, The University of Melbourne, 813 Swanston Street, Parkville VIC 3010, Australia}
\email{yaping.yang1@unimelb.edu.au}

\author[G.~Zhao]{Gufang~Zhao}
\address{School of Mathematics and Statistics, The University of Melbourne, 813 Swanston Street, Parkville VIC 3010, Australia}
\email{gufangz@unimelb.edu.au}


\begin{document}
\begin{abstract}
In this paper, we study the algebraic Thom spectrum $\MSL$ in Voevodsky's motivic stable homotopy category  over an arbitrary perfect field $k$. Using the motivic Adams spectral sequence, we compute the 
geometric part of the $\eta$-completion of $\MSL$. As an application, we  study  Krichever's elliptic genus with integral coefficients, restricted to $\MSL$.   We determine its image, and identify its kernel as  the ideal generated by differences of $\SL$-flops. This was proved by B. Totaro in the complex analytic setting.
In the appendix, we prove some convergence properties of the motivic Adams spectral sequence.
\end{abstract}
\thanks{
The  authors Y.Y. and G.Z are grateful to Universit\"at Duisburg-Essen for hospitality  and excellent working conditions, and for financial support via DFG Schwerpunktprogramm
SPP 1786. We would like to thank Diarmuid Crowley for pointing out the reference [S68], and for very useful discussions. Y.Y. was partially supported by 
the Australian Research Council (ARC) via the award DE190101231. G.Z. was partially supported by ARC via the award
DE190101222. M.L. was partially supported by the DFG through the grant  LE 2259/7-2 and by the ERC through the project QUADAG. All three authors thank the referee for the detailed and very helpful comments, which have greatly improved this paper. This paper is part of a project that has received funding from the European Research Council (ERC) under the European Union's Horizon 2020 research and innovation programme (grant agreement No. 832833).}

\maketitle

\section{Introduction}

\subsection{Motivation}
In this paper, we study the algebraic Krichever elliptic genus, when restricted to the $\SL$-cobordism ring with integral coefficients. 

Elliptic genera in topology have a renowned rigidity property, conjectured by Witten and proved by many others, including Bott-Taubes \cite{BT}, Liu \cite{Liu}, and Ando \cite{An}, which says that for any Spin-manifold with $S^1$-action, the $S^1$-equivariant elliptic genus does not depend on the equivariant parameter. 
Similarly, there is a $\SU$-rigidity theorem for a certain elliptic genus with two parameters (referred to as  Krichever's elliptic genus) proved by Krichever \cite{Kr} and H\"ohn \cite{H}, which says the $S^1$-equivariant Krichever elliptic genus of any $\SU$-manifold is a constant. In \cite{Totaro} Totaro proved that a genus has the $\SU$-rigidity property if and only if its values on two birational manifolds related by a flop are equal. Further more, he proved that  Krichever's elliptic genus is  universal with respect to this property. The proof in \cite{Totaro} uses topological constructions, which do not have direct counterparts for varieties over an arbitrary field.

In \cite{LYZ}, the authors gave a purely algebraic proof of the fact that the kernel of Krichever's elliptic genus coincides with the ideal generated by differences of flops, a proof which works for varieties over an arbitrary perfect field. Moreover, we obtained the existence of a corresponding motivic oriented cohomology theory representing elliptic cohomology. 
With rational coefficients,  we have a description of the coefficient ring of the motivic elliptic cohomology. 
For a summary of the main results in \cite{LYZ}, see \S~\ref{subsec:summary_LYZ1}.

The rigidity property in topology suggests that after the restriction to $\SL$-cobordism,  this algebraic Krichever elliptic genus has better properties; this is the subject of the present paper. For example, it is natural to expect,  as in topology, that the image of $\SL$-cobordism is noetherian, even though  the image under the elliptic genus of the $\GL$-cobordism  is  not. 

The coefficient ring of algebraic $\SL$-cobordism has not been fully investigated, unlike its topological analogue, which goes back to a classical result of Novikov in the study of Adams spectral sequences \cite{Nov}. In the present paper, following Novikov's approach, we explore the motivic Adams spectral sequence to get the information about the coefficient ring of $\SL$-cobordism necessary for the study of elliptic genus. 
Along the way,  we study convergence property of motivic Adams spectral sequence in Appendix~\ref{app}, with the results summarized in \S~\ref{subsec:intro_MASS}. 

\subsection{Main theorems}
Let $k$ be an arbitrary perfect field. Let $p$ be the exponential characteristic of $k$. That is, $p=\text{char }k$ if $\text{char }k>0$, $p=1$ if $\text{char }k=0$.
Let $\SH(k)$ be the motivic stable homotopy category of $\bbP^1$-spectra (for the conventions we are following, see \S~\ref{sec:prelim}); we often refer to an object of $\SH(k)$ as a {\em motivic spectrum}. For $\sE\in \SH(k)$, let $\sE^\wedge_{\eta}$ denote the completion of $\sE$ with respect to the stable algebraic Hopf map $\eta: \G_m=S^{1, 1}\to S^0=S^{0, 0}$ and write $\sE^n$ for the ``algebraic part''  $\sE^{2n,n}(k)$ of the coefficient ring $\sE^{*,*}(k)$. Let $\MGL,\MSL\in\SH(k)$ be the Thom spectra of $(\GL_n)_n$ and $(\SL_n)_n$ respectively.

 Let ${\MSL}^{div*}\subset {\MSL}^{\wedge}_{\eta}[1/2p]^{*}$ be the maximal subgroup that is $l$-divisible for all primes $l\neq 2, p$ and let $\overline{\MSL}^{\wedge}_{\eta}[1/2p]^{*}$ be the quotient ${\MSL}^{\wedge}_{\eta}[1/2p]^{*}/{\MSL}^{div*}$.  There is a natural map $\MSL^\wedge_\eta[1/2p]^*\to \MGL[1/2p]^*$ (Lemma~\ref{lem:MGL eta}) which   induces an embedding $\overline{\MSL}^{\wedge}_{\eta}[1/2p]^{*}\hookrightarrow  \MGL[1/2p]^*$ (see Theorem~\ref{thm:MSL} below). 

There is a formal group law with coefficients in $\MGL^*$ induced by the  first Chern class of line bundles in $\MGL$-theory, giving the classifying homomorphism $\phi_{\MGL}: \Laz \to \MGL^*$, where $\Laz$ is the Lazard ring. It follows from Hopkins-Morel isomorphism, proved by Hoyois \cite[Theorem 5.11]{Hoy}, together with work of Spitzweck \cite[Corollary 4.9]{Spitz} that $\phi_{\MGL}$ is an isomorphism after inverting $p$.

Let $\tilde\Ell$ be the ring $\bbZ[a_1,a_2,a_3, (1/2)a_4]$. We have the elliptic curve $\sE_{\tilde\Ell}$ over $\Spec \tilde\Ell$  defined as the base-change of the Weierstrass curve $y^2 + \mu_1xy + \mu_3y = x^3 + \mu_2x^2 + \mu_4x + \mu_6$ along 
\[
\mu_1\mapsto2a_1,\ \mu_2\mapsto 3a_2-a_1^2,\
\mu_3\mapsto-a_3,\
\mu_4\mapsto -\frac{1}{2}a_4+3a_2^2-a_1a_3,\
\mu_6\mapsto0.
\]
 We let $\Delta\in \tilde\Ell$ denote the discriminant of this family and set $\Ell:=\tilde\Ell[\Delta^{-1}]$.

The local uniformizer $t=y/x$ of the elliptic curve gives rise to a formal group law over $\tilde\Ell$, with the corresponding classifying homomorphism $\Laz\to \tilde\Ell$. Via Hoyois' theorem, this gives the ring homomorphism  $\phi[1/p]: \MGL^*[1/p]\to \tilde\Ell[1/p]$, which is the algebraic Krichever elliptic genus \cite[\S~3.1]{LYZ}.
The restriction of $\phi[1/2p]:\MGL[1/2p]^*\to\tilde\Ell[1/2p]$ to $\overline{\MSL}^\wedge_\eta[1/2p]^*$ is denoted by $\overline{\phi}$. 

Let  $\calI_{fl}\subseteq \MGL[1/p]^*$ be the ideal generated by differences of flops.
Define the ideal of $\SL$-flops $\calI^{\SL}_{fl}\subseteq\overline{\MSL}^\wedge_\eta[1/2p]^*$ to be $\overline{\MSL}^\wedge_\eta[1/2p]^*\bigcap \calI_{fl}[1/2p]$. 
In this paper, we prove
\begin{thm}\label{thm:elliptic genus}
The kernel of $\overline{\phi}$ in $\overline{\MSL}^{\wedge}_{\eta}[1/2p]^*$ is $\calI^{\SL}_{fl}$ and the image of $\overline{\phi}$ is the polynomial ring
$\bbZ[1/2p][3a_2, a_3, a_4]$, with $\deg(a_i)=-i$. 
\end{thm}

Although this statement is similar to its topological analogue \cite[Theorem~6.1]{Totaro}, the proof here is more involved. Not knowing the homotopy groups of motivic Thom spectrum $\MSL$, we need to study the motivic Adams spectral sequence, which differs from its classical analogue in many respects. One of the major differences is the non-nilpotence of the algebraic Hopf map $\eta$.

We calculate of the $E_2$ page of the mod-$l$ motivic Adams spectral sequence for $\MSL$, which, thanks to the general convergence properties of the motivic Adams spectral sequence discussed in Appendix~\ref{app} (summarized in Appendix~\ref{app},  \S~\ref{subsec:intro_MASS}),  converges to the  coefficient ring $\MSL_{\eta,l}^{\wedge *,*}(k)$ of the $\eta, l$-completion of $\MSL$. We  show that certain differentials vanish (Proposition~\ref{diff_MGLMSL}), proving the following result, which is a key ingredient in the proof of Theorem~\ref{thm:elliptic genus}.

\begin{thm}
\label{thm:MSL} Let ${\MSL}^{div*}\subset {\MSL}^{\wedge}_{\eta}[1/2p]^{*}$ be the maximal subgroup that is $l$-divisible for all primes $l\neq 2, p$ and let $\overline{\MSL}^{\wedge}_{\eta}[1/2p]^{*}$ be the quotient ${\MSL}^{\wedge}_{\eta}[1/2p]^{*}/{\MSL}^{div*}$. 
\ 
\begin{enumerate}
\item The ring $\overline{\MSL}^{\wedge}_{\eta}[1/2p]^{*}$ is a  polynomial ring over $\Z[1/2p]$ in generators $v_2, v_3,\ldots$, with  $\deg v_i=-i$
\item The canonical map $\MSL\to \MGL$ induces an injection $\overline{\MSL}^{\wedge}_{\eta}[1/2p]^{*}\to {\MGL}[1/2p]^{*}$.
\item Let $X$ be a smooth projective $k$-scheme of dimension $n\ge2$ and let $[X]_{\MGL}\in \MGL^{-n} $ be the corresponding class (see \S~\ref{sec:CY}). Let  $s_n(X):=\deg_k(c_{(n)}(T_X))$ be the Chern number of degree $n$ associated to the symmetric polynomial $\xi_1^n+\cdots +\xi_n^n$. Then  $[X]_{\MGL}$ is a polynomial generator of $\overline{\MSL}^{\wedge}_{\eta}[1/2p]^{ *}$  if and only if the following  holds:
\[
s_n(X)=\left\{
  \begin{array}{ll}
    \pm l \cdot 2^ap^b, & \hbox{ if $n$ is a power of an odd prime $l$;} \\
    \pm l \cdot 2^ap^b, & \hbox{ if $n + 1$ is a power of an odd prime $l$;} \\
    \pm 2^ap^b, & \hbox{ otherwise,}
  \end{array}
\right.
\]
for  some non-negative integers $a, b$.
\end{enumerate}
\end{thm}

{
\begin{remark}
With 2 inverted, the motivic stable homotopy category $\SH(k)$ splits as
\[
\SH(k)[\frac{1}{2}]=\SH(k)[\frac{1}{2}]^{+} \oplus \SH(k)[\frac{1}{2}]^{-}. 
\]
The projection of $\eta$ to $\SH(k)[\frac{1}{2}]^{+}$ is 0 while the projection of
$\eta$ to $\SH(k)[\frac{1}{2}]^{-}$ is invertible and $\MSL[\frac{1}{2}]^{\wedge}_{\eta}= \MSL[\frac{1}{2}]^+$ is
precisely the plus part of the spectrum $\MSL$. 
In \cite{B18},  Tom Bachmann establishes a precise relation between the minus part of $\SH(k)$ and the classical
stable homotopy category $\SH$, in particular, one has $\SH(\mathbb{R})[\frac{1}{2}]^{-}\cong \SH[\frac{1}{2}]$. Hence the computation from Theorem \ref{thm:MSL} gives  the
previously unknown half of the algebraic diagonal $\MSL[\frac{1}{2p}]^{2*,*}$  of the coefficient ring for $\MSL[\frac{1}{2p}]$, modulo the prime to $2p$-divisible part ${\MSL}^{div*}$. 
\end{remark}
}

\begin{remark} We conjecture that the prime to $2p$-divisible part ${\MSL}^{div*}$ is zero. Presumably, an analysis of the mod $l$-theory for $l$ a prime $\neq 2,p$ would show that 
${\MSL}^{div*}$ is uniquely divisible, but we have not pursued this question here.
\end{remark}

\subsection{Convergence of motivic Adams spectral sequence}\label{subsec:intro_MASS}
In topology, a number of convergence properties of the Adams spectral sequences were  proven in \cite{Adams}. Bousfield \cite{Bous} developed a simpler approach,    based on notions  of localization going back to Ravenel \cite{Rav}. 

In Appendix~\ref{app}, we study the convergence properties of the motivic Adams spectral sequence \cite{Mor},  following the method of \cite{Bous}. The role of Postnikov tower in {\it loc. cit.} is replaced by the slice tower of Voevodsky \cite{Voev_open}, together with results on this tower developed in \cite{RSO}.

Without going to the technical details, we state  our analogue of Bousfield's theorem. For a spectrum $\sE$, we let  $\sE^\wedge_{\eta,l}$ denote the $\eta$-$l$-completion of $\sE$ (see \S~\ref{sec:MASS}).

\begin{thm}[Theorem~\ref{thm:app}]\label{thm:intro_MASS}
Let $l$ be a prime different than $\Char(k)$. Let $\sE\in\SH(k)$ be a slice connective motivic spectrum, that is,  
$f_N\sE=\sE$ for $N \ll 0$. Let $\sE_{H\bbZ/l}^{\wedge}$ be the homotopy limit of the Adams tower \eqref{eqn:AdamsTower}. If $\sE$ has a cell presentation of finite type (\cite[\S~3.3]{RSO}) and satisfies condition \eqref{eqn:fin}, then   $\sE_{H\bbZ/l}^{\wedge}\cong \sE^\wedge_{\eta,l}$.
\end{thm}

\begin{remark} Recently L. Mantovani \cite{Mantovani} has transferred Bousfield's methods analyzing $E$-completion and $E$-localization to the motivic setting. As a particular case, his results show that, for $\sE\in \SH(k)$ a connective spectrum (see \S~\ref{subsec:Background}) one has $\sE_{H\bbZ/l}^{\wedge}\cong \sE^\wedge_{\eta,l}$. As $\MGL$ and $\MSL$ are connective, Mantovani's results also yield the convergence properties that we need; we include our discussion of Theorem~\ref{thm:intro_MASS} here as our proof also gives  convergence results for non-connective spectra.
\end{remark}

\subsection*{Organization of the paper}
In \S~\ref{sec:prelim}, we recall  some background on the motivic stable homotopy category and some basic facts about the motivic spectra $\MGL$ and $\MSL$. We recall the construction of the class in $\MGL^*$ associated to a smooth projective $k$-scheme in \S\ref{sec:CY} and show how to lift the class of a Calabi-Yau variety to a class in $\MSL^*$. In \S~\ref{sec:Steenrod}, we introduce the motivic mod-$l$ Steenrod algebra $A^{*,*}$ and recall some of its basic properties. We study the motivic cohomologies  $H^{*,*}\MGL$ and  $H^{*,*}\MSL$ as modules over $A^{*,*}$
in \S~\ref{sec:SteenMod}. This information will be used in the calculation of the motivic Adams spectral sequence in \S~\ref{sec:MASS}, which enables us  to prove Theorem~\ref{thm:MSL}. In \S~\ref{sec:EllGenus}, we use Theorem~\ref{thm:MSL} to  prove Theorem~\ref{thm:elliptic genus}.
The convergence properties of the motivic Adams spectral sequence that we need will be discussed in Appendix~\ref{app} and in Appendix~\ref{App:Novikov} we supply a proof of Novikov's lemma \cite[Lemma 16]{Nov}, which is stated without proof in {\it loc. cit.}.

\section{Preliminaries}
\label{sec:prelim}

\subsection{Background on motivic  homotopy theory}\label{subsec:Background} We fix a perfect field $k$ and let $p$ denote the exponential characteristic of $k$.
Let $\Sm_k$ be the category of separated and finite type schemes ({\em varieties}), smooth over $k$. We have the category of {\em  spaces over $k$}, $\Spc(k)$, and  {\em pointed  spaces over $k$}, $\Spc_\bullet(k)$, which are  the categories of presheaves of simplicial sets (resp.  pointed simplicial sets) on $\Sm_k$; sending $X\in \Sm_k$ to the constant  simplicial set $X$ (resp. constant  pointed simplicial set $X_+$) on the representable presheaf $\Hom_{\Sm_k}(-,X)$ identifies $\Sm_k$ with a subcategory of $\Spc(k)$ (resp.  $\Spc_\bullet(k)$). The  {\em motivic unstable homotopy category} $\sH_\bullet(k)$  is a certain localization of $\Spc(k)$ and the pointed motivic homotopy category is constructed similarly (see  \cite[\S3]{MorelVoev}); in particular, each $X\in \Sm_k$ defines the objects $X$ of $\sH(k)$ and  $X_+$ of $\sH_\bullet(k)$.  

We  have suspension functors
\[
\Sigma_{S^1},\ \Sigma_{\G_m},\ \Sigma_{\P^1}:\sH_\bullet(k)\to \sH_\bullet(k)
\]
and for $a\ge b\ge0$ the suspension $\Sigma^{a,b}:=\Sigma_{S^1}^{a-b}\Sigma_{\G_m}^b$. For $\sX\in \Spc_\bullet(k)$ and $a\ge b\ge0$,  the {\em unstable $\A^1$-homotopy sheaf} $\pi_{a,b}^{\A^1}(\sX)$ is the Nisnevich sheaf associated to the presheaf on $\Sm_k$
\[
U\mapsto \Hom_{\sH_\bullet(k)}(\Sigma^{a, b}U_+, \sX).
\]
We set $\pi_n^{\A^1}(\sX):=\pi_{n,0}^{\A^1}(\sX)$.

Let $\SH(k)$ be the motivic stable homotopy category of $\PP^1$-spectra: a $\PP^1$-spectrum consists of a sequence $(E_0, E_1,\ldots)$, $E_n\in  \Spc_\bullet(k)$, together with bonding maps $\epsilon_n:\Sigma_{\P^1}E_n\to E_{n+1}$ in $\Spc_\bullet(k)$, and $\SH(k)$ is a certain localization of the category of $\PP^1$-spectra  (see e.g. \cite{Jardine} for details).  Recall that one has the {\em infinite $\PP^1$-suspension functor} $\Sigma^\infty_{\PP^1}(-):\sH_\bullet(k)\to \SH(k)$,
\[
\Sigma^\infty_{\PP^1}(\sX):=(\sX, \Sigma_{\P^1}\sX,\ldots, \Sigma^n_{\P^1}(\sX),\ldots),\ \epsilon_n=\id: \Sigma_{\P^1}(\Sigma^n_{\P^1}(\sX))\to  \Sigma^{n+1}_{\P^1}(\sX),
\]
and  suspension functors 
\[
\Sigma^n_{S^1}, \Sigma_{\G_m}^n, \Sigma_{\PP^1}^n:\SH(k)\to \SH(k)
\]
for all $n\in \ZZ$. $\SH(k)$ is a triangulated tensor category with translation $\Sigma_{S^1}$,  tensor product $(\sE, \sF)\mapsto \sE\wedge \sF$ and unit the motivic sphere spectrum $\bS_k:=\Sigma^\infty_{\PP^1}\Spec k_+$.  We write $\Sigma^{n,m}$ for $\Sigma^{n-m}_{S^1}\circ \Sigma_{\G_m}^m$ and note that $\Sigma_{\PP^1}=\Sigma^{2,1}$. An object $\sE$ of $\SH(k)$ and integers $n,m$ define the $\sE$-cohomology functor $\sE^{n,m}$ from $\Sm_k^\op$ to the category of abelian groups by
\[
\sE^{n,m}(X):=\Hom_{\SH(k)}(\Sigma^\infty_{\PP^1} X_+, \Sigma^{n,m} \sE).
\]
We denote $\bigoplus_{n,m}\sE^{n,m}(X)$ by $\sE{\bideg}(X)$,  $\bigoplus_n\sE^{2n,n}(X)$ by $\sE^*(X)$. We also have the $\sE$-homology $\sE_{a,b}(X):=\Hom_{\SH(k)}(\Sigma^{a,b}\bS_k,   \sE\wedge\Sigma^\infty_{\PP^1}X_+)$, and its $(2n,n)$-part $\sE_n(X):=\sE_{2n,n}(X)$, $\sE_*(X):=\oplus_n\sE_n(X)$. We often write $\sE^{a,b}$ for $\sE^{a,b}(\Spec k)$,  $\sE^n$ for $\sE^n(\Spec k)$, 
$\sE_{a,b}$ for $\sE_{a,b}(\Spec k)$, etc.  

{We have the presheaf on $\Sm_k$,
\[
U\mapsto \sE^{-a, -b}(U)
\]
giving as associated Nisnevich sheaf on $\Sm_k$ the bi-graded homotopy sheaf $\pi_{a,b}(\sE)$. Reflecting the homotopy $t$-structure on $\SH(k)$ \cite{MorelICTP}, we often reindex this sheaf by $\pi_n(\sE)_m:=\pi_{n-m,-m}(\sE)$. $\sE$ is said to be {\em $a-1$-connective}  if $\pi_n(\sE)_m=0$ for $n<a$ and all $m\in \Z$. We note that  $\pi_{a,b}(\sE)(k)=\sE_{a,b}=\sE^{-a,-b}$.}

 \subsection{$\MGL$ and $\MSL$}\label{subsec:MGLMSL}

We collect some facts about our main players; the main references are \cite{PPR2, PW}. For $n\ge0$, we have the object $\BGL_n\in \sH(k)$. There are various models for $\BGL_n$ in $\Spc(k)$, we will most often use the construction as an infinite Grassmannian
\[
\BGL_n=\colim_m\Gr(n, n+m)
\]
with respect to the usual embeddings $\Gr(n, n+m)\to \Gr(n, n+m+1)$. We have the tautological $n$-plane bundle $\sV_{n,n+m}\to \Gr(n, n+m)$, taking the colimit in $\Spc(k)$ gives $\sV_n\to \BGL_n$ and we have the maps $i_n:\BGL_n\to \BGL_{n+1}$ with $i_n^*\sV_{n+1}\cong \sV_n\oplus \sO$.  The Thom spectrum $\MGL$ is built from the sequence $(\MGL_0, \MGL_1,\ldots)$, with $\MGL_n$ the Thom space $\Th(\sV_n):=\P(\sV_n\oplus \sO)/\P(\sV_n)$. The isomorphism   
$i_n^*\sV_{n+1}\cong \sV_n\oplus \sO$ gives the isomorphism $\Th(i_n^*\sV_{n+1})\cong \Sigma_{\P^1}\Th(\sV_n)$, which together with the evident map $\Th(i_n^*\sV_{n+1})\to \Th(\sV_{n+1})$ gives the bonding maps $\Sigma_{\P^1}\MGL_n\to \MGL_{n+1}$ defining the $\P^1$-spectrum $\MGL$.

$\MSL$ is defined similarly: Let $\sO_n(-1)\to \BGL_n$ be the determinant bundle $\det\sV_n$. Then $p_n:\BSL_n\to \BGL_n$ is defined as the associated $\G_m$-bundle $\sO_n(-1)\setminus 0_{\sO_n(-1)}$. We have the bundle $\tilde\sV_n=p_n^*\sV_n\to \BSL_n$ and $\MSL_n:=\Th(\tilde\sV_n)$. The $\P^1$-spectrum $\MSL$ is $(\MSL_0, \MSL_1,\ldots)$ with bonding maps $\Sigma_{\P^1}\MSL_n\to \MSL_{n+1}$ defined as for $\MGL$.  It follows directly from these descriptions that both $\MGL$ and $\MSL$ can also be described as
\[
\MGL=\colim_n\Sigma^{-n}_{\P^1}\Sigma^\infty_{\P^1}\MGL_n;\quad 
\MSL=\colim_n\Sigma^{-n}_{\P^1}\Sigma^\infty_{\P^1}\MSL_n
\]
and the projections $\BSL_n\to \BGL_n$ give rise to maps $\MSL_n\to \MGL_n$ and a canonical map of $\P^1$-spectra $\MSL\to \MGL$.

The object $\MGL$ of $\SH(k)$ has many useful properties. It has the structure of a commutative monoid in $\SH(k)$, that is $\MGL$ is a {\em motivic commutative ring spectrum} (see   \cite[\S 2.1]{PPR2}). Recall that an {\em orientation} on a motivic commutative ring spectrum $\sE\in \SH(k)$ is an element $\theta_\sE\in \sE^{2,1}(\P^\infty)$ such that the restriction of $\theta_\sE$ to $i_1^*\theta_\sE\in \sE^{2,1}(\P^1)=\sE^{0,0}(k)$ is the unit for the ring spectrum $\sE$.  As $\P^\infty=\BGL_1$, the 0-section of $\sV_1$ gives the map $\P^\infty\to  \MGL_1$, which after composition with $\MGL_1\to \Sigma_{\P^1}\MGL$ gives the map $\Sigma^\infty_{\P^1}\P^\infty\to \Sigma_{\P^1}\MGL$, that is, an element $\theta_{\MGL}\in \MGL^{2,1}(\P^\infty)$. It is shown in \cite[Theorem 2.7]{PPR2} that the motivic commutative ring spectrum $\MGL$ with the element $\theta_{\MGL}\in \MGL^{2,1}(\P^\infty)$ gives the universal oriented motivic ring spectrum in $\SH(k)$. 

As $\P^\infty$ represents the functor $\Pic$ in $\sH(k)$ \cite[Proposition 3.8]{MorelVoev}, an orientation $\theta_\sE\in \sE^{2,1}(\P^\infty)$ gives rise to a theory of first Chern classes of line bundles by
\[
L\to X\leadsto \phi_L:X\to \P^\infty\leadsto c_1^\sE(L):=\phi_L^*(\theta_\sE)\in \sE^{2,1}(X).
\]
This in turn gives the associated formal group law $F_{\sE}(u,v)\in \sE^*(k)[[u,v]]$ with 
$F_{\sE, \theta}(c_1^\sE(L), c_1^\sE(M))=c_1^\sE(L\otimes M)\in \sE^{2,1}(X)$ for $L, M$ line bundles on some $X\in \Sm_k$.  

We recall the {\em Lazard ring} $\Laz$, which is the coefficient ring of the universal formal group law $F_{univ}(u,v)\in \Laz[[u,v]]$. Each formal group law $F\in R[[u,v]]$ gives rise to a unique ring homomorphism $\phi_{F,R}:\Laz\to R$ with $F=\phi_{F,R}(F_{univ})$ and thus we have the canonical homomorphism of graded rings
\[
\phi_\sE:\Laz\to\sE^*(k)
\]
classifying $F_{\sE}$. For $\sE=\MGL$ we have the theorem of Hopkins-Morel-Hoyois and Spitzweck

\begin{theorem}[\hbox{\cite{Hoy, Spitz}}]\label{thm:HMH} After inverting the exponential characteristic $p$ of $k$, the ring homomorphism $\phi_{\MGL}:\Laz\to \MGL^*(k)$ is an isomorphism. Moreover $\MGL^{2n+a, n}[1/p]=0$ for all $a>0$ 
\end{theorem} 

\begin{proof} The main results of \cite{Hoy, Spitz} give a strongly convergent spectral sequence
\[
E_2^{r,s}(n)=H^{r-s, n-s}(k, \Z)\otimes \Laz^{s}\Rightarrow \MGL^{r+s, n},
\]
after inverting $p$. The properties of motivic cohomology given by   Theorem~\ref{Thm19.3} below yield the partial computation of the $E_2^{r,s}(n)$-terms as
\[
E_2^{r,s}(n)=\begin{cases}0&\text{ for }r+s>2n\\
0&\text{ for }r+s=2n, r\neq s\\
0&\text{ for }r+s=2n-1, (r,s)\neq (n, n-1)\\
\Laz^n&\text{ for }r=s=n
\end{cases}
\]
Feeding this into the spectral sequence gives $\MGL^{2n+a,n}=0$ for $a>0$ and $\MGL^{2n,n}=\Laz^n$ for all $n$ (note that $\Laz=\oplus_{n\le 0}\Laz^n$). 
\end{proof}

We recall the fundamental connectivity theory of Morel in the $\P^1$-stable setting. A pointed space $\sX\in \Spc_\bullet(k)$ is called $a-1$-connected if each Nisnevich stalk $\sX_x\in \Spc_\bullet$, $x\in X\in \Sm_k$ is an $a-1$-connected space. 

\begin{theorem}[\hbox{\cite[Theorem 3]{MorelStable}}] Let $\sX\in \Spc_\bullet(k)$ be $a-1$-connected. Then the suspension spectrum $\Sigma^\infty_{\P^1}\sX\in \SH(k)$ is $a-1$-connected, that is $\pi_n(\Sigma^\infty_{\P^1}\sX)_*=0$ for $n<a$. 
\end{theorem} 
In particular, taking $\sX=X_+$ for some $X\in \Sm_k$, since $X_+\in \Spc_\bullet(k)$ is -1-connected,  we see that  $\Sigma^\infty_{\P^1}X_+$ is a -1-connected $\P^1$-spectrum.

\begin{corollary} Both $\MGL$ and $\MSL$ are -1-connected $\P^1$-spectra.
\end{corollary}

\begin{proof} Write $\Gr(n,n+m)$ as a finite union of open subschemes $U_i$ such that $\sE_{n, m|U_i}\to U_i$ is a trivial bundle $\sO_{U_i}^n$. For $U\subset \Gr(n,n+m)$ that trivializes $\sE_{n,m}$, we have  $\Th(\sE_{n,m|U})\cong \Th(\sO_U^n)=\Sigma_{\P^1}^nU_{+}\cong S^n\wedge\G_m^{\wedge n}\wedge U_{+}$, we see that the space $\Th(\sO_{U}^n)$ is $n-1$-connected, and thus $\Sigma^\infty_{\P^1}\Th(\sO_{U}^n)\in \SH(k)$ is also $n-1$-connected. The decomposition of $\Th(\sE_{n,m})$ as  a union of the $\Th(\sO_{U_i}^n)$ and a Mayer-Vietoris argument shows by  induction that 
 $\Sigma^\infty_{\P^1}\Th(\sE_{n,m|\cup_{i=1}^rU_i})$  is  $n-1$-connected for every $r$, and thus $\Sigma^\infty_{\P^1}\Th(\sE_{n,m})$ is itself $n-1$-connected. As the collection of $n-1$-connected spectra is closed under filtered colimits, we see that $\Sigma^\infty_{\P^1}\Th(\sE_n)$ is $n-1$-connected and thus $\Sigma_{\P^1}^{-n}\Sigma^\infty_{\P^1}\Th(\sE_n)=\Sigma^{-2n,-n}\wedge  \Sigma^\infty_{\P^1}\Th(\sE_n)$ is $-1$ connected. Again taking a filtered colimit, we find that $\MGL=\colim_n\Sigma_{\P^1}^{-n}\Sigma^\infty_{\P^1}\Th(\sE_n)$ is -1-connected.
 
The proof for $\MSL$ is exactly the same.
\end{proof}

We conclude this section with a computation of the motivic cohomology and the motives of $\MGL$ and $\MSL$. We recall the motivic cohomology spectrum $H\Z\in \SH(k)$. This is a $\P^1$-spectrum with a model as a commutative monoid in motivic  symmetric spectra (see \cite[Example 3.4]{DRO}), which allows one to define the model category of $H\Z$-modules $H\Z\lMod$. We denote the homotopy category of $H\Z\lMod$ by $\DM(k)$. The reader may be annoyed by this duplication of the standard notation for a version of Voevodsky's triangulated category of motives $\DM_{\text{V}}(k)$ \cite{VSF}. In fact, for $k$ of characteristic zero, R\"ondigs-{\O}stv{\ae}r have shown that $\DM(k)$ and
$\DM_{\text{V}}(k)$ are equivalent, and for $k$ of positive characteristic $p$, the analogous result has been  proven by Hoyois-Kelly-{\O}stv{\ae}r \cite[Theorem 5.8]{HSO}, after inverting $p$, so for our purposes, there should be no confusion. Forgetting the $H\Z$-module structure defines the  motivic {\em Eilenberg-MacLane  functor} $\EM:\DM(k)\to \SH(k)$, which has left adjoint the free $H\Z$-module functor $M:\SH(k)\to \DM(k)$, $M(\sE):=H\Z\wedge \sE$; for $\sE\in \SH(k)$, its {\em motive} is $M(\sE)\in \DM(k)$. We similarly have the motivic commutative ring spectrum $H\Z/l$. The {\em pure Tate motive of weight $n$} is $\Sigma^n_{\P^1}H\Z$ and the pure mod $l$ Tate motive of weight $n$ is 
$\Sigma^n_{\P^1}H\Z/l$. 

We mention the following partial computation of motivic cohomology
\begin{theorem}\cite[Corollary 4.2, Theorems 19.1, 19.3]{MVW}\label{Thm19.3}
For  $X$  in $\Sm_k$ and $A$ an abelian group, we have $H^{p, q}(X, A)=0$, if $p>q+\dim(X)$, if $p>2q$ or if $q<0$. Moreover $H^{p,0}(\Spec k,\Z)=\Z$ for $p=0$, $H^{p,0}(\Spec k,\Z)=0$ for $p\neq0$ and  $H^{p,1}(\Spec k, \Z)=0$ for $p\neq1$.
\end{theorem}

 As detailed in  \cite[\S 1.2]{PPR2}, for $(\sE, \theta_\sE)$ a motivic commutative ring spectrum with orientation $\theta_\sE$, one has  {\em Thom classes}  $\thom(V)\in \sE^{2r, r}(\Th(V))$ for each rank $r$ vector bundle $p:V\to X$, $X\in \Sm_k$, and cup product with $\thom(V)$ gives the {\em Thom isomorphism} $\vartheta_V:\sE^{a,b}(X)\to \sE^{2r+a, r+b}(\Th(V))$.  This is reflected in the isomorphism in $\SH(k)$, $\sE\wedge\Th(V)\cong \Sigma^{2r, r}\sE\wedge  X_+$, defined as the composition
\begin{multline*}
\sE\wedge \Th(V)\xrightarrow{\id\wedge\delta_V}\sE\wedge \Th(V)\wedge V_+
\xrightarrow{\id\wedge p}\sE\wedge \Th(V)\wedge X_+
\\\xrightarrow{\id\wedge \thom(V)\wedge\id}\sE\wedge S^{2r,r}\wedge \sE\wedge X_+\xrightarrow{\mu_\sE\wedge\id} S^{2r, r}\wedge \sE\wedge X_+,
\end{multline*}
which in turn gives the  $\sE$-homology isomorphism
\[
\sE_{2r+*, r+*}(\Th(V))\cong \sE_{**}(X).
\]
We recall from \cite[\S 4]{Voev03} that motivic cohomology is an oriented theory.

A smooth dimension $d$ $k$-scheme $X$ is {\em cellular} if $X$ admits a filtration by closed subsets $F_nX\subset X$ 
\[
X=F^0X\supset F^1X\supset\ldots\supset F^dX\supset F^{d+1}X=\emptyset
\]
with $F^nX\setminus F^{n+1}X$ a disjoint union of $s_n$ copies of $\A^{d-n}$. 

\begin{lemma}\label{lem:FilDecom} Let $\sE\in \SH(k)$ be an oriented  motivic commutative ring spectrum that has a model as a commutative monoid in symmetric $\P^1$-spectra and let $X$ be a cellular $k$-scheme with filtration as above. Suppose $\pi_{a,b}\sE=0$ for $b>0$. Then in $\Ho\sE\lMod$, we have a canonical isomorphism
\[
\sE\wedge \Sigma^\infty_{\PP^1}X_+\cong \oplus_{n=0}^d \Sigma^{2n,n}\sE^{\oplus s_n}.
\]
\end{lemma}
 
\begin{proof} Letting $N_n$ be the (rank $n$) normal bundle of $F^n\setminus F^{n+1}$ in $X\setminus F^{n+1}$,  the Morel-Voevodsky purity isomorphism $(X\setminus F^{n+1})/(X\setminus F^n)\cong \Th(N_n)$ and the Thom isomorphism $\pi_{a,b}(\sE\wedge(\Th(N_n)))\cong \pi_{a-2n,b-n}(\sE\wedge(F^n\setminus F^{n+1}))$ gives the canonical isomorphism
\[
\pi_{a,b}(\sE\wedge (X\setminus F^{n+1})/(X\setminus F^n))\cong \pi_{a-2n, b-n}(\sE)^{\oplus s_n}. 
\]
In particular, for $b>n$, $\sE_{a,b}((X\setminus F^{n+1})/(X\setminus F^n))=0$.
Taking $a=2n, b=n$ gives the $s_n$ classes $\bar\alpha_{n,i}\in \pi_{2n,n}(\sE\wedge (X\setminus F^{n+1})/(X\setminus F^n))(k)$ corresponding to the unit in $\sE^{0,0}(k)$ for the $i$th component of $F_n\setminus F_{n+1}=\amalg_{i=1}^{s_n}\A^{d-n}$.

For each $m$ we have the cofiber sequence
\[
(X\setminus F^{m+1})/(X\setminus F^m)\to X/(X\setminus F^m)\to X/(X\setminus F^{m+1})
\]
which induces the long exact sequence of homotopy sheaves
\begin{multline*}
\ldots\to \pi_{a,b}(\sE\wedge(X\setminus F^{m+1})/(X\setminus F^m))\to 
 \pi_{a,b}(\sE\wedge X/(X\setminus F^m))\\\to  \pi_{a,b}(\sE\wedge X/(X\setminus F^{m+1}))\xrightarrow{\delta} 
\pi_{a-1,b}(\sE\wedge (X\setminus F^{m+1})/(X\setminus F^m))\to \ldots
\end{multline*}
Let $\alpha^*_{n,i}\in \pi_{2n,n}(\sE\wedge X/(X\setminus F^n))(k)$ be the image of $\bar\alpha_{n,i}$ under the map induced by $\sE\wedge (X\setminus F^{n+1})/(X\setminus F^n)\to \sE\wedge X/(X\setminus F^n)$. 

As above, we have $\pi_{a,n}(\sE\wedge (X\setminus F^{m+1})/(X\setminus F^m))=0$ for all $m<n$ and all $a$. The long exact sequence of homotopy sheaves shows that the  maps
\[
\pi_{2n,n}(\sE\wedge X_+)\to \pi_{2n,n}(\sE\wedge X/(X\setminus F^1))\to\ldots\to 
\pi_{2n,n}(\sE\wedge X/(X\setminus F^{n-1}))\to 
\pi_{2n,n}(\sE\wedge X/(X\setminus F^n))
\]
are all isomorphisms. Thus, we can uniquely lift the classes $\alpha^*_{n,i}$ to $\alpha_{n,i}\in  \pi_{2n,n}(\sE\wedge X_+)(k)$. 

Using the multiplication on $\sE$, each element $\alpha_{n,i}\in  \pi_{2n,n}(\sE\wedge X_+)(k)$ gives a map of sheaves
\[
\pi_{*-2n,*-n}(\sE)\to \pi_{*, *}(\sE\wedge X_+)
\]
which is a map of bi-graded sheaves of $\pi_{*,*}(\sE)$-modules. This gives the map of the free bi-graded $\pi_{*,*}(\sE)$-modules $\theta_X:\oplus_{n}\pi_{*-2n,*-n}(\sE)^{s_n}\to \pi_{*, *}(\sE\wedge X_+)$. 

The map $\theta_X$ is clearly an isomorphism if $F^1=\emptyset$, in other words, $\theta_{X\setminus F^1}$ is an isomorphism.  Assuming that $\theta_{X\setminus F^j}$ is an isomorphism, we  use the exact sequences of homotopy sheaves as above  to show that $\theta_{X\setminus F^{j+1}}$ is an isomorphism, and thus $\theta_X$ is an isomorphism by induction. In other words, the family $(\alpha_{n,i})$ gives a basis for the sheaf $\pi_{*,*}(\sE\wedge X_+)$ as a free $\pi_{*,*}(\sE)$-module. 

Viewing each $\alpha_{n,i}$ as a map $\alpha_{n,i}:\Sigma^{2n, n}\bS_k\to \sE\wedge X_+$, we extend to the $\sE$-module map $\alpha_{n,i}: \Sigma^{2n,n}\sE\to \sE\wedge X_+$ using the multiplication in $\sE$. As the map on $\pi_{*,*}$ induced by 
\[
\sum_{n,i}\alpha_{n,i}:\oplus_n \Sigma^{2n, n}\sE^{\oplus s_n}\to \sE\wedge X_+
\]
is the map $\theta_X$,    $\sum_{n,i}\alpha_{n,i}$ is an isomorphism in $\Ho\sE\lMod$.
\end{proof}

\begin{remark} Since the isomorphism class of $\sE\wedge\Th(N_n)$ in $\SH(F^n\setminus F^{n+1})$  depends only on the class  of $N_n$ in $K_0(F^n\setminus F^{n+1}))$ (see \cite[Th\'eor\`em 1.5.18]{Ayoub}, \cite[Prop. 4.1.1]{RiouRR}) and this class is the same as the trivial bundle of rank $n$,   the conclusion of Lemma~\ref{lem:FilDecom} also holds without assuming the $\sE$ is oriented, except that the isomorphism is no longer canonical.
\end{remark}

We consider the graded polynomial ring $\Z[b_1, b_2,\ldots]$ with $b_n$ in degree $-n$; we may also consider $\Z[b_1, b_2,\ldots]$ as bi-graded with $b_n$ in bi-degree $(-2n,-n)$. We let 
$\Z[c_1, c_2,\ldots]$ denote the graded polynomial ring with $c_n$ in degree $n$ (or in bi-degree $(2n,n)$). Viewing $c_n$ as the $n$th elementary symmetric function in variables $\xi_1, \xi_2,\ldots$, we have for $I=(i_1,\ldots, i_n)\in \N^n$, the $I$th Conner-Floyd class $c_I$ corresponding to the symmetric function $\sigma_I$, which in turn is defined by the identity
\[
\prod_{i,j}(1+\xi_i^jT_j)=1+\sum_I\sigma_I(\xi_1, \xi_2,\ldots)T^I.
\]
Giving $\Z[c_1, c_2,\ldots]$ the coproduct $\delta(c_n)=\sum_{i=0}^n c_i\otimes c_{n-i}$ (with $c_0=1$), we have $\delta(c_K)=\sum_{I+J=K}c_I\otimes c_J$, so we may identify $\Z[c_1, c_2,\ldots]$  as the (homogeneous) dual co-algebra to the polynomial algebra $\Z[b_1, b_2,\ldots]$, where a monomial $b^I:=b_1^{i_1} \cdots b_n^{i_n}$ is dual to $c_I$. For an index $I=(i_1,\ldots, i_n)$ we let $|I|$ denote the weighted degree, $|I|:=\sum_jj\cdot i_j$, so $b^I$ is in bi-degree $(-2|I|, -|I|)$. 

The $i$th Chern classes $c_i(\sV_{n,m})\in H^{2i,i}(\Gr(n, n+m))$  are compatible with respect to pull-back by the inclusions $\Gr(n, n+m)\to \Gr(n, n+m+1)$; it is shown in \cite[Theorem 2.2]{PPR2} that $(c_i(\sV_{n,m}))_m\in \lim_m H^{2i,i}(\Gr(n, n+m))$ gives rise to a unique element $c_i(\sV_n)\in H^{2i,i}(\BGL_n)$ and then to a unique element $c_n\in H^{2i,i}(\BGL)$; together these define an isomorphism of $H^{*,*}(k)[c_1, c_2,\ldots]\cong H^{*,*}(\BGL)$. 

For $\sE\in \SH(k)$, and a free graded $\Z$-module $\Z\cdot b$ with generator $b$ in degree $m$, we set $\Z\cdot b\otimes_\Z\sE:=\Sigma^{-m}_{\P^1}\sE$ and extend this notation in the obvious way to define $M_*\otimes_\Z\sE$ for $M_*=\oplus_{n\in\Z}M_n$ a direct sum of free graded $\Z$-modules. For example,  we have the pure Tate motive $M_*\otimes_\Z H\Z$  and the pure mod $l$ Tate motive $M_*\otimes H\Z/l$.

\begin{theorem}\label{thm:MGLMSLCohHom}  Let $\sE\in \SH(k)$ be an oriented motivic commutative ring spectrum such that $\sE^{a,b}(k)=0$ for $b<0$. Then\\[5pt]
1. $\sE^{**}(\MGL)$ and $\sE^{**}(\MSL)$ are the $\sE^{**}(k)$-modules
\[
\sE^{**}(\MGL)=\Z[c_1, c_2,\ldots]\otimes_\Z \sE^{**}(k),\ 
\sE^{**}(\MSL)=\Z[c_1, c_2,\ldots]/(c_1)\otimes_\Z \sE^{**}(k).
\]
2.  $\sE\wedge \MGL\cong \Z[b_1, b_2,\ldots]\otimes \sE$ and  $\sE\wedge \MSL\cong  (c_1)^\perp\otimes \sE$.\\
where $(c_1)^\perp\subset  \Z[b_1, b_2,\ldots]$ is the $\Z$-free summand of $\Z[b_1, b_2,\ldots]$ killed by the ideal $(c_1)\subset \Z[c_1, c_2,\ldots]$.
\end{theorem} 

\begin{proof}  We first show that
\[
\sE^{**}(\BGL)=\Z[c_1, c_2,\ldots]\otimes_\Z \sE^{**}(k),\ 
\sE^{**}(\BSL)=\Z[c_1, c_2,\ldots]/(c_1)\otimes_\Z \sE^{**}(k).
\]
$\sE\wedge \BGL\cong \Z[b_1, b_2,\ldots]\otimes \sE$ and $\sE\wedge \BSL\cong  (c_1)^\perp\otimes \sE$.

Using the standard filtration of $\Gr(n,m+n)$ by the Schubert varieties, it follows from Lemma~\ref{lem:FilDecom} that $\sE\wedge\Gr(n,m+n)\cong \oplus_{i=0}^{nm}\Sigma^{2i,i}\sE^{\oplus s_{i, n,m}}$ for suitable integers $s_{i, n,m}$. By choosing a suitable flag of subspaces of $\A^{m+n}$ and $\A^{n+m+1}$ to define the Schubert filtrations, we may assume that   $\Gr(n,m+n)\cap F^i\Gr(n, m+n+1)=F^i\Gr(n,m+n)$. This shows that the map $\id\wedge \sE\wedge\Gr(n,m+n)\to \sE\wedge\Gr(n,m+n+1)$ induces an inclusion
$\oplus_{i=0}^{nm}\Sigma^{2i,i}\sE^{\oplus s_{i, n,m}}\to \oplus_{i=0}^{n(m+1)}\Sigma^{2i,i}\sE^{\oplus s_{i, n,m+1}}$ as a summand. Noting that the integers $s_{i, n,m}$ are eventually constant in $n$, and taking   the colimit, we get
\[
\sE\wedge\BGL_n\cong \oplus_{i=0}^\infty\Sigma^{2i,i}\sE^{\oplus s_{i, n}}
\]
with $s_{i,n}$ the rank of the degree $i$ part of $\Z[c_1, c_2,\ldots, c_n]$.

By  \cite[Theorem 2.2]{PPR2}, $\sE^{*,*}(\BGL_n)\cong \sE^{*,*}(k)[[c_1,c_2,\ldots, c_n]]_h$ and $\sE^{*,*}(\BGL)\cong \sE^{*,*}(k)[[c_1,c_2,\ldots]]_h$, where $\sE^{*,*}(k)[[c_1,c_2,\ldots]]_h\subset \sE^{*,*}(k)[[c_1,c_2,\ldots]]$ is the $\sE^{*,*}(k)$ submodule generated by the bi-homogeneous elements and $\sE^{*,*}(k)[[c_1,c_2,\ldots, c_n]]_h$ is defined similarly; the fact that $\sE^{a,b}(k)=0$ for $b<0$ says that the bi-homogeneous elements in $\sE^{*,*}(k)[c_1,c_2,\ldots]$ and $\sE^{*,*}(k)[[c_1,c_2,\ldots]]$ agree, so $\sE^{*,*}(\BGL)=\sE^{*,*}(k)[c_1,c_2,\ldots]$ and 
$\sE^{*,*}(\BGL_n)\cong \sE^{*,*}(k)[c_1,c_2,\ldots, c_n]$.

As 
\[
\sE^{*,*}(\BGL_n)=\Hom_{\Ho\sE\lMod}(\oplus_{i=0}^\infty\Sigma^{2i,i}\sE^{\oplus s_{i, n}}, \Sigma^{**}\sE)=\prod_{i=0}^\infty \sE^{*-2i, *-i}(k)^{s_{i,n}}
\]
we see that  a homogeneous $\Z$-basis for the dual  $\Z[b_1, b_2,\ldots]^{(c_{n+1}, c_{n+2},\ldots)\perp}\subset \Z[b_1, b_2,\ldots]$ to $\Z[c_1,c_2,\ldots, c_n]$ furnishes a $\sE$-basis for $\sE\wedge\BGL_n$ in $\Ho\sE\lMod$, that is, 
\[
\sE\wedge\BGL_n\cong \Z[b_1, b_2,\ldots]^{(c_{n+1}, c_{n+2},\ldots)\perp}\otimes \sE.
\]
Taking the colimit over $n$ gives the isomorphism 
\[
\sE\wedge\BGL\cong \Z[b_1, b_2,\ldots]\otimes \sE.
\]
The Whitney formula for the Chern classes imply dually that the maps $\BGL_n\times \BGL_{n'}\to \BGL_{n+n'}$ classifying $p_1^*\sV_n\oplus p_2^*\sV_{n'}$  induces a product on  $\sE_{**}(\BGL)$ 
 corresponding via this isomorphism to the multiplication in $\Z[b_1, b_2,\ldots]$. 

For $\BSL$, we recall that  $\BSL_n$ is the complement  of the 0-section in the line bundle $\sO_n(-1)=\det\sV_n$, so we have the localization sequence
\[
\ldots\to \sE^{a,b}(\Th(\sO_n(-1))\xrightarrow{\alpha} \sE^{a,b}(\sO_n(-1))\to \sE^{a,b}(\BSL_n)\xrightarrow{\delta}  
\sE^{a+1,b}(\Th(\sO_n(-1))\to \ldots
\]
By homotopy invariance, $\sE^{a,b}(\sO_n(-1))=\sE^{a,b}(\BGL_n)$, and the Thom isomorphism identifies $\sE^{a,b}(\Th(\sO_n(-1))\xrightarrow{\alpha} 
\sE^{a,b}(\sO_n(-1))$ with $\times c_1:\sE^{a-2,b-1}(\BGL_n)\to \sE^{a,b}(\BGL_n)$. This gives
\[
\sE^{*,*}(\BSL_n)\cong \sE^{*,*}(k)[c_1, c_2,\ldots, c_n]/(c_1)\cong \sE^{*,*}(k)[c_2,\ldots, c_n] 
\]
and thus 
\[
\sE^{*,*}(\BSL)\cong \sE^{*,*}(k)[c_1, c_2,\ldots]/(c_1)\cong \sE^{*,*}(k)[c_2,\ldots] 
\]

Just as for $\BGL_n$, taking the colimit of the homotopy cofiber sequences $\sO_n(1)\setminus 0_{\sO_n(1)}\to \sO_n(1)\to \Th(\sO_n(1))$, applying $\sE\wedge-$ and using the Thom isomorphism gives a homotopy cofiber sequence
\[
\sE\wedge\BSL_n\to\sE\wedge\BGL_n \xrightarrow{\alpha_n} \Sigma^{2,1}\sE\wedge\BGL_n
\]
Passing to the colimit over $n$ gives the homotopy cofiber sequence
\[
\sE\wedge\BSL\to\sE\wedge\BGL \xrightarrow{\alpha} \Sigma^{2,1}\sE\wedge\BGL
\]

Inserting our computation of $\sE\wedge\BGL$, we have 
\[
\alpha:\Z[b_1, b_2,\ldots]\otimes \sE\to \Z[b_1, b_2,\ldots]\otimes \Sigma^{2,1}\sE
\]
Applying $\Hom_{\Ho-\sE\lMod}(-, \Sigma^{**}\sE)$ gives
\[
\alpha^*:\Z[c_1, c_2,\ldots]\otimes_\Z\sE^{*, *}(k)\to \Z[c_1, c_2,\ldots]\otimes_\Z\sE^{*, *}(k);
\]
by the very definition of $c_1$, $\alpha^*$ is multiplication by $c_1$. The split injective map $\times c_1:\Z[c_1, c_2,\ldots]_{d-1}\to \Z[c_1, c_2,\ldots]_d$  is the $\Z$-dual of a split surjection 
$\alpha_*:\Z[b_1, b_2,\ldots]_{-d}\to \Z[b_1, b_2,\ldots]_{1-d}$ with kernel the free $\Z$-module $[(c_1)^\perp]_{-d}$, where the subscript $m$ mean the degree $m$ part. Thus the map $\alpha$ splits in $\Ho-\sE\lMod$ and defines an isomorphism $\sE\wedge\BSL\cong (c_1)^\perp\otimes_\Z \sE$.

For each $n>0$, we have the Thom isomorphism $\sE^{*,*}(\BGL_n)\to \sE^{2n+*, n+*}(\Th(\sV_n))=\sE^{2n+*, n+*}\MGL_n$,  and thus an isomorphism
\[
th_n: \sE^{*,*}(\BGL_n)\to \sE^{*,*}\Sigma^{-2n,-n}\MGL_n
\]
Recalling that $\MGL=\colim_n\Sigma^{-2n,-n}\MGL_n$, it is shown in   \cite[Lemma 2.5]{PPR2}, that the maps  $th_n$ fit together to define an isomorphism $\sE{\bideg}(\BGL)\xrightarrow{\sim} \sE{\bideg}(\MGL)$. Similarly, the Thom isomorphisms $\sE_{**}(\BGL_n)\to \sE_{2n+*, n+*}(\MGL_n)$ fit together to give the isomorphism $\sE_{**}(\MGL)=\sE_{**}(k)[b_1, b_2,\ldots]$, or more generally, $\pi_{**}(\sE\wedge\MGL)\cong \pi_{**}(\sE)[b_1, b_2,\ldots]$. Arguing as in the proof of Lemma~\ref{lem:FilDecom}, we arrive at the isomorphism in $\Ho\sE\lMod$
\[
\sE\wedge\MGL\cong \Z[b_1, b_2,\ldots]\otimes\sE.
\]

Applying the Thom isomorphism to the pullback $\tilde\sV_n\to \BSL_n$  of $\sV_n$  gives the Thom isomorphism $\sE{\bideg}(\BSL)\xrightarrow{\sim} \sE{\bideg}(\MSL)$. 
Just as for $\MGL$, the Thom isomorphisms $\sE^{*,*}(\BSL_n)\cong \sE^{2n+*, n+*}(\MSL_n)$ fit together to give $\sE^{**}(\MSL)\cong \sE^{*,*}(\BSL)$ and $\sE\wedge\MSL\cong \sE\wedge\BGL$.
\end{proof}

\begin{corollary}\label{cor:MGLMSLCohHom} Let $\sE\in \SH(k)$ be an oriented motivic commutative ring spectrum such that $\sE^{a,b}(k)=0$ for $b<0$.  Then for $M=\MGL, \MSL$ and for $m\ge1$ an integer, the natural map
\[
 \sE^{**}(M)^{\otimes_{\sE^{**}(k)} m}\to \sE^{**}(M^{\wedge m})
\]
is an isomorphism.
\end{corollary}

\begin{proof} By Theorem~\ref{thm:MGLMSLCohHom}(2), we have
\[
\sE\wedge (M^{\wedge m})\cong (\sE\wedge M)^{\wedge_{\sE} m}\cong
\oplus_i \Sigma^{2n_i, n_i}\sE
\]
for a suitable sequence of non-negative integers $n_i$. Moreover, for fixed $n$, there are only finitely many indices $i$ for which $n_i=n$. 

We have
\[
\sE^{a,b}(M^{\wedge m})=\Hom_{\Ho\sE\lMod}(\oplus_i\Sigma^{2n_i,n_i}\sE, \Sigma^{a, b}\sE)=\prod_i \sE^{a-2n_i, b-n_i}(k) =\oplus_i \sE^{a-2n_i, b-n_i}(k) 
\]
the last identity following from the vanishing $\sE^{a,b}(k)=0$ for  $b<0$ and our assertion on the sequence $\{n_i\}$. If we write
\[
\sE\wedge M\cong \oplus_j \Sigma^{2m_j, m_j}\sE
\]
then comparing $(\oplus_j \Sigma^{2m_j, m_j}\sE)^{\wedge_{\sE}m}$ and $\oplus_i \Sigma^{2n_i, n_i}\sE$ gives the isomorphism, induced by the product in $\sE$-cohomology
\[
 \sE^{**}(M)^{\otimes_{\sE^{**}(k)} m}\to \sE^{**}(M^{\wedge m})
\]
\end{proof}

\begin{remark}\label{rem:Pairing}  For $\sE\in \SH(k)$ a motivic commutative ring spectrum with multiplication $\mu_\sE:\sE\wedge\sE\to \sE$,  and $\sF\in \SH(k)$ arbitrary, we have the pairing 
\[
\sE_{a,b}(\sF)\otimes_{\sE^{0,0}(k)}\sE^{a,b}(\sF)\to \sE^{0,0}(k)
\]
defined by sending  $\alpha:\Sigma^{a,b}\mS_k\to \sE\wedge \sF$ and $\beta:\sF\to \Sigma^{a,b}\sE$ to $\mu_\sE\circ (\id_\sE\wedge\beta)\circ \alpha$. 

Recall that for an index $I=(i_1,\ldots,i _n)$, we set $b^I:=b_1^{i_1}\cdots b_n^{i_n}$ and $|I|:=\sum_jj\cdot i_j$. Via the isomorphism $H\Z\wedge\MGL\cong \Z[b_1, b_2,\ldots]\otimes H\Z$, we may view each monomial $b^I$ as an $H\Z$-module map $b^I:\Sigma^{2|I|, |I|}H\Z\to H\Z\wedge\MGL$, or equivalently, as a map $b^I:\Sigma^{2|I|, |I|}\mS_k\to H\Z\wedge\MGL$ in $\SH(k)$, that is, $b^I\in H\Z_{2|I|, |I|}(\MGL)$. By our construction, the pairing $\Z[c_1, c_2,\ldots]\times \Z[b_1, b_2,\ldots]\to \Z$ with the Conner-Floyd basis $\{c_I\}$ for  $\Z[c_1, c_2,\ldots]$ dual to the monomial basis $\{b^I\}$ of $\Z[b_1, b_2,\ldots]$ is compatible with the pairing $H\Z_{a,b}(\MGL)\times H\Z^{a,b}(\MGL)\to H\Z^{0,0}(k)=\Z$ defined in the previous paragraph. This gives us an interpretation of the generators $b_1, b_2, \ldots$ as coming from $H\Z_{2*,*}(\P^\infty)$, which we proceed to describe.

The orientation $\thom_{\MGL}\in \MGL^{2,1}(\P^\infty)$, viewed as a map $\thom_{\MGL}:\Sigma^\infty_{\P^1}\P^\infty\to \Sigma^{2,1}\MGL$, gives rise to a map  on $H\Z$-homology
\[
\thom_{\MGL*}:H\Z_{2n+2,n+1}(\P^\infty)\to H\Z_{2n,n}(\MGL).
\]
Arguing as in the proof of Theorem~\ref{thm:MGLMSLCohHom}, the isomorphism $H\Z^{2*,*}(\P^\infty)=\Z[\xi]$, with $\xi=c_1(\sO_{\P^\infty}(-1))$,  gives us elements $\beta_i\in H\Z_{2i,i}(\P^\infty)$, that is $\beta_i:\Sigma^{2i, i}\mS_k\to H\Z\wedge\P^\infty$, with $\<\xi^j,\beta_i\>=\delta_{ij}$, and a corresponding decomposition $H\Z\wedge\P^\infty=\oplus_{i\ge0}\Sigma^{2i,i}H\Z\cdot \beta_i$. As  $\xi^{i+1}=\thom_{\MGL}^*c_1^i$, we have dually $\thom_{\MGL*}(\beta_{i+1})=b_i$. Indeed,  we have $\thom_{\MGL}^*c_I=0$ if $I=(i_1, \ldots, i_r)$ and $i_j>1$ for some $j$, and $\thom_{\MGL}^*c_I=\thom_{\MGL}^*c_1^j$ if $i_j=1$ and $i_l=0$ for all $l\neq j$. 
See for example \cite[Part I, \S3, 4]{Adams74} for details in the topological setting, which is exactly parallel to the motivic one.
\end{remark}

\section{The $\MSL$-class of a Calabi-Yau variety}
\label{sec:CY}

For $X$ a smooth projective scheme of dimension $d$ over our base-field $k$, there is a corresponding cobordism class $[X]_{\MGL}\in \MGL^{-2d, -d}(k)$.  Here we will recall the definition of $[X]_{\MGL}$ and describe how to lift the class $[X]_{\MGL}$ to a class $[X,\theta_X]_{\MSL}\in \MSL^{-2d, -d}(k)$, given a trivialization $\theta_X:\det T_{X/k}\to \sO_X$  of the determinant bundle $\det T_{X/k}:=\Lambda^{d_X}T_{X/k}$. This is just an algebraic version of the classical Pontryagin-Thom construction, relying on the six-functor formalism in $\SH(-)$, established by constructions of Voevodsky \cite[\S2]{Voev}, Ayoub \cite{Ayoub}, Cisinski-D\'eglise \cite{CD} and Hoyois \cite{Hoy6Op}, which we will briefly recall. 

Fix a base-scheme $B$ and let $\Sch_B$ be the category of quasi-projective $B$-schemes; for the remainder of this section, we suppress the mention of $B$ and speak of a scheme $X$ rather than a quasi-projective $B$-scheme and a morphism of schemes rather than a morphism in $\Sch_B$.  For $S$  a scheme, we let $\Sm_S\subset \Sch_S$ be the full subcategory of smooth (quasi-projective) $S$-schemes. Let $\mathbf{Tr}$ be the category of symmetric monoidal triangulated categories.  There is a functor $\SH(-):\Sch_B^\op\to \mathbf{Tr}$ with $\SH(S)$ the motivic stable homotopy category over $S$ and for $f:T\to S$ in $\Sch_B$, $\SH(f)$ is the pull-back $f^*:\SH(S)\to \SH(T)$. $f^*$ admits the right adjoint $f_*:\SH(T)\to \SH(S)$ and for $f$ smooth, $f^*$ has the left adjoint $f_\#:\SH(T)\to \SH(S)$. There is a canonical isomorphism $f_\#(1_T)$ with the $\P^1$-suspension spectrum  $\Sigma^\infty_{\P^1_S}T_+\in \SH(S)$. There is an additional adjoint pair $f_!:\SH(T)\xymatrix{\ar@<3pt>[r]&\ar@<3pt>[l]}\SH(S):f^!$ and a natural transformation $\eta_f:f_!\to f_*$, which is an isomorphism if $f$ is projective. 

Let $p:V\to X$ be a vector bundle over some scheme $X$ with zero-section $s:X\to V$. We have the endofunctors $\Sigma^V:=p_\#\circ s_!$ and $\Sigma^{-V}:=s^!\circ p^*$ of $\SH(X)$, with $\Sigma^V$ the left adjoint to $\Sigma^{-V}$. In fact, $\Sigma^V$ and $\Sigma^{-V}$ are inverse autoequivalences on $\SH(X)$.  The assignments $V\mapsto \Sigma^V$, $V\mapsto \Sigma^{-V}$ and natural isomorphisms $\Sigma^V\circ \Sigma^{-V}\cong \id\cong \Sigma^{-V}\circ \Sigma^V$ extends to a map of groupoids from the path groupoid of the $K$-theory space of $X$ to the groupoid $\Aut(\SH(X))$ of autoequivalences of $\SH(X)$ and natural isomorphisms of such. 

For $f:Y\to X$ a smooth morphism, there are canonical isomorphisms
\[
f_!\cong f_\#\circ \Sigma^{-T_f},\quad f^!\cong \Sigma^{T_f}\circ f^*
\]
where $T_f\to Y$ is the relative tangent bundle, that is $T_f=\Spec_{\sO_Y} \Sym^*_{\sO_Y}\Omega_{Y/X}$, the vector bundle over $Y$ whose sheaf of sections is the $\sO_Y$-dual of the locally free sheaf of relative K\"ahler differentials $\Omega_{Y/X}$. 

Let $p_X:X\to S$ be a smooth $S$-scheme and let $V\to X$ be a rank $r$-vector bundle on $X$. Then $p_{X\#}(\Sigma^V 1_X)$ is represented by the suspension spectrum $\Sigma^\infty_{\P^1}\Th(V)$, where $\Th(V)$ is the Thom space $\P(V\oplus \sO)/\P(V)$. Taking $S=\Spec k$, the Thom isomorphisms in $\MGL$-theory translate into a canonical isomorphism $\MGL^{a,b}(p_{X\#}(\Sigma^V 1_X))=\MGL^{a,b}(\Th(V))\cong \MGL^{a-2r, b-r}(X)$.  Replacing $\Sigma^V$  with $\Sigma^{-V}$ gives rise to a canonical isomorphism $\MGL^{a,b}(p_{X\#}(\Sigma^{-V} 1_X))\cong \MGL^{a+2r, b+r}(X)$. In particular, we have the  element $[-V]_{\MGL}\in \MGL^{-2r,-r}(p_{X\#}(\Sigma^{-V} 1_X))$ corresponding to the unit $1_X^{\MGL}:=p_X^*(1^{\MGL})\in \MGL^{0,0}(X)$. 

For $p_X:X\to S$, $p_Y:Y\to S$ and  $f:X\to Y$ a projective morphism in $\Sm_S$,  
we have the natural transformation $P_f:p_{Y!}\to p_{X!}\circ f^*$ defined as the composition
\[
p_{Y!}\xrightarrow{u_f} p_{Y!}\circ f_*\circ f^*\xrightarrow{\eta_f^{-1}}p_{Y!}\circ f_!\circ f^*\cong
p_{X!}\circ f^*.
\]
Here $u_f$ is the map induced by the unit $\id\to f_*\circ f^*$ of the adjunction. Applying $P_f(1_Y)$ gives the map $f^*:p_{Y!}(1_Y)\to p_{X!}(1_X)$  in $\SH(S)$. For $\sE\in \SH(S)$, $f^*$ induces the map
\[
f_*:\sE^{a,b}(p_{X!}(1_X))\to \sE^{a,b}(p_{Y!}(1_Y))
\]
by $f_*(\alpha):=\alpha\circ f^*$.

\begin{definition} Let $k$ be a field and let $X\in \Sm_k$ be a smooth projective $k$-scheme of pure dimension $d$ over $k$ with structure morphism $p_X:X\to \Spec k$. Let $T_{X/k}\to X$ be the tangent bundle of $X$ over $\Spec k$.  Define the cobordism class $[X]_{\MGL}$ by
\[
[X]_{\MGL}:=p_{X*}([-T_{X/k}]_{\MGL})\in \MGL^{-2d, -d}(k)=\MGL_{2d,d}(k).
\]
noting that $p_{k!}(1_k)=1_k$ and $p_{X!}(1_X)=p_{X\#}(\Sigma^{-T_{X/k}}1_X)$.
\end{definition}

We wish to lift $[X]_{\MGL}$ to a class in $\MSL_{2d, d}(k)$, assuming that $X$ is a Calabi-Yau variety. More precisely,  given an isomorphism $\phi_X:\det T_{X/k}\xrightarrow{\sim} \sO_X$, we will define a class $[X;\phi_X]_{\MSL}\in \MSL^{-2d, -d}(k)$ mapping to $[X]_{\MGL}$ via the canonical map $ \MSL^{-2d, -d}(k)\to  \MGL^{-2d, -d}(k)$. For this, we recall from \cite[\S 5]{PW} that $\MSL$  is an {\em $\SL$-oriented} motivic commutative ring spectrum, that is, given a rank $r$-vector bundle $V\to X$ on some $X\in \Sm_k$ and an isomorphism $\phi:\det V\to \sO_X$,  there is a  Thom class $\theta_{V, \phi}\in \MSL^{2r, r}(\Th(V))$ such that cup product with $\theta_{V, \phi}$ defines an isomorphism $\vartheta_{V,\phi}:\MSL^{a,b}(X)\to \MSL^{a+2r, b+r}(\Th(V))$. Thus, just as for $\MGL$, we have the class $[-V,\phi]_{\MSL}\in \MSL^{-2r, -r}(p_{X\#}(\Sigma^{-V} 1_X))$ corresponding to the unit $1_X^{\MSL}\in \MSL^{0,0}(X)$. As the canonical map $\alpha:\MSL\to \MGL$ sends $1_X^{\MSL}$ to $1_X^{\MGL}$ and maps $\theta_{V, \phi}\in \MSL^{2r, r}(\Th(V))$ to $\theta_V\in  \MGL^{2r, r}(\Th(V))$, we have $\alpha_*([-V,\phi]_{\MSL})=[-V]_{\MGL}$.

\begin{definition}  Let $p_X:X\to \Spec k$ smooth and projective of pure dimension $d$ over $k$, and suppose we have an isomorphism $\phi_X:\det T_{X/k}\xrightarrow{\sim} \sO_X$. Define $[X, \phi]_{\MSL}$ by 
\[
[X, \phi]_{\MSL}:=p_{X*}([-T_{X/k}, \phi_X])\in \MSL^{-2d,-d}(k)=\MSL_{2d,d}(k).
\]
\end{definition}
As $\alpha_*([-T_{X/k},\phi_X]_{\MSL})=[-T_{X/k}]_{\MGL}$, we have $\alpha_*([X, \phi]_{\MSL})=[X]_{\MGL}$, giving us a (non-unique!) lifting of $[X]_{\MGL}$ to a class in $\MSL^{-2d,-d}(k)$.

We recall the description of the cobordism class $[X]_{\MGL}$ in terms of the Chern numbers of $X$. As detailed in \S\ref{sec:prelim}, we have the graded polynomial rings $\Z[c_1, c_2,\ldots]$, $\Z[b_1, b_2,\ldots]$ with $\deg c_i=i$, $\deg b_i=-i$ and perfect pairings on homogeneous summands $\<-,-\>:\Z[c_1, c_2,\ldots]_n\times \Z[b_1, b_2,\ldots]_{-n}\to \Z$ by making the Conner-Floyd Chern class $c_I$ dual to the monomial $b^I$. Here $I=(i_1, i_2, \ldots, i_r)$ is an index with $|I|:=\sum_{j=1}^r j\cdot i_j=n$.

 We have as well the identification of $\Z[b_1, b_2,\ldots]_{-n}$ with $H_{2n, n}(\MGL)=\Hom_{\SH(k)}(\Sigma^{2n,n}\mS_k, H\Z\wedge \MGL)$ (Theorem~\ref{thm:MGLMSLCohHom}). The unit $H\Z$ induces the map  $\MGL=\mS_k\wedge \MGL\to H\Z\wedge \MGL$  which in turn induces the motivic Hurewicz map $h_{\MGL}:\MGL_{2n, n}(k) \to H_{2n, n}(\MGL)$. 
 
For a vector bundle $V\to X$, we have the Conner-Floyd Chern polynomial $c_{\bullet, t}(V):=1+\sum_Ic_I(V)t^I\in H^{2*,*}(X)[[t_1, t_2, \ldots]]$. As $c_{\bullet, t}(-)$ is multiplicative in exact sequences, $V\mapsto c_{\bullet, t}(V)$ extends to a group homomorphism $c_{\bullet, t}(-):K_0(X)\to (1+(t_1,\ldots, t_n)H^{*,*}(X)[[t_1, t_2, \ldots]])^\times$, in particular, we have classes $c_I(-T_{X/k})\in H^{2|I|,|I|}(X)$. 

\begin{prop}\label{prop:Hurw} 1. After inverting the exponential characteristic of $k$, the Hurewicz map $h_{\MGL}:\MGL_{2*, *}(k) \to H_{2*, *}(\MGL)=\Z[b_1, b_2,\ldots]$ is injective. \\
2. For  $X$ smooth and projective of dimension $d$ over $k$, 
\[
h_{\MGL}([X]_{\MGL})=\sum_{I,\ |I|=d}\deg_k(c_I(-T_{X/k}))\cdot b^I
\]
\end{prop}

\begin{proof} (1) Let $\lambda_b(t)\in \Z[b_1, b_2,\ldots][[t]]$ be the power series $\lambda_b(t)=t+\sum_{i\ge 1}t^{i+1}b_i$, with functional inverse $\lambda_b^{-1}(t)\in  \Z[b_1, b_2,\ldots][[t]]$, i.e., $\lambda_b(\lambda_b^{-1}(t))=t=\lambda_b^{-1}(\lambda_b(t))$. Let $F_b(u,v):=\lambda_b(\lambda_b^{-1}(u)+\lambda_b^{-1}(v))\in \Z[b_1, b_2,\ldots][[u,v]]$. Then $F_b(u,v)$ is a formal group law over $\Z[b_1, b_2,\ldots]$, hence there is a unique ring homomorphism $\log:\Laz\to \Z[b_1, b_2,\ldots]$ sending the universal formal group law $F_{univ}(u,v)$ to $F_b(u,v)$. In fact $\log$ is injective (see for example \cite[Part II, Theorem 7.8]{Adams74}).

We have the universal first Chern classes of $O(1)\to \P^\infty$, $c_1^{\MGL}\in \MGL^{2,1}(\P^\infty)$ and $c_1^{H\Z}\in H\Z^{2,1}(\P^\infty)$. Via the respective unit maps $\MGL\to H\Z\wedge \MGL$, $H\Z\to H\Z\wedge \MGL$, we consider $c_1^{\MGL}$ and $c_1^{H\Z}$ both as elements of $(H\Z\wedge \MGL)^{2,1}(\P^\infty)$. The proof of \cite[Part II, Corollary 6.6]{Adams74}, modified in the evident way, shows that  
\[
c_1^{\MGL}=c_1^{H\Z}\cdot(1+ \sum_{i\ge1}(c_1^{H\Z})^ib_i)\in (H\Z\wedge \MGL)^{2,1}(\P^\infty)=
H^{**}(k)[b_1, b_2,\ldots][[c_1^{H\Z}]]^{2,1}.
\]
Noting that $c_1^{H\Z}(L\otimes M)=c_1^{H\Z}(L)+c_1^{H\Z}(M)$, this shows that, applying the Hurewicz map $h_{\MGL}$ to the coefficients of the formal group law $F_{\MGL}(u,v)\in \MGL^{2*,*}(k)[[u,v]]$, we have
\[
h_{\MGL}(F_{\MGL})(u,v)=F_b(u,v)\in H\Z^{2*,*}(\MGL)[[u,v]]=\Z[b_1, b_2,\ldots ][[u,v].
\]
Recalling the   ring homomorphism $\phi_{\MGL}:\Laz\to \MGL^{-2*,-*}(k)$ classifying $F_{\MGL}$, this gives the commutative diagram
\[
\xymatrix{
\Laz\ar[r]^-{\phi_{\MGL}}\ar[d]_{\log}&\MGL_{2*,*}(k)\ar[d]^{h_{\MGL}}\\
\Z[b_1, b_2,\ldots]\ar@{=}[r]&H_{2*,*}(\MGL)
}
\]
Since $\log$ is injective and $\phi_{\MGL}$ is an isomorphism after inverting the exponential characteristic of $k$ (Theorem~\ref{thm:HMH}), we see that $h_{\MGL}$ is injective after  inverting the exponential characteristic.

For (2), since $p_X:X\to \Spec k$ is smooth and projective over $k$, the object $p_{X!}(1_X)$ of $\SH(k)$ is canonically isomorphic to the dual of the suspension spectrum $\Sigma^\infty_{\P^1}X_+$, and the map $P_{p_X}:\mS_k=\id_{\Spec !}(1_{\Spec k})\to p_{X!}(1_X)$ is the dual of $\Sigma^\infty_{\P^1}(p_X):\Sigma^\infty_{\P^1}X_+\to \mS_k$ (see \cite[Theorem 5.22, Theorem 6.9]{Hoy6Op}).  Taking a Jouanolou cover $q:\tilde{X}\to X$, we can find a vector bundle $\nu_X\to \tilde{X}$ and an isomorphism $q^*T_{X/k}\oplus \nu_X\cong \sO_{\tilde{X}}^N$ for suitable $N$. This gives an isomorphism  $\Sigma^{-T_{X/k}}\cong \Sigma_{\P^1}^{-N}\circ \Sigma^{\nu_X}$ of autoequivalences of $\SH(k)$, and thus an isomorphism $p_{X!}(1_X)\cong \Sigma^{-N}_{\P^1}\Sigma^\infty_{\P^1}\Th(\nu_X)$. 

Letting $r=N-d=\text{rank}(\nu_X)$, we have the map $f_\nu:\tilde{X}\to \BGL_r$ classifying $\nu_X$ and inducing the map on homology 
\[
f_{\nu*}:H_{2d, d}(X)=H_{2d, d}(\tilde{X})\to H_{2d, d}(\BGL_r)\cong 
H_{2(d+r), d+r}(\MGL_r)\to  H_{2d, d}(\MGL)
\]
and a pull-back map on cohomology, $f_\nu^*:H^{2d, d}(\MGL)\to H^{2d, d}(X)$. Similarly, the map $P_{p_X}:\mS_k\to \Sigma^{-N}_{\P^1}\Sigma^\infty_{\P^1}\Th(\nu_X)$ sends $1\in H_{0,0}(\Spec k)$ to $p_X^*(1)\in H_{0,0}(\Sigma^{-N}_{\P^1}\Sigma^\infty_{\P^1}\Th(\nu_X))=H_{2d, d}(X)$. The map $p_X^*$ defined here via duality in $\SH(k)$ agrees with the map $p_X^*$ defined via duality in $\DM(k)$, and one sees thereby that 
$p_X^*(1)\in H_{2d, d}(X)=\CH_{d}(X)$ is just the fundamental class $[X]_{H\Z}$, that is, the cycle $1\cdot X$. 

As $f_\nu$ classifies $\nu_X$, we have the commutative diagram
\[
\xymatrix{
\nu_X\ar[r]^{\tilde{f}_\nu}\ar[d]&\sV_r\ar[d]\\
\tilde{X}\ar[r]_-{f_\nu}&\BGL_r
}
\]
and the map $\tilde{f}_\nu$ induces $\tho(\tilde{f}_\nu):\Th(\nu_X)\to \Th(\sV_r)=\MGL_r$. The map  $[-T_{X/k}]_{\MGL}:p_{X!}(1_X)\to \Sigma^{2d, d}\MGL$ is by definition the map induced by $\Sigma^{-N}_{\P^1}\Sigma^\infty_{\P^1}(\tho(\tilde{f}_\nu))$ followed by $\iota_r:\Sigma^\infty_{\P^1}\MGL_r\to \Sigma^r\MGL$ and thus $[X]_{\MGL}$ is induced by $\iota_r\circ 
\Sigma^{-N}_{\P^1}\Sigma^\infty_{\P^1}(\th(\tilde{f}_\nu))\circ p_X^*$. 

Tracing through the various Thom isomorphisms,  it follows that $h_{\MGL}([X]_{\MGL})\in H_{2d, d}(\MGL)= H_{2d, d}(\BGL)$ is given by
\[
h_{\MGL}([X]_{\MGL})=f_{\nu*}([X]_{H\Z})
\]
and thus for a given index $I$ with $|I|=d_X$, we have
\[
\<c_I, h_{\MGL}([X]_{\MGL})\>=\<c_I, f_{\nu*}([X]_{H\Z})\>=\<f_\nu^*(c_I), p_X^*(1)\>=
\<p_{X*}(f_\nu^*(c_I)),1\>.
\]
Here we consider $f_\nu(c_I)$ as an element of $H^{2d_X, d_X}(X)$ via the isomorphism
$q^*: H^{2d, d}(X)\to H^{2d, d}(\tilde{X})$ and similarly consider  $p_X^*(1)\in H_{2d, d}(X)$. The map $f_\nu$ classifies the bundle $\nu_X$ on $\tilde{X}$, and as $[\nu_X]+[q^*T_{X/k}]=[\sO_{\tilde{X}}^N]$ in $K_0(\tilde{X})\cong K_0(X)$, it follows that $f_\nu^*(c_I)=c_I(-[T_{X/k}])$ in $H_{0,0}(X)=\CH^{d}(X)$. Thus 
\[
\<c_I, h_{\MGL}([X]_{\MGL})\>=\<p_{X*}(f_\nu^*(c_I)),1\>=\deg_k(c_I(-[T_{X/k}]))
\]
and so $h_{\MGL}([X]_{\MGL})=\sum_{I, |I|=d_X}c_I(-T_{X/k})b^I$, as claimed.
\end{proof}

We collect the main results of this section in the following theorem.

\begin{theorem}\label{thm:MSLLifitng} 1. After inverting the exponential characteristic of $k$, the Hurewicz map $h_{\MGL}:\MGL_{2*,*}(k)\to H\Z_{2*,*}(\MGL)=\Z[b_1, b_2,\ldots]$ is injective.\\
2. For $X$ a smooth projective $k$-scheme of dimension $d$, the class $[X]_{\MGL}\in \MGL_{2d, d}(k)$ satisfies $h_{\MGL}([X]_{\MGL})=\sum_{I,\ |I|=d}c_I(-T_{X/k})\cdot b^I$.\\
3.  Let $X$ be a smooth projective $k$-scheme of dimension $d$ with a trivialization $\theta_X:\det T_{X/k}\xrightarrow{\sim} \sO_X$  of the determinant bundle $\det T_{X/k}$. Then there is a class $[X,\theta_X]_{\MSL}\in \MSL_{2d, d}(k)$ mapping to $[X]_{\MGL}\in \MGL_{2d, d}(k)$ under the projection $\MSL\to \MGL$. \\
4. After inverting the exponential characteristic of $k$, ${\MGL}_{2d,d}(k)$ is generated by the classes $[X]_{\MGL}$, $X$ a smooth projective $k$-scheme of dimension $d$.
\end{theorem}

\begin{proof} We have proven everything except for (4), which follows from Theorem~\ref{thm:HMH} and  \cite[Corollary10.8]{Adams74}.
\end{proof}

\section{ The mod-$l$ motivic Steenrod algebra}\label{sec:Steenrod}
In this section, we recall some basic facts of the mod-$l$ motivic Steenrod algebra  introduced by Voevodsky in \cite{Voev03}.
For the rest of the paper (with the exception of Appendix A), $l$ will be an odd prime different from $\Char(k)$. 
For such $l$,  the motivic Steenrod algebra behaves quite  similarly   to its topological counterpart, as discussed in \cite{Adams}, \cite{Mil} and \cite{Nov}. Nevertheless, for the convenience of the readers, we collect  the properties relevant to us.

Let $A\bideg:=A\bideg(k, \Z/l)$ be the mod-$l$ \textit{motivic Steenrod algebra}. By definition \cite[\S11]{Voev03}, $A\bideg$ is the subalgebra of  {the algebra of bi-stable operations on mod $l$ motivic cohomology of smooth $k$-schemes generated by the motivic reduced Steenrod power operators $P^i, i\geq 0$,  the  Bockstein homomorphism $\beta$, and operators of the form $u\mapsto au$, where $a\in H\bideg$. We have  $P^0=1$, 
\[
\deg(P^i)=(2i(l-1), i(l-1)), \,\
\deg(\beta)=(1, 0). 
\]
As mod $l$ motivic cohomology is represented by the spectrum $H\Z/l\in \SH(k)$,  any endomorphism  of $H\Z/l$ represents a bi-stable operation; in fact (\cite[Theorem 3.49]{Voev}, \cite[Theorem 1.1]{HSO}) this map determines an isomorphism $\End_{\SH(k)}^{**}(H\Z/l)\cong A\bideg$, and thus $A\bideg$ acts on $H\bideg(\sF, \Z/l):=\Hom_{\SH(k)}(\sF, \Sigma^{**}H\Z/l)$ for any $\sF\in \SH(k)$. }

{The product on motivic cohomology gives rise to a bi-graded coproduct $\Delta:A^{**}\to A^{**}\otimes_{H^{**}}A^{**}$  via the identity for $\theta\in A^{**}$
\[
\Delta(\theta)=\sum_i \alpha_i\otimes \beta_i\Leftrightarrow
\theta(xy)=\sum_i(-1)^{ab_i}\alpha_i(x)\beta_i(y)
\]
if $x$ has bi-degree $(a, a')$ and $\beta_i$ has bi-degree $(b_i, b_i')$. Here we follow the convention in \cite{Voev03}, that all tensors over $H\bideg$ are tensors as left $H\bideg$-modules. }

{
 Following \cite[\S11]{Voev03}, there is an action of $A\bideg\otimes_{H\bideg}A\bideg$ on $A\bideg\otimes_{\Z/l}A\bideg$ with values in $A\bideg\otimes_{H\bideg}A\bideg$ given by
\[
(u\otimes v)\cdot (u'\otimes v')=uu'\otimes vv'.
\]
Let $(A\bideg\otimes_{H\bideg}A\bideg)_r \subset A\bideg\otimes_{H\bideg}A\bideg$ denote the subset of elements $f$ such that, if $x, y$ are in $A\bideg\otimes_{\Z/l}A\bideg$ with $x=y$ in $A\bideg\otimes_{H\bideg}A\bideg$, then $f\cdot x=f\cdot y$ in $A\bideg\otimes_{H\bideg}A\bideg$. The above product gives rise to a well-defined ring structure on 
$(A\bideg\otimes_{H\bideg}A\bideg)_r$ and gives  $A\bideg\otimes_{H\bideg}A\bideg$ an $(A\bideg\otimes_{H\bideg}A\bideg)_r$-module structure. }

{
The  image of $\Delta$ is contained in $(A\bideg\otimes_{H\bideg}A\bideg)_r$ and defines a ring homomorphism 
\[
\Psi^*: A\bideg\to (A\bideg\otimes_{H\bideg}A\bideg)_r,
\]
so in this  modified sense, we may refer to the ring $A^{**}$ with coproduct $\Psi^*$ as a bi-algebra. We refer the reader to \cite[Lemmas 11.6-.9]{Voev03} for details.} 

{
 We have the motivic Milnor basis 
\cite[Section 13]{Voev03} \[\{\rho(E, R)=Q(E) P^R\in A\bideg\mid R=(r_1, r_2, \ldots), E=(\epsilon_0, \epsilon_1, \ldots )\} ,\]  
where $R$ and $E$  are two sequences of integers,  $r_i \geq 0$ and $\epsilon_i\in \{0, 1\}$, which are non-zero for only finitely many $i$.  
For any $r\in\bbN$,  let $e_r=(0, \ldots, 0, 1, 0, \ldots )$, where the $1$ is on the $r$-th place, and set  $Q_r := Q(e_r)$.  We have $P^i=P^{i\cdot e_1} $ and $Q_{0}=\beta$. }

{
The topological mod $l$ Steenrod algebra $A^{\topo}$ has $\Z/l$-algebra generators $P^i_{\topo}$, $Q_i^{\topo}$, $i=0, 1, \ldots$, with $P^0_{\topo}=1$ and $Q_0^{\topo}=\beta_{\topo}$, where $\beta_{\topo}$ is the Bockstein operator \cite[Chap. IV, \S1, 2]{SteenrodEpstein}. There is also the classical Milnor basis \cite{Mil}, $\{\rho_{\topo}(E, R)=Q^{\topo}(E) P_{\topo}^R\}$, giving us the particular elements $Q_r^{\topo}:= Q^{\topo}(e_r)$, with  $P^i_{\topo}=P^{i\cdot e_1}_{\topo}$; $A^{\topo}$ has a coproduct $\Delta_{\topo}:A^{\topo}\to A^{\topo}\otimes_{\Z/l}A^{\topo}$, characterized via the products in $H^*(-,\Z/l)$ as for $A^{**}$, and making $A^{\topo}$ into a Hopf algebra over $\Z/l$.}

{
\begin{lemma}\label{lem:steen_generators} Let $l$ be an odd prime and let $k$ be a field of characteristic 0 or of characteristic $p>0$ with $p$ prime to $l$. The following properties hold.
\begin{enumerate}
\item  There is a unique ring homomorphism $\Xi:A^{\topo}\to A^{**}$ with $\Xi(Q^{\topo}(E) P_{\topo}^R)=Q(E) P^R$ for all $(E,R)$.   The $H^{**}$-linear extension of $\Xi$,    $\id\otimes\Xi:H^{**}\otimes_{\Z/l}A^{\topo}\to A^{**}$, is an isomorphism of left $H^{**}$-modules and  $\Xi$ is a map of bi-algebras. 
\item $Q_r$ has bi-degree $(2l^r-1, l^r-1)$ and for each $k\ge1$,  $P^{k\cdot e_r}$ has bi-degree $(2k(l^r-1), k(l^r-1))$. 
\item Define a partial order on $\{P^{k\cdot e_r}\mid k, r\ge1\}\cup \{Q_r\mid r\ge0\}$ by $ P^{k\cdot e_r}<P^{k'\cdot e_{r+1}}<Q_{r'}<Q_{r'+1}$ for all $k, k' r, r'$, and with $P^{k\cdot e_r}, P^{k'\cdot e_r}$ incomparable for $k\neq k'$. Then the set of finite monomials in the $P^{k\cdot e_r}, Q_{r'}$ which are strictly increasing when read from left to right, together with $1=P^0$,  forms a basis for $A^{**}$ as a left $H^{**}$-module. 
\item The $Q_i$ and $Q_j$ anti-commute and $Q_i^2=0$.
\item (The Cartan formulas).  
\begin{align*}
&\Psi^*(P^R)=\sum_{R_1+R_2=R}  P^{R_1}\otimes P^{R_2}, \\
&\Psi^*(Q_i)=Q_i\otimes 1+1\otimes Q_i.
\end{align*}
\item For $X\in \Sm_k$, let $L\to X$ be a line bundle, $c_1(L)\in H^{2,1}(X,\Z/l)$ the mod $l$ first Chern class. Then $P^0(c_1(L))=c_1(L)$, $P^1(c_1(L))=c_1(L)^l$ and $P^i(c_1(L))=0$ for $i\ge2$. Moreover, $(\beta\circ P^i)(c_1(L))=0$ for $i\ge0$. For $M$ a complex manifold and $L\to M$ a $\CC$-line bundle, the analogous formulas hold for $P^i_{\topo}(c_1^{\topo}(L))$ and $(\beta_{\topo}\circ P^i_{\topo})(c_1^{\topo}(L))$ in $H^*(M, \Z/l)$.
\end{enumerate}
\end{lemma}}

\begin{proof}  For (1),   the $P^i_{\topo}, \beta^{\topo}$ satisfy the Adem relations \cite[Chap. VI, \S1(6)]{SteenrodEpstein}, $(\beta^{\topo})^2=0$ and  $A^{\topo}$ is defined as the  $\Z/l$-algebra generated by the $P^i_{\topo}, \beta^{\topo}$ and satisfying the Adem relations. As the $P^i$ and $\beta$ also satisfy the Adem relations  \cite[Theorem 5.1]{HSO} and $\beta^2=0$, there is a unique ring homomorphism $\Xi:A^{\topo}\to A^{**}$ sending $P^i_{\topo}$ to $P^i$ and $ \beta_{\topo}$ to $\beta$. 

The so-called admissible monomials in the $P^i_{\topo}$ and  $\beta_{\topo}$  form a $\Z/l$-basis of $A^{\topo}$ \cite[Chap. VI, \S 2, Theorem 2.5]{SteenrodEpstein} and by  \cite[Lemma 11.1, Corollary 11.5]{Voev}, the admissible monomials in the $P^i$ and $\beta$ form an $H^{**}$-basis of $A^{**}$ (as left $H^{**}$-module).  Thus $\id\otimes\Xi:H^{**}\otimes_{\Z/l}A^{\topo}\to A^{**}$ is an isomorphism of left $H^{**}$-modules. 

For the Cartan formulas for $\Psi^*(P^n)$ and $\Psi^*(\beta)$,  see \cite[Proposition 9.7]{Voev03}. 
The coproduct on $A^{**}$ is defined via the dual of the Cartan formula, and similarly for the coproduct on $A^{\topo}$; as the analog of the Cartan formula holds in $A^{\topo}$ \cite[Chap. VI, \S2]{SteenrodEpstein},  $\Xi$  thus defines a map of bi-algebras.

There are  elements $\tau_i, \xi_i$ in the $H^{**}$-dual $A_{**}$ of $A^{**}$, which by  \cite[Lemma 12.3]{Voev03} are  the respective duals of $P^{e_i}\cdot \beta$ and $P^{e_i}$ with respect to the $H^{**}$-basis of  $A^{**}$ given by  admissible monomials in $\beta$ and the $P^i$. Similarly (\cite[Chap VI, \S 3]{SteenrodEpstein}), there are elements $\tau_i^{\topo}, \xi_i^{\topo}$ in the $\Z/l$-dual $A_{\topo}$ of $A^{\topo}$ which are  the respective duals to $P^{e_i}_{\topo}\cdot \beta_{\topo}$ and  $P^{e_i}_{\topo}$, with respect to the $\Z/l$ basis of $A^{\topo}$ given by the admissible monomials in $\beta_{\topo}$ and the $P^i_{\topo}$.    The elements $\{Q(E) P^R\}$ are defined as the dual basis to the basis of $A_{**}$ consisting  of  monomials $\{\tau^E\xi^R\}$ in the $\tau_i, \xi_i$ and the elements $\{Q^{\topo}(E) P_{\topo}^R\}$ are similarly defined using the monomials $\{\tau^E_{\topo}\xi^R_{\topo}\}$  in $\tau^{\topo}_i, \xi_i^{\topo}$.  Since $\Xi$ is a map of bi-algebras,  $\Xi(Q^{\topo}(E) P_{\topo}^R)=Q(E) P^R$; the uniqueness of such a $\Xi$ follows from the identities $P^i=P^{i\cdot e_1}$, $\beta=Q_0$, $P^i_{\topo}=P^{i\cdot e_1}_{\topo}$, $\beta_{\topo}=Q^{\topo}_0$, and the uniqueness for $\Xi$ that we have already established. 

The formulas for the bi-degrees for $P^{k\cdot e_r}$ and $Q_r$  follow directly from their definition in \cite[\S 13]{Voev03}. The assertions (3) and (4) follow from  (1) and  the corresponding assertions in $A^{\topo}$ \cite[\S 2.2]{Nov} or \cite[pg. 200]{Adams}. Here we note that in the notation of Novikov and Adams, we have $e_{k,r}=P_{\topo}^{k\cdot e_r}$ and $e'_r=Q_r^{\topo}$, and that we have stated a less general assertion than given in {\it loc. cit.}.

The Cartan formulas (5) for $\Psi^*(P^R)$ and $\Psi^*(Q_i)$ are the duals of the multiplication of the monomials $\tau^E\xi^R$.

For (6), the formulas $P^0(c_1(L))=c_1(L)$, $P^1(c_1(L))=c_1(L)^l$ and $P^i(c_1(L))=0$ for $i\ge2$ are found in  \cite[Lemma 6.7]{Voev03}. Since $c_1(L)$ is the reduction mod $l$ of the integral class $c_1(L)_\Z\in H^{2, 1}(X,\Z)$, $\beta(c_1(L))=0$ and $\beta(c_1(L)^l)=0$. For the topological case, the formulas for $P^i_{\topo}(c_1^{\topo}(L))$ follow from the axioms (2), (3), (4) of  \cite[Chap. VI, \S 1]{SteenrodEpstein} and the vanishing of the $\beta_{\topo}(c_1^{\topo}(L))$ and $\beta(c_1^{\topo}(L)^l)$ follows exactly as in the motivic case.
\end{proof}

\begin{remark} The essential difference between the topological and motivic Steenrod algebras for an odd prime $l$ is that the motivic Steenrod algebra acts in general non-trivially on $H^{**}$, so one must carefully distinguish between the left and right $H^{**}$-module structures. In particular, even though the isomorphism $\id\otimes\Xi:H^{**}\otimes_{\Z/l}A^{\topo}\to A^{**}$ of left $H^{**}$-modules is multiplicative on $A^{\topo}$, it  is {\em not} in general an isomorphism of rings.
\end{remark}

\section{Modules over the motivic Steenrod algebra}
\label{sec:SteenMod}

\subsection{A quotient of the Steenrod algebra}

Let $B \subset A\bideg$ be the $\Z/l$-subalgebra generated by $\{Q_i\}_{i\geq 0}$.   Let 
$M_B:=A\bideg/ A\bideg(Q_0, Q_1, \ldots )$ be the quotient of $A\bideg$ by the left ideal generated by $(Q_0, Q_1, \ldots )$. We similarly define $B_{\topo}\subset A_{\topo}$ and  $M_B^{\topo}:=A^{\topo}/A^{\topo}(Q_0^{\topo}, Q_1^{\topo},\ldots)$.

By essentially the same argument as in \cite[Lemma~8, 10]{Nov}, one has the following.
\begin{lemma}\label{lem:ext_iso}
There is an isomorphism of tri-graded abelian groups
\[\bigoplus_{s, t, u}\Ext^{s,(t-s,u)}_{A\bideg}(M_B,H\bideg)\cong 
\bigoplus_{s, t, u}\Ext^{s,(t-s,u)}_B(\Z/l, H\bideg).\]
\end{lemma}
 
\begin{proof}We rephrase the argument of \cite[\hbox{\it loc. cit.}]{Nov}.   By Lemma~\ref{lem:steen_generators}   $M_B$ is isomorphic as a left $H^{*, *}$-module to the free 
left $H^{*, *}$-module on the monomials $P^{k_1\cdot e_{i_1}}\cdot \ldots\cdot P^{k_m\cdot e_{i_m}}$ with $1\le i_1<\ldots<i_m$, $m=0, 1, \ldots$ and $k_j\ge1$ for all $j$, and thus 
$A\bideg$ is a free bi-graded right $B$-module with basis $\{u_\alpha\cdot P^{k_1\cdot e_{i_1}}\cdot \ldots\cdot P^{k_m\cdot e_{i_m}}\}$, where $(u_\alpha)_\alpha$ is a bi-graded $\Z/l$-basis of $H^{**}$ and the monomials $P^{k_1\cdot e_{i_1}}\cdot \ldots\cdot P^{k_m\cdot e_{i_m}}$ are the basis for $M_B$ over $H^{**}$ described above. In particular, $A\bideg$ is flat over $B$. Moreover, we have the isomorphism of bi-graded left $A\bideg$-modules
\[
M_B\cong A\bideg\otimes_B\Z/l.
\]
Thus if $C^*_B\to \Z/l$ is a resolution of the $B$-module $\Z/l$ by free bi-graded $B$-modules, then $A\bideg \otimes_BC^*_B\to M_B$ is a resolution of $M_B$ by free bi-graded $A\bideg$-modules. For $N$ a bi-graded left $A\bideg$-module, which becomes a bi-graded left  $B$-module by restriction,  the isomorphism of complexes
\[
\Hom_{A\bideg}(A\bideg \otimes_BC^*_B, N)\cong \Hom_{B}(C^*_B, N)
\]
therefore induces an isomorphism of tri-graded Ext-groups
\[
\Ext^{*,(*,*)}_{A\bideg}(M_B,N)\cong \Ext^{*,(*,*)}_{B}(\Z/l,N)
\]
\end{proof}

We now compute the right hand of the isomorphism in Lemma \ref{lem:ext_iso}. 
By Lemma~\ref{lem:steen_generators},   $B$ is an exterior algebra over $\Z/l$ on $\{Q_r\}$. 
Let $V$ be the $\Z/l$-vector space spanned by $\{Q_r\}$.  We have the following Koszul resolution of $\Z/l$ by free bi-graded $B$-modules: 
 \begin{equation}\label{equ:Koszul}
 \cdots\to \Sym^2V\otimes_{\Z/l}B\to V\otimes_{\Z/l}B \to B\to \Z/l\to 0.
 \end{equation}

\begin{lemma}\label{lem:ExtAlg1}
\begin{enumerate}
\item For any $u$, we have $\Ext^{s,(t-s,u)}_B(\Z/l,H\bideg)=0$  if  $t>2u$.
\item  $\bigoplus_{s, u}\Ext^{s,(2u-s,u)}_B(\Z/l, H\bideg)$ is a polynomial algebra $\Z/l[\{h'_r\}_{r\geq0}]$ over $\Z/l$, where $\deg(h'_r)=(1,(1-2l^r,1-l^r))$. 
\item We have $\Ext^{0,(2u-1,u)}_B(\Z/l,H\bideg)=0$ for $u\neq1$,  $\Ext^{0,(1,1)}_B(\Z/l,H\bideg)=H^{1,1}$,  and the product map
\begin{multline*}
H^{1,1}\otimes\bigoplus_{s, u}\Ext^{s,(2u-s,u)}_B(\Z/l, H\bideg)=
\Ext^{0,(1,1)}_B(\Z/l,H\bideg)\otimes\bigoplus_{s, u}\Ext^{s,(2u-s,u)}_B(\Z/l, H\bideg)\\\to
\bigoplus_{s, u}\Ext^{s,(2u-s+1,u+1)}_B(\Z/l, H\bideg)
\end{multline*}
is surjective.
\end{enumerate}
\end{lemma}

\begin{proof}
Applying $\Hom_B(-,H\bideg)$ to the   Koszul resolution \eqref{equ:Koszul}  of $\Z/l$, one deduces that
$ \Ext^{s,(*,*)}_B(\Z/l, H\bideg)$ is a subquotient of $H\bideg\otimes_{\Z/l}\Sym^sV^\vee$. 
As the basis element $Q_r$ of $V$ has $\deg(Q_{r})=(2l^r-1,l^r-1)$,
$\Sym^sV^\vee$ has bi-degrees $(s-2\sum_{i=1}^s l^{r_i}, s-\sum_{i=1}^s l^{r_i})$, for $ r_1, \ldots, r_s\in \N$. 
The part of $H^{p, q}\otimes\Sym^sV^\vee$ that contributes to $ \Ext^{(s, (t-s,u))}$ satisfies 
\[
(t-s,u)=(p+s-2\sum_{i=1}^s l^{r_i}, q+s-\sum_{i=1}^s l^{r_i}), 
\]
which implies $t-2u=p-2q$. By Theorem \ref{Thm19.3}, we know that $H^{p,q}=0$,  if $p>q$ or if $q<0$. This forces $t\leq 2u$,  and $t=2u$ iff   $p=q=0$. 
This implies assertion (1). 

We now prove (2). Let $\bar{B}\subset B$ be augmentation ideal.  By our calculation above, the part of $H^{p, q}\otimes\Sym^sV^\vee$ that contributes to $ \Ext^{(s, (2u-s,u))}$ satisfies $0=p-2q$. Since $H^{p,q}=0$ for $p>q$, or for $q<0$, this forces $(p,q)=(0,0)$. Since $Q_r$ has bi-degree $(2l^r-1, l^r-1)$, $\bar{B}$ acts by zero on $H^{0,0}$, and since $H^{-1,0}=0$ and $H^{p,q}=0$ for $q<0$, $\bar{B}\cdot H^{*,*}\cap H^{0,0}=\{0\}$. This shows that the differentials in the complex 
\[
\oplus_u\Hom_B^{2u-*,u}(\Sym^*V\otimes B, H^{*,*})=(\Sym^*V^\vee)
\]
are all zero, and thus $\oplus_u\Ext^{s, (2u-s, u)}=\Sym^sV^\vee$. One checks directly that this gives an isomorphism of the Ext-algebra $\oplus_{s, u} \Ext^{s,(2u-s, u)}_B(\Z/l,H^{*,*})$ with the symmetric algebra $\Sym^*(V^\vee)$. 
Thus, letting $h'_r\in\Ext_B^{1,(1-2l^r,1-l^r)}(\Z/l,H^{*,*})$ correspond to the dual of $Q_r$,  $\{h'_r \mid r\geq0\}$ is a set of  polynomial generators over $\bbZ/l$ of $\oplus_{s, u} \Ext^{s,(2u-s, u)}_B(\Z/l,H^{*,*})$, and the unit is given by the canonical identification $\Z/l=H^{0,0}$ 
  in $\Ext^{0,(0,0)}_B(\Z/l,H^{*,*})=\Hom_{\Z/l}(\Z/l,H^{0,0})$. 

The proof of (3) is similar.  The part of $H^{p, q}\otimes\Sym^sV^\vee$ that contributes to $ \Ext^{(s, (2u-1-s,u))}$ satisfies  $-1=p-2q$. The vanishing of $H^{p,q}$ for $p>q$ and $q<0$ implies only $H^{1,1}\otimes(\Sym^sV^\vee)^{2u-s-2,u-1}$ contributes to $\Ext^{(s, (2u-1-s,u))}$. As above,
$\bar{B}\cdot H^{1,1}=0$ for degree reasons, so the differential leaving $H^{1,1}\otimes(\Sym^sV^\vee)^{2u-s-2,u-1}$ is zero, and thus $\Ext^{(s, (2u-1-s,u))}$ is a quotient of 
\[
H^{1,1}\otimes(\Sym^sV^\vee)^{2u-s-2,u-1}=H^{1,1}\otimes \Ext^{(s, (2u-2-s,u-1))}. 
\]
For $s=0$, there is no differential mapping to 
$H^{1,1}\otimes(\Sym^0V^\vee)^{2u-2,u-1}$ and the only non-zero term is for $u=1$, giving
\[
\Ext^{0, (2u-1, u)}_B(\Z/l, H^{*,*})=\begin{cases}0&\text{ for }u\neq1\\ H^{1,1}&\text{ for }u=1
\end{cases}
\]
\end{proof}

Let $N$ be a left $A\bideg$ module. The left $A\bideg\otimes_{\Z/l}A\bideg$ module structure on $N\otimes_{\Z/l} N$ descends to a left  $(A\bideg\otimes_{H\bideg}A\bideg)_r$ module structure on $N\otimes_{H\bideg} N$, which via $\Psi^*$ makes $N\otimes_{H\bideg} N$ a left $A\bideg$ module. We will apply this for  $N=M_B$, and note that the map
\[
\Psi^*_{M_B}:M_B\to M_B\otimes_{H\bideg}M_B
\]
induced by $\Psi^*$ is $A\bideg$ linear. 

\begin{remark}\label{rem:ProductStructureExtGps}
For $N, M, N', M'$ $A\bideg$ modules, and $f:N\to M$, $f':N'\to M'$  $A\bideg$-linear maps, the tensor product $f\otimes f':N\otimes_{\Z/l}N'\to M\otimes_{H\bideg}M'$ descends to an $A\bideg$-linear map $f\otimes_{H\bideg} f':N\otimes_{H\bideg}N'\to M\otimes_{H\bideg}M'$. Thus, if $N$ or $N'$ is flat as an $H\bideg$-module, this induces a map on the $\Ext$-groups
\[
\Ext^{*,(*,*)}_{A\bideg}(N, M)\otimes_{H\bideg}\Ext^{*,(*,*)}_{A\bideg}(N', M')\to
\Ext^{*,(*,*)}_{A\bideg}(N\otimes_{H\bideg}N', M\otimes_{H\bideg}M')
\]
The co-associativity of $\Psi^*$ shows that this product is associative.
\end{remark}
 
Similarly, the coproduct $\Delta:B\to B\otimes_{\Z/l}B$, $\Delta(Q_r)=Q_r\otimes 1+1\otimes Q_r$,   makes $H\bideg\otimes_{\Z/l}H\bideg$ a left $B$ module for which the product $m:H\bideg\otimes_{\Z/l}H\bideg\to H\bideg$ is $B$-linear. This gives $\Ext_B^{*,(*,*)}(\Z/l, H\bideg)$ a tri-graded ring structure via the composition
\[
\Ext_B^{*,(*,*)}(\Z/l, H\bideg)\otimes_{\Z/l}\Ext_B^{*,(*,*)}(\Z/l, H\bideg)\xrightarrow{\boxtimes}
\Ext_B^{*,(*,*)}(\Z/l, H\bideg\otimes_{\Z/l} H\bideg)\xrightarrow{m_*}
\Ext_B^{*,(*,*)}(\Z/l, H\bideg).
\]

{
We consider $A^{**}\otimes_{\Z/l}A^{**}$ as a bi-graded left $H^{**}$-module by $u\cdot(a\otimes b):=ua\otimes b$. Recalling the isomorphism of left $H^{**}$-modules  $\Xi:H^{**}\otimes_{\Z/l}A^{\topo}\to A^{**}$ of Lemma~\ref{lem:steen_generators}, we have the coproduct $\Delta:A^{**}\to A^{**}\otimes_{\Z/l}A^{**}$ defined as 
$H^{**}$-linear extension of   $(\Xi\otimes\Xi)\circ\Delta_{\topo}:A^{\topo}\to A^{**}\otimes_{\Z/l}A^{**}$. $\Delta$ is co-associative, and, letting $\pi:A^{**}\otimes_{\Z/l}A^{**}\to A^{**}\otimes_{H^{**}}A^{**}$ be the canonical surjection, we have $\Psi^*=\pi\circ\Delta$.}

\begin{lemma}\label{lem:MB}
The $A\bideg$-module homomorphism $\Psi^*_{M_B}:M_B\to M_B\otimes_{H\bideg}M_B$ induces a ring structure on $\Ext_{A\bideg}^{*(*,*)}(M_B,H\bideg)$. 

In addition, the isomorphism in Lemma~\ref{lem:ext_iso} is an isomorphism of graded algebras.
\end{lemma}
\begin{proof}
We note that $M_B$ is flat as an $H\bideg$-module. We have the external product
\begin{equation}\label{equ:kunneth}
\Ext_{A\bideg}^{*(*,*)}(M_B,H\bideg)\otimes_{H\bideg}\Ext_{A\bideg}^{*(*,*)}(M_B,H\bideg)\to \Ext_{A\bideg}^{*(*,*)}(M_B\otimes_{H\bideg}M_B,H\bideg).
\end{equation} 
and pullback by $\Psi^*_{M_B}:M_B\to M_B\otimes_{H\bideg}M_B$ gives the map\begin{equation}\label{equ:multi}
\Ext_{A\bideg}^{*(*,*)}(M_B\otimes_{H\bideg}M_B,H\bideg)\to 
\Ext_{A\bideg}^{*(*,*)}(M_B,H\bideg).
\end{equation}
The product on $\Ext_{A\bideg}^{*(*,*)}(M_B,H\bideg)$ we are looking for is the composite of \eqref{equ:kunneth} and \eqref{equ:multi}. The associativity follows from the coassociativity of $\Psi^*$ and $\Psi^*_{M_B}$.

The ring structure on $\Ext^{*,(*, *)}_B(\Z/l, H\bideg)$ is similarly induced from the external product (over $\Z/l$) on $\Ext$-groups and  the coproduct for $B$.

Letting $P_*\to \Z/l$ be a $B$-projective resolution of $\Z/l$, the isomorphism  in Lemma~\ref{lem:ext_iso} is induced by the isomorphism of adjunction
\[
\theta_{P_*}: \Hom_{B}(P_*, H\bideg)\xrightarrow{\sim} \Hom_{A\bideg}(A\otimes_BP_*, H\bideg)
\]
We have the following diagram
\[
\xymatrix{
\Hom_{B}(P_*, H\bideg)\otimes_{\Z/l} \Hom_{B}(P_*, H\bideg)\ar[d]^\wr_{\theta_{P_*}\otimes\theta_{P_*}}\ar[r]^\boxtimes&\Hom_B(P_*\otimes_BP_*, H\bideg)\ar[d]_\wr^{\theta_{P_*\otimes_BP_*}}\\
\Hom_{A\bideg}(A\bideg\otimes_BP_*, H\bideg)\otimes_{\Z/l} \Hom_{A\bideg}(A\bideg\otimes_BP_*, H\bideg)\ar[r]_-\boxtimes&
\Hom_{A\bideg}(A\bideg\otimes_B(P_*\otimes_BP_*), H\bideg)
}
\]
where the two products $\boxtimes$ are the maps induced by the respective external products (over $\Z/l$) and the respective coproducts $\Delta_B:B\to B\otimes_{\Z/l}B$ and $\Delta:A\bideg\to A\bideg\otimes_{\Z/l} A\bideg$. In fact,  the subalgebra $B$ is closed under $\Delta$, {and as $\Xi(Q_i^{\topo})=Q_i$,}  $\Delta_B$ is just the restriction of $\Delta$ to $B$.  This readily implies that the diagram commutes. In addition, the coproduct $\Psi^*_{M_B}$ is the map induced by $\Delta$ via the isomorphism $M_B\cong A\otimes_B\Z/l$. The lemma follows from this and the commutativity of the diagram.
\end{proof}

\subsection{The Ext-algebras}\label{subsec:Ext}

Let $H\bideg:=H\bideg(k, \ZZ/l)$ and for $X\in \SH(k)$, write $H\bideg(X)$ for $H\bideg(X,\Z/l)$. Suppose we have $X, X'\in \SH(k)$ such that the canonical map
$H\bideg(X)\otimes_{H\bideg}H\bideg(X')\to H\bideg(X\wedge X')$ is an isomorphism and with $H\bideg(X)$ or $H\bideg(X')$ flat over $H\bideg$. We have the product defined Remark~\ref{rem:ProductStructureExtGps}
\[
\Ext_{A\bideg}(H\bideg(X), H\bideg)\otimes\Ext_{A\bideg}(H\bideg(X'), H\bideg)\to
\Ext_{A\bideg}(H\bideg(X)\otimes_{H\bideg}H\bideg(X'), H\bideg);
\]
composing with the inverse of the isomorphism $H\bideg(X)\otimes_{H\bideg}H\bideg(X')\to H\bideg(X\wedge X')$ defines the product
\[
\Ext_{A\bideg}(H\bideg(X), H\bideg)\otimes\Ext_{A\bideg}(H\bideg(X'), H\bideg)\to
\Ext_{A\bideg}(H\bideg(X\wedge X'), H\bideg).
\]
If  we take $X=X'$ and $X$ is a motivic ring spectrum with multiplication $\mu_X:X\wedge X\to X$, we may compose with $(\mu_X^*)^*:\Ext_{A\bideg}(H\bideg(X\wedge X), H\bideg)\to
\Ext_{A\bideg}(H\bideg(X), H\bideg)$ to give $\Ext_{A\bideg}(H\bideg(X), H\bideg)$ the structure of an associative $\Z/l$-algebra.

In this section, we study the cases $X=\MGL, \MSL$  in detail.  We recall from Theorem~\ref{thm:MGLMSLCohHom} that for $X=\MGL, \MSL$,  $H^{**}(X)$ is a free $H^{**}$-modules and from Corollary~\ref{cor:MGLMSLCohHom}, that the natural map $H^{**}(X)\otimes_{H\bideg} H^{**}(X)\to H^{**}(X\wedge X)$ is an isomorphism. Thus, we have a natural algebra structure on $\Ext_{A\bideg}(H\bideg(X), H\bideg)$.

We say a partition $\yu=(\yu_1,\ldots,\yu_k \ldots)$ is \textit{$l$-adic} if for some $i$,  $\yu_i$ is of the form $\yu_i=l^s-1$ for some $s\ge1$. We say a partition $\yu=(\yu_1,\ldots,\yu_k,  \ldots)$ is \textit{$l$-admissible},  if for each $r\ge0$, the number of terms $i$ with $\yu_i=l^r$ is a multiple of  $l$ (possibly zero)\footnote{Note the mistranslation: ``divides'' for  ``is a multiple of'' in the English translation of \cite[Definition, pg. 29]{Nov}}. 
The set of all non-$l$-adic partitions will be denoted by $P$, and the set of all $l$-admissible, non-$l$-adic partitions by $P_l$.

We use the correspondence as described in \S\ref{subsec:MGLMSL} of partitions $I$ with symmetric functions $u_I$, with $(k)$ corresponding to $u_{(k)}=\sum_it_i^k$.  For later use, we record the following fact:

\begin{lemma}\label{lem:NovikovClass} We have an identity in $H^{**}(\MSL,\Z_l)=H^{**}(\MGL,\Z_l)/(c_1)$ of the form
\[
u_{(l^{r+1})}=l\cdot (\lambda_1\cdot u_{(l^r, \ldots, l^r)}+\sum_{i=2}^s\lambda_i\cdot u_{\omega_i}) 
\]
with $\lambda_1$ prime to $l$ and the $\omega_i$ $l$-admissible partitions of $l^{r+1}$, distinct from $(l^r, \ldots, l^r)$.
\end{lemma}

\begin{proof} Since there are only finitely many $l$-admissible partitions of $l^{r+1}$, it suffices to prove the result  in $H^{**}(\MGL,\Z/l^{n+1})/(c_1)$ for every $n$. We proceed by induction on $n$, starting with $n=1$.

Recall Kummer's theorem \cite{Kummer}: Let $l$ be a prime number,  $m_1,\ldots, m_r$ positive integers and let $m=\sum_{i=1}^rm_i$.  The maximal power of $l$ dividing $m!/\prod_{i=1}^r m_i!$ is the number of carries in performing the base $l$ addition $m_1+\cdots+m_r=m$. 
This yields the identity
\[
u_{(l^{r+1})}-(\sum_iu_i)^{l^{r+1}}=l\cdot\sum_{\alpha=(\alpha_1\cdot l^r,\ldots, \alpha_s\cdot l^r)} r_\alpha\cdot u_\alpha +l^2\cdot \sum_\beta r_\beta\cdot u_\beta
\]
with $r_\alpha, r_\beta\in \Z$ and $r_\alpha$  prime to $l$, and where  $1\le \alpha_j\le l-1$ and $\sum_j\alpha_j=l$. For $\alpha$ of this form with some  $\alpha_i\neq 1$, suppose $\alpha_j=1$ for $j=1,\ldots, m$ and $1<\alpha_j\le l-1$ for $j=m+1,\ldots, s$. Then $u_\alpha$ is not $l$-admissible, but we see from the algorithm described in Lemma~\ref{lem:division} that modulo the ideal $(\sum_i u_i, l)$, $u_\alpha$ is equivalent to a sum of  symmetric functions corresponding to $l$-admissible partitions with no occurence of $l^r$. Since Proposition~\ref{prop:Admissible} tells us that the  symmetric functions $u_\omega$ with $\omega$ $l$-admissible form a basis of the quotient ring of symmetric functions,  $\Z/l[u_1, u_2,\ldots]^{S_\infty}/(\sum_i u_i)$, this proves the lemma for $n=1$;  for $n>1$, we use  induction on  $n$ and Proposition~\ref{prop:Admissible}.
\end{proof} 

\begin{definition}\label{def:NovikovClass} For a prime number $l$, we define the virtual partition $\yu_r$   (with $\Z_l$-coeffiecients) by
\[
\yu_r=\lambda_1\cdot (l^r,\ldots, l^r)+ \sum_{i=2}^s\lambda_i\cdot {\omega_i},
\]
that is, 
\[
u_{\yu_r}=\lambda_1\cdot u_{(l^r,\ldots, l^r)}+\sum_{i=2}^s\lambda_i\cdot u_{\omega_i},
\]
following the notation in Lemma~\ref{lem:NovikovClass}.
\end{definition}

The main result of this section is the following. 
\begin{prop}
\label{prop:ext_alg_MGL}
\begin{enumerate}
\item We have $\Ext^{s,(t-s,u)}_{A\bideg}(H^{*, *}(\MGL),H\bideg)=0$, for $t>2u$. Moreover, the algebra $\bigoplus_u\bigoplus_s\Ext^{s,(2u-s,u)}_{A\bideg}(H^{*, *}(\MGL),H\bideg)$ is isomorphic to the polynomial ring over $\bbZ/l$ in the following generators:
 \begin{align*}
&\ h'_r\in \Ext^{1,(1-2l^r,1- l^r)}_{A\bideg}(H^{*, *}(\MGL),H\bideg),\ r\geq0;\\
&z_{(k)}\in\Ext^{0,(-2k,-k)}_{A\bideg}(H^{*, *}(\MGL),H\bideg),\  k\ge1, k\text{ not of the form } l^r-1,\text{ for all } r\geq0,
\end{align*}
and with unit $1\in\Ext^{0,(0,0)}_{A\bideg}(H^{*, *}(\MGL),H\bideg)$ corresponding to the augmentation $H^{*, *}(\MGL)\to H\bideg$. Finally,   $\Ext^{0, (2u-1,u)}_{A\bideg}(H^{*, *}(\MGL),H\bideg)=0$ for $u\neq1$, $\Ext^{0, (1,1)}_{A\bideg}(H^{*, *}(\MGL),H\bideg)=H^{1,1}$, and the product map
\begin{multline*}
H^{1,1}\otimes_{\Z/l}\bigoplus_{u,s}\Ext^{s,(2u-s,u)}_{A\bideg}(H^{*, *}(\MGL),H\bideg)\\=
\Ext^{0, (1,1)}_{A\bideg}(H^{*, *}(\MGL),H\bideg)\otimes_{\Z/l}\bigoplus_{u,s}\Ext^{s,(2u-s,u)}_{A\bideg}(H^{*, *}(\MGL),H\bideg)\\\to \bigoplus_{u,s}\Ext^{s,(2u-s+1,u+1)}_{A\bideg}(H^{*, *}(\MGL),H\bideg)
\end{multline*}
is surjective.
\item We have $\Ext^{s,(t-s,u)}_{A\bideg}(H^{*, *}(\MSL),H\bideg)=0$ for $t>2u$. 

Moreover, the algebra $\bigoplus_{u,s}\Ext^{s,(2u-s,u)}_{A\bideg}(H^{*, *}(\MSL),H\bideg)$ is isomorphic to the polynomial ring over $H\bideg$ in the following generators:
\begin{align*}
&h'_r\in \Ext^{1,(1-2l^r, 1-l^r)}_{A\bideg}(H^{*, *}(\MSL),H\bideg),\ r\geq0;\\
&z_{(k)}\in\Ext^{0,(-2k,-k)}_{A\bideg}(H^{*, *}(\MSL),H\bideg),\ k\ge2, k\text{ not of the form } l^r,\ l^r-1,\text{ for all } r\geq0;\\
&z_{(\yu_r)}\in\Ext^{0,(-2l^{r+1},-l^{r+1})}_{A\bideg}(H^{*, *}(\MSL),H\bideg),\ r\geq0, 
\end{align*}
and with unit $1\in\Ext^{0,(0,0)}_{A\bideg}(H^{*, *}(\MSL),H\bideg)$ corresponding to the augmentation $H^{*, *}(\MSL)\to H\bideg$. Finally, we have $\Ext^{0, (2u-1,u)}_{A\bideg}(H^{*, *}(\MSL),H\bideg)=0$ for $u\neq1$, $\Ext^{0, (1,1)}_{A\bideg}(H^{*, *}(\MSL),H\bideg)=H^{1,1}$, and the product map
\begin{multline*}
H^{1,1}\otimes_{\Z/l}\bigoplus_{u,s}\Ext^{s,(2u-s,u)}_{A\bideg}(H^{*, *}(\MSL),H\bideg)\\=
\Ext^{0, (1,1)}_{A\bideg}(H^{*, *}(\MSL),H\bideg)\otimes_{\Z/l}\bigoplus_{u,s}\Ext^{s,(2u-s,u)}_{A\bideg}(H^{*, *}(\MSL),H\bideg)\\\to \bigoplus_{u,s}\Ext^{s,(2u-s+1,u+1)}_{A\bideg}(H^{*, *}(\MSL),H\bideg)
\end{multline*}
is surjective.
\end{enumerate}
\end{prop}

The rest of this section is devoted to the proof of Proposition \ref{prop:ext_alg_MGL}, following an argument similar to \cite{Nov}. We note that (3) follows from (1) and (2) and Lemmas~\ref{lem:ext_iso}, \ref{lem:ExtAlg1} and \ref{lem:MB}.

\begin{remark} As we shall see from the proof of Proposition \ref{prop:ext_alg_MGL} below, the generator $z_{(k)}$ corresponds to the dual of the generator $u_{(k)}$ via the decomposition of $H^{*, *}(\MGL)$, resp. $H^{*, *}(\MSL)$ described in Lemma~\ref{lem:MSL_Mbeta}. Similarly, the generator $z_{\yu_r}$ is the dual of the   generator $u_{\yu_r}$. Finally, the generator $h_r$ corresponds to the generator of  $\Ext^{1,(1-2l^r, 1-l^r)}_{A\bideg}(M_B\cdot u_{(0)},H\bideg)$ via this same decomposition and using the description of the $\Ext$-group given in Lemma~\ref{lem:ExtAlg1} and Lemma~\ref{lem:ext_iso}; $u_{(0)}$ is the unit in $H^{*, *}(\MGL)$.
\end{remark}

\Omit{
We have the following isomorphism of left $A^{*, *}$ modules
\begin{equation}
 A\bideg/ A\bideg(Q_0, Q_1, \ldots )\otimes h(\mathbb{L})^{\vee} \to H^{*, *}(\MGL), 
\end{equation}
where $i: h(\mathbb{L})\subset \Z/l[b_1, b_2, \ldots]$ is a polynomial subring 
$ \Z/l[b_n\mid  n\neq l^r-1]$, and $h(\mathbb{L})^\vee\subset H^{*, *}(\MGL)$ is dual to the embedding $i$. 
}

By Theorem~\ref{thm:MGLMSLCohHom}, $H{\bideg}(\MGL)$  is the polynomial ring $H{\bideg}[c_1, c_2, \ldots]$, where $c_n\in H^{2n,n}(\MGL)$ corresponds via the Thom isomorphism to the $n$th Chern class of  the universal bundle over $\BGL_N$ for $N>>0$. Associating as usual $c_n$ to the $n$th elementary symmetric function $\sigma_n$, each partition $\yu$ gives rise to an element $u_{\yu}\in H^{*, *}(\MGL)$ of bidegree $(2|\yu|,|\yu|)$, namely, the polynomial in the $c_n$ with $\Z$-coefficients which gives  the monomial symmetric function associated to $\yu$ by substituting  $\sigma_n$ for $c_n$. 

We note that  $Q_i(u_{\yu})=0$  for all $i$. Indeed, it follows from the computation of $H{\bideg}(\MGL)$ as a polynomial ring over $H{\bideg}$ and the vanishing of motivic cohomology, $H^{a,b}(k)=0$ for $a>b$ or $b<0$ (Theorem \ref{Thm19.3}) that $H^{2*+1,*}(\MGL)=0$. Since $Q_i(u_{\yu})$ has bi-degree $(2(l^i+|\yu|)+1,l^i+|\yu|)$, we have $Q_i(u_{\yu})=0$.

This gives us the $A\bideg$--module homomorphism 
\[
\Phi_{\yu}: M_{B} u_{\yu} \to H^{*, *}(\MGL), \,\ a\cdot u_{\yu} \mapsto a(u_{\yu}).
\]

Similarly, Theorem~\ref{thm:MGLMSLCohHom} says that the canonical map $\MSL\to \MGL$ induces an isomorphism
\[
H\bideg(\MSL)\cong H\bideg(\MGL)/(c_1).
\]
{By the same degree reasoning as for $\MGL$}
% Thus, if $\yu\in P_l$ is $l$-admissible,
%%
 we have the $A\bideg$-module homomorphism
$\Phi_{\yu}: M_B u_{\yu}\to H^{*, *}(\MSL)$ (see \cite[Lemma 16]{Nov}\footnote{This result is stated without proof in \cite{Nov}; we give a proof in Appendix~\ref{App:Novikov}.}). 
As in the topological setting (see \cite[Lemma 4, Lemma 16]{Nov}, \cite[Theorem 2]{Mil2}) we have the following result.
\begin{lemma}\label{lem:MSL_Mbeta}
The maps $\Phi_{\yu}$ induce  isomorphisms of $A\bideg$-modules 
\[
H^{*, *}(\MGL)\cong\bigoplus_{\yu\in P}M_B u_{\yu}, \,\
H^{*, *}(\MSL)\cong\bigoplus_{\yu\in P_l}M_B u_{\yu}.
\]
\end{lemma}
\begin{proof}
The arguments for the two isomorphisms are similar; we first give the details for $\MSL$.

{
We have the $\Z/l$-subalgebra $A^{**}_0\subset A^{**}$ generated by the $P^i$ and $\beta$; by  Lemma~\ref{lem:steen_generators}, $\Xi:A^{\topo}\to A^{**}$ induces an isomorphism $A^{\topo}\to A^{**}_0$ and $A^{**}\cong H^{**}\otimes_{\Z/l}A^{**}_0$.  Similarly, we have the $\Z/l$-subalgebra  $M_{B0}\subset M_B$ generated by the $P^i$ and the subalgebra $H^{**}_0(\MSL)= 
\Z/l[c_1, c_2,\ldots]\subset H^{**}(\MSL)$. By reason of degree, $A^{**}_0$ acts on $H^{**}_0(\MSL)$, with $\beta$ acting trivially,  and the map $M_B u_{\yu}\to H^{*, *}(\MSL)$ restricts to a map of $A^{**}_0$-modules $M_{B0} u_{\yu}\to  H^{**}_0(\MSL)$. As left $H^{**}$-modules, we have   $H^{**}(\MSL)=H^{**}\otimes_{\Z/l}H^{**}_0(\MSL)$ and $M_B=H^{**}\otimes_{\Z/l}M_{B0}$.}

{
The map  $\Phi_{\yu}$ induces the map $\Phi_{\yu}^0:M_{B0}\cdot u_{\yu}\to H^{**}_0(\MSL)$, giving the map 
\[
\Phi^0:=\prod_{\yu\in P_l}\Phi_{\yu}^0:\bigoplus_{\yu\in P_l}M_{B0}\cdot u_{\yu}\to H^{**}_0(\MSL).
\] 
Suppose that $\Phi^0$ is an isomorphism of $\Z/l$-vector spaces. As $M_B=H^{**}\otimes_{\Z/l}M_{B0}$, $H^{**}(\MSL)=H^{**}(k)\otimes_{\Z/l}H^{**}_0(\MSL)$ and  $\Phi_{\yu}=\id_{H^{**}(k)}\otimes \Phi^0_{\yu}$, it follows that $\Phi:=\prod_{\yu\in P_l}\Phi_{\yu}:\bigoplus_{\yu\in P_l}M_B(k)\cdot u_{\yu}\to H^{**}(k)(\MSL)$ is an isomorphism of $\Z/l$-vector spaces and hence an isomorphism of $A^{**}$-modules. We proceed to show that $\Phi^0$ is an isomorphism of $\Z/l$-vector spaces.}

{
Let $H^*_{\topo}(-)$ denote mod $l$ singular cohomology $H^*(-, \Z/l)$. We have $H^*_{\topo}(\BU)=\Z/l[c_1^{\topo}, c_2^{\topo},\ldots]=H^*_{\topo}(\MU)$ with $c_n^{\topo}$ coming from the $n$th Chern class of the universal bundle $V_N\to \BU_N$ for $N>>0$; similarly
$H^*_{\topo}(\BSU)=\Z/l[c_1^{\topo}, c_2^{\topo},\ldots]/(c_1)=H^*_{\topo}(\MSU)$. We thus have the isomorphism $\rho:H^*_{\topo}(\BSU)\cong H^{**}_0(\MSL)$ sending $c_n^{\topo}\to c_n$.}

{
 We claim that $\rho$ is an isomorphisms  of left modules for $A^{\topo}\cong A^{**}_0$, in other words, 
\[
\rho(P^i_{\topo}(c_n^{\topo}))= P^i(c_n)
\]
Indeed,  the action of   $P^i$ on $H^{**}_0(\MSL)$ arises from its action on $H^{**}(\Gr(n, m+n), \Z/l)$, with $c_j$ corresponding to the product of Chern classes  $c_j(\sV_{n,m})\cdot c_n(\sV_{n,m})$. By the classical splitting principle, the value $P^i((c_j\cdot c_n)(\sV_{n,m}))$ is determined by $P^i((\sigma_j\cdot \sigma_n)(\xi_1,\ldots, \xi_n))$, where $\xi_r=c_1(L_r)$,  $L_1,\ldots, L_n$ are the tautological line bundles on the full flag variety $Fl(\sV_{n,m})$ over $\Gr(n, m+n)$, and $\sigma_j, \sigma_n$ are the  elementary symmetric functions. By the Cartan formula \cite[Proposition 9.7]{Voev03}, $P^i((\sigma_j\cdot \sigma_n)(\xi_1,\ldots, \xi_n))$ is determined by the values $P^s(\xi_r)$, $s=0, 1,\ldots$,  and by Lemma~\ref{lem:steen_generators}(6),  we have $P^0(\xi_m)=\xi_m$, $P^1(\xi_m)=\xi_m^l$ and $P^s(\xi_m)=0$ for $s\ge2$.}

{
Exactly the same argument shows that $P^i_{\topo}(c_n^{\topo})$ is determined by the classes
$P^s_{\topo}(\xi_r^{\topo})$, where  $\xi_r^{\topo}=c_1^{\topo}(L_r^{\topo})$ and $L_r^{\topo}$ is the $\CC$-line bundle on the complex flag manifold $Fl(\sV^{\CC}_{n,m})$ corresponding to $L_r$.  Again by Lemma~\ref{lem:steen_generators}(6),  we have $P^0_{\topo}(\xi^{\topo}_m)=\xi^{\topo}_m$, $P^1_{\topo}(\xi^{\topo}_m)=(\xi^{\topo}_m)^l$ and $P^s_{\topo}(\xi^{\topo}_m)=0$ for $s\ge2$. Thus, if we express $P^i_{\topo}(c_n^{\topo})$ as a polynomial in the $c_j^{\topo}= \sigma_j(\xi_1^{\topo},\ldots, \xi^{\topo}_n)$, exactly the same polynomial in the  $c_j$ will yield $P^i(c_n)$, that is, $\rho(P^i_{\topo}(c_n^{\topo}))=P^i(c_n)$, as claimed.}

{
We have the $A^{\topo}$-module $M^{\topo}_B:=A^{\topo}/A^{\topo}(Q_0^{\topo}, Q_1^{\topo},\ldots)$, the $A^{\topo}$-module $M^{\topo}_B u_{\yu}$, isomorphic to $M^{\topo}_B$, and the   $A^{\topo}$-module $\bigoplus_{\yu\in P_I}M_{B}^{\topo} u_{\yu}$. The map $\Xi$ induces the isomorphism $M^{\topo}_B\to M_{B0}$ of $A^{\topo}=A^{**}_0$-modules. Novikov defines the map $\Phi_{\yu}^{\topo}:M_{B}^{\topo} u_{\yu}\to H^{\topo}({\MU})$  exactly as our map $\Phi_{\yu}^0$ and shows \cite[Lemma 16]{Nov} that 
\[
\Phi^{\topo}:=\prod_{\yu\in P_l} \Phi_{\yu}^{\topo}:\bigoplus_{\yu\in P_l} M_{B}^{\topo} u_{\yu}\to 
H^{\topo}({\MSU})
\]
is an isomorphism of $A^{\topo}$--modules. This gives us the commutative diagram
\[
\xymatrixcolsep{40pt}
\xymatrix{
\bigoplus_{\yu\in P_l} M_{B}^{\topo} u_{\yu}\ar[d]^\wr_{\Xi} \ar[r]_\sim^{\prod_{\yu\in P_l}\Phi_{\yu}^{\topo}}&H^{\topo}({\MSU})\ar[d]^{\rho}_\wr\\
\bigoplus_{\yu\in P_l} M_{B0} u_{\yu}  \ar[r]_{\prod_{\yu\in P_l}\Phi_{\yu}^0}&H^{**}_0({\MSL)}
}
\]
which shows that $\Phi^0$ is an isomorphism of    $A^{**}_0$-modules and hence $\Phi:\bigoplus_{\yu\in P_l}M_B\cdot u_{\yu}\to H^{**}(\MSL)$ is an isomorphism of $A^{**}$-modules.}

{
The proof for $\MGL$ is the same, using \cite[Lemma 4]{Nov} or \cite[Theorem 2]{Mil2}  to handle the parallel case of $H^*(\MU)$.}
\end{proof}

\begin{remark}\label{rem:MultStruct}
$\MSL$ is a ring spectrum \cite[Theorem 4.3]{PW}. As in Lemma~\ref{lem:MB} and using Corollary~\ref{cor:MGLMSLCohHom},   the multiplication $\MSL\wedge\MSL\to \MSL$ induces a map $H^{*, *}(\MSL) \to H^{*,*}(\MSL\wedge \MSL)\cong H^{*, *}(\MSL) \otimes_{H^{*, *}} H^{*, *}(\MSL)$ of $A\bideg$-modules.
By Lemma \ref{lem:MSL_Mbeta}, this map induces the bottom map of the following diagram
\[\xymatrix@R=1.5em
{
H^{*, *}(\MSL) \ar[r]^(0.4){\Delta}& H^{*, *}(\MSL) \otimes_{H^{*, *}} H^{*, *}(\MSL)\\
\bigoplus_{\yu\in P_l}M_B u_{\yu} \ar[r]^(0.3){\Delta}\ar[u]^{\Phi}_\wr
&(\bigoplus_{\yu\in P_l}M_B u_{\yu})\otimes_{H\bideg} (\bigoplus_{\yu\in P_l}M_B u_{\yu})\ar[u]_{\Phi\otimes\Phi}^\wr, 
}\]
%%Applying the topological realization,
Comparing with the similarly defined map 
\[
\Delta_{\topo}:\bigoplus_{\yu\in P_l}M_B^{\topo} u_{\yu} 
\to (\bigoplus_{\yu\in P_l}M_B^{\topo} u_{\yu})\otimes_{\Z/l} (\bigoplus_{\yu\in P_l}M_B^{\topo} u_{\yu}) 
\]
via $\Xi$ and $\rho$, as in the proof of Lemma \ref{lem:MSL_Mbeta},
  it follows from the Cartan formula \cite[Lemma 7]{Nov} that
\begin{equation}\label{eqn:CartanYu}
\Delta(u_{\yu})=\sum_{(\yu_1, \yu_2)=\yu,\yu_1\neq\yu_2}[u_{\yu_1}\otimes u_{\yu_2}+u_{\yu_2}\otimes u_{\yu_1}]+\sum_{(\yu_1,\yu_1)=\yu}u_{\yu_1}\otimes u_{\yu_1}
\end{equation}
\end{remark}

%%modulo the augmentation ideal $\ker (H\bideg\to \bbZ/l)$.

Let $z_{\yu}\in\Ext^{0,(-2|\yu|,-|\yu|)}_{A\bideg}(H^{*, *}(\MSL),H\bideg)=\Hom^{(-2|\yu|,-|\yu|)}_{A\bideg}(H^{*, *}(\MSL),H\bideg)$ be elements such that $(z_{\yu},u_{\yu'})=\delta_{\yu,\yu'}$.  The relation \eqref{eqn:CartanYu} gives the dual relation $z_{\yu}z_{\yu'}=z_{(\yu,\yu')}$  in the Ext-algebra $\Ext^{0,(2*,*)}_{A\bideg}(H^{*, *}(\MSL),H\bideg)$.  Proposition \ref{prop:ext_alg_MGL} then follows from the Lemmas~\ref{lem:ext_iso}-\ref{lem:MSL_Mbeta}, as the $z_{(k)}$, $k\neq l^r-1$ are the indecomposable elements  in the polynomial ring  $\Z_l[\{z_{\yu}\mid \yu\in P\}]$, and the $z_{(k)}$, $k\neq l^r-1, l^r$, and $z_{\yu_r}$, $r\ge0$,  are the indecomposable elements in  $\Z_l[\{z_{\yu}\mid \yu\in P_l\}]$. Alternatively,  one can rely on the result in topology \cite[Lemma 17]{Nov} and a comparison of the Ext-algebra $Ext^{0,(2*,*)}_{A\bideg}(H^{*, *}(\MSL),H\bideg)$ with its topological counterpart, as in the proof of Lemma~\ref{lem:MSL_Mbeta}.

\section{The motivic Adams spectral sequence  {for} $\MSL$}
\label{sec:MASS}

We remind the reader that $p$ denotes the exponential characteristic of $k$ and $l$ will be an odd prime different from $p$.

\subsection{Some completions of $\MGL$}\label{subsec:completion}

Letting $X\in \SH(k)$ be a motivic spectrum, we construct a tower $C_{X}$ under $X$ in the following way.
Let $E\in \SH(k)$ be   {a motivic commutative ring spectrum.}
Let $\overline E$ be the homotopy fiber of ${\bbS_k} \to   E${, giving us the distinguished} triangle 
\begin{equation}\label{eqn:triang_E}
\overline E\to {\bbS_k}\to E.
\end{equation}
Let $\overline{E}^{s}$ be $\overline{E}\wedge\cdots\wedge\overline{E}$ ($s$-times).
Smashing  \eqref{eqn:triang_E} with $\overline E^{\wedge s}\wedge X$, we get
\begin{equation}\label{eqn:triang_Es}
\overline E^{\wedge s+1} \wedge X \to \overline E^{\wedge s} \wedge X \to 
E\wedge \overline E^{\wedge s} \wedge X. 
\end{equation}
Write $X_s:=\overline E^{\wedge s} \wedge X$, $W_s:= E\wedge \overline E^{\wedge s}$, and $W_s(X)=E\wedge X_s=W_s  \wedge X$, the above triangle becomes
\[
X_{s+1} \to  X_s \to  W_s(X). 
\]
and we have the tower and homotopy cofiber sequences
\begin{equation}\label{eqn:AdamsTower0}
 \xymatrix{
 \ldots\ar[r]&X_{s+1}\ar[d]\ar[r]&X_s\ar[r]\ar[d]&\ldots\ar[r]&X_1\ar[r]\ar[d]&X_0\ar[d]\ar@{=}[r]&X\\
 &W_{s+1}(X)&W_s(X)&&W_1(X)&W_0(X)
 }
 \end{equation}

Let $C_{s-1}(X)$ be the homotopy cofiber of $X_s\to  X_0$, $C_{-1}(X)=0$. There are induced maps $C_{s}(X) \to C_{s-1}(X)$ with fiber $W_s(X)$. One gets a tower under $X$, with homotopy fiber sequences,  of the form
\begin{equation}\label{eqn:AdamsTower}
\xymatrix{
&&W_2(X)\ar[d]&W_1(X)\ar[d]&W_0(X)\ar@{=}[d]\\
X \ar[r]& \dots  \ar[r]& C_2(X)  \ar[r]& C_1(X)  \ar[r]& C_0(X)\ar[r]&0
}
\end{equation} 
The homotopy limit of the above tower is called the {\em $E$-nilpotent completion} of $X$, denoted by $X^{\wedge}_{E}$.

We have the stable algebraic Hopf map $\eta:\Sigma^{1,1}\mS_k\to \mS_k$, induced by the unstable version $\A^2\setminus\{0\}\to \P^1$, $(x,y)\mapsto [x:y]$ and the isomorphisms $(\A^2\setminus\{0\},(0,1))\cong S^{3,2}$, $(\P^1,\infty)\cong S^{2,1}$. We let $\mS_k/\eta^n$ denote the homotopy cofiber of $\eta^n:\Sigma^{n,n}\mS_k\to \mS_k$. We have the map $\mS_k/\eta^{n+1}\to \mS_k/\eta^n$ induced by the commutative diagram
\[
\xymatrix{
\Sigma^{n+1,n+1}\mS_k\ar[r]^-{\eta^{n+1}}\ar[d]_{\eta}&\mS_k\ar@{=}[d]\\
\Sigma^{n,n}\mS_k\ar[r]_{\eta^n}&\mS_k
}
\]
Set $X/\eta^n:=\mS_k/\eta^n\wedge X$ and   let $X^\wedge_\eta$ be the $\eta$-completion of $X$, that is, the homotopy inverse limit of the tower 
\[
\cdots\to X/\eta^3\to X/\eta^2\to X/\eta.
\]

{For an integer $m$, we let $S\Z/m \in\SH(k)$ be the corresponding motivic Moore spectrum, that is, the cofiber of $\times m: \mS_k\to \mS_k$, and for $X\in \SH(k)$ we set $X/m:=S\Z/m\wedge X$. 
Let $X^\wedge_l$ be the homotopy inverse limit of the tower $\ldots\to X/l^n\to X/l^{n-1}\to\ldots$, and let $(X^*)^{\wedge}_{l}$ be the $l$-adic completion of the graded abelian group $X^*$. }

{
Let  $H\Z$ denote the motivic  Eilenberg-MacLane spectrum. The same construction as in \cite[Example 3.4]{DRO}, with $\Z_{tr}$ replaced by $\Z_{tr}/l$, shows that  $E:=H\Z/l$  has the structure of a motivic commutative ring spectrum. }
The main theorem of \S~\ref{subsec:completion} is the following. 
\begin{prop}\label{prop:compl of MGL}
We have the isomorphism $(\MGL^{*})^{\wedge}_l\cong  (\MGL^{\wedge}_{H\bbZ/l})^{*}$. 
\end{prop}
By Theorem \ref{thm:intro_MASS}, $\MGL^{\wedge}_{{H\Z/l}}$ is isomorphic to the completion $\MGL^{\wedge}_{l, \eta}$ of $\MGL$ at $l$ and $\eta$. Proposition \ref{prop:compl of MGL} {is a consequence of  } the following lemmas.
\begin{lemma}\label{lem:MGL eta}
We have $(\MGL^{\wedge}_{ \eta})^*=\MGL^*$. In particular, the map $\MSL\to \MGL$ induces a map $(\MSL^{\wedge}_{ \eta})^*\to\MGL^*$.
\end{lemma}
\begin{proof}
This follows from the fact that the multiplication by $\eta$ is a zero map on $\MGL$. To see this, we have the homotopy cofiber sequence 
\[
\A^2\setminus\{0\}\xrightarrow{\eta} \P^1\to \Th(O_{\P^1}(-1))\to \Sigma_{S^1}\A^2\setminus\{0\};
\]
giving the distinguished triangle of strongly dualizable objects in $\SH(k)$
\[
S^{3,2}\wedge \mS_k\xrightarrow{\eta} S^{2,1}\wedge \mS_k\to \Sigma^\infty_{\P^1}\Th(O_{\P^1}(-1))\to
S^{4,2}\wedge \mS_k
\]
The Thom isomorphism gives the isomorphism in $\SH(k)$
\[
\Cone(\times\eta:S^{-3,-1}\wedge \MGL\to S^{-4, -2}\wedge \MGL)=\sHom(\Th(O_{\P^1}(-1)), \MGL)\cong \Sigma^{-1}_{\P^1}\MGL\oplus \Sigma^{-2}_{\P^1}\MGL.
\]
This defines a splitting to the sequence
\[
S^{-3, -1}\wedge \MGL\xrightarrow{\times\eta}\Sigma^{-2}_{\P^1}\MGL\to \Cone(\times\eta)\to \Sigma^{-1}_{\P^1}\MGL
\]
and thus $\times\eta$ is a zero map.

The tower defining $\MGL^\wedge_\eta$ is thus a direct sum of the identity tower on $\MGL$ with the tower
\[
\cdots\to \Sigma^{n+2, n+1}\MGL\xrightarrow{\eta}\Sigma^{n+1, n}\MGL\to\cdots
\]
and since $\times\eta$ is a zero map, it follows from the Moore sequence
\[
0\to R^1\lim_\leftarrow\pi_{a+1,b}\Sigma^{n+1, n}\MGL\to \pi_{a,b}\holim \Sigma^{n+1, n}\MGL\to  \lim_\leftarrow\pi_{a,b}
\Sigma^{n+1, n}\MGL\to0
\]
that the homotopy inverse limit over this latter tower is zero, and thus the canonical map $\MGL\to \MGL^\wedge_\eta$ is an isomorphism.
\end{proof}

For an abelian group $A$,  we write $A^\wedge_l$  for $\varprojlim_n A/l^n$ and   write the $l^n$ torsion elements  in  $A$ as $_{l^n-\tors}(A)$.

\begin{lemma}\label{lem:for MGL}
We have
$(\MGL^{*})^{\wedge}_l\cong (\MGL^{\wedge}_l) ^{*}$. 
\end{lemma}
\begin{proof}
We need to show  $\varprojlim_n( (\MGL/l^n)^{*})\cong (\MGL^{*})^\wedge_l$ and $(\holim_n\MGL/l^n)^{*} \cong \varprojlim_n ((\MGL/l^n)^{*})$. We use the partial computation of $\MGL^{*,*}$ given in Theorem~\ref{thm:HMH}.

Multiplication by $l$ induces the following commutative diagram
\begin{equation}\label{eqn:diag_MGLModP}
\xymatrix @R=1em{
\cdots \ar[r] & \MGL\ar[r]^{l^n} \ar[d]^{l} & \MGL \ar[r] \ar[d]^{\id} & \MGL/l^n \ar[r]^{[1]}  \ar[d]& \Sigma^{1, 0}\MGL \ar[r]\ar[d]^l &\cdots\\
\cdots \ar[r] & \MGL \ar[r]^{l^{n-1}} & \MGL\ar[r] & \MGL/l^{n-1} \ar[r]^{[1]} 
& \Sigma^{1, 0}\MGL\ar[r] & \cdots
}
\end{equation}
Applying the functor $[\mathbb{S}, \Sigma^{s, t}-]$ to \eqref{eqn:diag_MGLModP}, we get 
\begin{equation}\label{equ:ses of mod p}
\xymatrix@R=1em{
0 \ar[r] & \MGL^{s, t}/l^n \ar[r] \ar[d]& (\MGL/l^n)^{s, t} \ar[r] \ar[d]& _{l^n-\tors}(\MGL^{s+1, t})\ar[r] \ar[d]^{l}& 0\\ 
0 \ar[r] & \MGL^{s, t}/l^{n-1} \ar[r] & (\MGL/l^{n-1})^{s, t} \ar[r]& _{l^{n-1}-\tors}(\MGL^{s+1, t})\ar[r] & 0\\ 
}
\end{equation}
Taking $\varprojlim_n$, we obtain the exact sequence
\begin{equation}\label{eqn:MGL_LES}
0\rightarrow \varprojlim_n( \MGL^{s, t}/l^n)\rightarrow 
\varprojlim_n ((\MGL/l^n)^{s, t}) \rightarrow  \varprojlim_n ({}_{l^n-\tors}(\MGL^{s+1, t}))\rightarrow \varprojlim {}^{1}(\MGL^{s+1, t}/l^n)
\end{equation}
Using the fact that $\MGL^{2t+1, t}=0$, we conclude the isomorphism 
$\varprojlim_n (\MGL^{*}/l^n) \cong \varprojlim_n ((\MGL/l^n)^{*})$. 

Taking $(s, t)=(2m-1, m)$ in \eqref{equ:ses of mod p}, the third term, the $l^n$-torsion in the Lazard ring $\MGL^{2m, m}=\Laz^m$, is zero. 
Therefore, the system $\{(\MGL/ l^n)^{2m-1, m} \}$ is isomorphic to $\{ \MGL^{2m-1, m}/l^n \}$, which is a surjective system, and in particular has the Mittag-Leffler property. Hence,   $ \varprojlim_n^1(\MGL/l^n)^{2m-1, m}=0$. Using the following short exact sequence
\begin{equation}\label{eq:MGL p-compl}
0\to \varprojlim_n{}^1((\MGL/l^n)^{s-1, t}) \to (\holim_n\MGL/l^n)^{s, t} \to \varprojlim_n ((\MGL/l^n)^{s, t})
\to 0,
\end{equation}
we conclude that $(\holim_n\MGL/l^n)^{*} \cong \varprojlim_n ((\MGL/l^n)^{*})$. This concludes the proof.
\end{proof}

\subsection{The motivic Adams spectral sequence}
\label{subsec:MASS}

In this section, we discuss the $E_2$-term and convergence properties of the spectral sequence associated to the Adams tower \eqref{eqn:AdamsTower}. We fix a prime $l$ and take $E=H\Z/l$. The Adams spectral sequence for a motivic spectrum $X$ is the spectral sequence of the tower \eqref{eqn:AdamsTower}. We use the following indexing convention:  
\begin{equation}\label{eqn:AdamsSS}
E_1^{s,t, u}:=W_s(X)^{t, u}\Rightarrow (X^\wedge_{H\Z/l})^{t, u}
\end{equation}
with $d_r^{s,t,u}:E_r^{s,t,u}\to E_r^{s+r, t+1, u}$, that is, $s$ is the filtration degree and $(t,u)$ is the cohomological bi-degree.  

We recall some facts concerning the motivic Adams spectral sequence from \cite{DI}. It is already mentioned in \cite[pg. 3845]{HKO} that their results enable the properties of the mod $l$ Adams spectral established in \cite{DI} (for $l=2$ and over a characteristic zero field) and \cite{HKO} (for arbitrary $l$ and over a characteristic zero field) to extend over an arbitrary field, but we thought it worthwhile to collect these results in a useful form below, without any claim of originality.

\begin{definition}\label{def:ASS}
1.   For $Y\in \SH(k)$, we say that $Y$ is a {\em motivically finite type wedge of  copies of $H\bbZ/l$} if for some $a,b\in \Z$, $\Sigma^{a,b}Y\cong \oplus_{\alpha\in S}\Sigma^{p_\alpha,q_\alpha}H\Z/l$ where the bi-degrees $(p_\alpha, q_\alpha)$ satisfy the following conditions: \\[3pt]
i.   $p_\alpha\ge 2q_\alpha\ge 0$ for all $\alpha\in S$\\
ii. For each $q\in \Z$, there are only finitely many $\alpha$ with $q_\alpha\le q $
\\[3pt]
2.  Let  $\angl{\bbS_k}_{H\Z/l}$ be the full localizing subcategory of $H\Z/l$-{\em cellular spectra}, that is,  $\angl{\bbS_k}_{H\Z/l}$ is the smallest full  subcategory containing $\bbS_k$ and closed under arbitrary coproducts (wedges) and the operations $H\Z/l\wedge-$, taking homotopy cofiber and suspension $\Sigma^{a,b}$, $a, b\in \Z$.
\end{definition}

\begin{remarks}\label{rem:MFT} 1. Our definition of ``a motivically finite type wedge of  copies of $H\bbZ/l$'' is a version of the hypothesis of \cite[Lemma 5.2]{HKO}, modified by making this hypothesis stable under suspension $\Sigma^{a,b}$, and differs from the definition of this term given in \cite{DI}. The conditions given in \cite{DI} will also yield the results described in the remainder of this remark, however, rather than relying on the ``elementary'' vanishing properties of motivic cohomology described in Theorem~\ref{Thm19.3}, these require in addition the vanishing of $H^{a,b}$ for $a<0, b>0$, which is the case, but requires the Block-Kato conjectures for their proof. \\
2. It follows from \cite[Lemma 5.2]{HKO} that for $Y$ a motivically finite type wedge of  copies of $H\bbZ/l$, the natural map
\[
H_{-*,-*}(Y)\to \Hom_{H^{**}(H\Z/l)}(H^{**}(Y), H^{**})
\]
is an isomorphism. Similarly, using the same proof as for Corollary~\ref{cor:MGLMSLCohHom}, we see that the canonical map
\[
H^{**}(Y)^{\otimes_{H^{**}} m}\to H^{**}(Y^{\wedge m})
\]
is an isomorphism.  \\
3. We have $A^{**}=H^{**}(H\Z/l)$ and letting $A_{-*,-*}:=\Hom_{H^{**}}(A^{**}, H^{**})$ be the dual Steenrod algebra, we have $A_{**}=H_{**}(H\Z/l)$ (see \cite[Theorem 1.1, Corollary 3.3, Proposition 5.3]{HKO}). Moreover, $H\Z/l\wedge H\Z/l$ is a motivically finite type wedge of copies of $H\bbZ/l$ (with $a=b=0$) \cite[Corollary 3.4]{HKO}. For $Y=\oplus_\alpha \Sigma^{p_\alpha, q_\alpha}H\Z/l$  a motivically finite type wedge of  copies of $H\bbZ/l$,  it follows from \cite[Lemma 5.2]{HKO} that  $H^{**}Y$ is the free bi-graded $A{\bideg}$-module $\oplus_\alpha \Sigma^{-p_\alpha, -q_\alpha}A{\bideg}$.\\
4. Take $X\in \angl{\bbS_k}_{H\Z/l}$. Then $X_s:=\overline{H\Z/l}^{\wedge s}\wedge X$ and $W_s(X):=H\Z/l\wedge X_s$ are both in $\angl{\bbS_k}_{H\Z/l}$. For $s=0$, this is clear, and in general this follows from the homotopy cofiber sequences $X_s\to X_{s-1}\to W_{s-1}$.\\
5. For $X\in  \angl{\bbS_k}_{H\Z/l}$, the canonical maps
\begin{align*}
W_s(X)_{*,*}&\to H_{**}(\overline{H\Z/l})^{\otimes_{H_{**}}s}\otimes_{H_{**}}H_{**}(X)\\
H_{**}W_s(X)&\to H_{**}(H\Z/l)\otimes_{H_{**}}H_{**}(\overline{H\Z/l})^{\otimes_{H_{**}}s}\otimes_{H_{**}}H_{**}(X)
\end{align*}
are isomorphisms. This follows from \cite[Lemma 5.4, case (1)]{HKO} with $R_{**}=H_{**}$ and $E=H\Z/l$ or $H\Z/l\wedge H\Z/l$. 
\end{remarks}

\begin{prop} [\cite{DI}, Remark~6.11, Proposition~6.14]\label{DI1} 
\label{DI2}
Assume $X$ is a $H\bbZ/l$-cellular spectrum  such that each $W_s(X)$ is a motivically finite type wedge of copies of  $H\bbZ/l$. The $E_2$-page of the mod $l$ Adams spectral sequence for $X$ is given by  
\begin{equation}\label{eqn:E2}
E_2^{s,t,u}=\Ext_{A\bideg}^{s,(t-s,u)}(H^{*, *}(X,\Z/l), H\bideg).
\end{equation}

Fix $(t,u)$. If in addition $\varprojlim_r^1 E^{s,t,u}_r(X)=0$ for each $s$, then the spectral sequence converges completely to $(X^{\wedge}_{H\Z/l})^{t,u}$. That is, let $F_s(X^{\wedge}_{H\Z/l})^{t,u}\subset (X^{\wedge}_{H\Z/l})^{t,u}$ be the kernel of $(X^{\wedge}_{H\Z/l})^{t,u}\to C_{s-1}(X)^{t,u}$. Then $(X^{\wedge}_{H\Z/l})^{t,u}\to \varprojlim_sC_s(X)^{t,u}$ is an isomorphism,   $\cap_sF_sX^{\wedge\ t,u}_{H\Z/l}=\{0\}$ and the natural map \[F_s(X^{\wedge}_{H\Z/l})^{t,u}/F_{s+1}(X^{\wedge}_{H\Z/l})^{t,u}\to E_\infty^{s,t,u}(X)\] is an isomorphism for all $s\ge0$.
\end{prop}

\begin{proof}[Proof, following \cite{DI}] 
  The $E_1$-complex is 
\[
E_1^{\bullet, *,u}:= W_0(X)^{*,u}\to \Sigma^{1,0}W_1(X)^{*,u}\to \ldots\to \Sigma^{s,0}W_s(X)^{*, u}\to \ldots
\]
with $E_1^{s,t,u}=W_s(X)^{t,u}=\Sigma^{s,0}W_s(X)^{t-s, u}$. By  Remark~\ref{rem:MFT}(5)   the natural maps
\begin{align*}
W_s(X)_{*,*}&\to H_{**}(\overline{H\Z/l})^{\otimes_{H_{**}}s}\otimes_{H_{**}}H_{**}(X)\\
H_{**}W_s(X)&\to H_{**}(H\Z/l)\otimes_{H_{**}}H_{**}(\overline{H\Z/l})^{\otimes_{H_{**}}s}\otimes_{H_{**}}H_{**}(X)
\end{align*}
are isomorphisms. Moreover, the multiplication map $H\Z/l\wedge H\Z/l\to H\Z/l$ splits the cofiber sequence $H\Z/l\wedge X_{s+1}\to H\Z/l\wedge X_s\to   H\Z/l\wedge W_s(X)$, so the complex
\begin{equation}\label{eqn:Seq1}
0\to H_{**}(X,\Z/l)\to H_{**}W_0(X)\to H_{*, *}\Sigma^{1,0}W_1(X)\to \ldots\to H_{*, *}\Sigma^{s,0}W_s(X)\to \ldots
\end{equation}
is a $H_{**}(H\Z/l)$-comodule resolution of $H_{**}(X,\Z/l)$ by extended  $H_{**}(H\Z/l)$-comodules, and is split exact as a sequence of $H_{**}$-modules; as $W_s(X)$ is a sum of copies of  suspensions of $H\bbZ/l$, these  extended  $H_{**}(H\Z/l)$-comodules are extended from free $H_{**}$-modules. This gives us the isomorphism
\[
\Ext^{s, (s-t, -u)}_{H_{**}(H\Z/l)}(H_{**}, H_{**}(X))\cong H^s(\Hom^{(s-t, -u)}_{H_{**}(H\Z/l)}(H_{**}, 
H_{*, *}\Sigma^{\bullet,0}W_\bullet)).
\]
where $\Ext_{H_{**}(H\Z/l)}(-,-)$ and $\Hom_{H_{**}(H\Z/l)}(-,-)$ are in the category of bi-graded $H_{**}(H\Z/l)$-comodules. 

Our description of $W_s(X)_{*,*}$ and $H_{**}W_s(X)$ given above imply that the natural map
\[
W_s(X)_{a,b}\to \Hom^{a,b}_{H_{**}H\Z/l}(H_{**}, 
H_{*, *}W_s(X))
\]
is an isomorphism. Comparing with the $E_1$-complex gives the isomorphism
\[
E_2^{s,t,u}=\Ext^{s, (s-t, -u)}_{H_{**}H\Z/l}(H_{**}, H_{**}(X,\Z/l)).
\]

Recall that $W_s=H\Z/l\wedge X_s$, and $X_s\to W_s$ is induced by the unit map $\mS_k\to H\Z/l$. Thus if we have  a map $x:X_s\to \Sigma^{a,b}H\Z/l$, we have the composition
\[
W_s=H\Z/l\wedge X_s\xrightarrow{\id\wedge x} \Sigma^{a,b}H\Z/l\wedge H\Z/l\xrightarrow{\mu}
\Sigma^{a,b}H\Z/l
\]
lifting $x$. The maps $H^{**}(W_s)\to H^{**}(X_s)$ are therefore surjective, which implies that the complex 
\begin{equation}\label{eqn:Seq2}
0\leftarrow H^{**}(X,\Z/l)\leftarrow H^{**}(W_0)\leftarrow\ldots\leftarrow H^{**}(\Sigma^{s,0}W_s)\leftarrow\ldots
\end{equation}
is a $H^{**}(H\Z/l)$-module resolution of $H^{**}(X,\Z/l)$ by free $H^{**}(H\Z/l)$-modules (see Remark~\ref{rem:MFT}). 

By Remark~\ref{rem:MFT}, the map $H_{-*,-*}(\Sigma^{s,0}W_s)\to \Hom_{H\bideg}(H^{**}(\Sigma^{s,0}W_s),H\bideg)$
is an isomorphism for each $s$. Thus,  applying $\Hom_{H\bideg}(-,H\bideg)$ to the subcomplex $H^{**}(\Sigma^{\bullet,0}W_\bullet)$ in \eqref{eqn:Seq2} yields the subcomplex $H_{**}(\Sigma^{\bullet,0}W_\bullet)$ in \eqref{eqn:Seq1}.  Since $\Hom^{a,b}_{H_{**}H\Z/l}(H_{**}, H_{**}H\Z/l)
=\Hom^{-a,-b}_{H^{**}H\Z/l}(H^{**}H\Z/l, H^{**})$,  this gives the isomorphism
\[
E_2^{s,t,u}\cong \Ext^{s, (s-t, -u)}_{H_{**}(H\Z/l)}(H_{**}, H_{**}(X,\Z/l))\cong 
\Ext^{s,(t-s,u)}_{H^{**}(H\Z/l)}(H^{**}(X,\Z/l), H^{**})=\Ext^{s,(t-s,u)}_{A\bideg}(H^{**}(X,\Z/l), H^{**}).
\]

The statement about convergence follows from the general convergence properties of the holim tower of a cosimplicial space; see \cite[\S 6]{Bous} and \cite[Chap. IX, \S 5.3, Lemma 5.4]{BK} for details.
\end{proof}

\subsection{Multiplicative structure} We assume as in Proposition~\ref{DI1}  that $X$ is an $H\bbZ/l$-cellular spectrum  such that each $W_s(X)$ is a motivically finite type wedge of copies of  $H\bbZ/l$. In addition, we assume that $X$ is a motivic commutative ring spectrum, with multiplication $\mu_X:X\wedge X\to X$, and we assume that the canonical map $H^{**}(X)\otimes_{H^{**}}H^{**}(X)\to H^{**}(X\wedge X)$ is an isomorphism. Finally, we assume that $H^{**}(X)$ is flat over $H\bideg$. From \S\ref{subsec:Ext}, this makes the tri-graded Ext-groups $\Ext_{A\bideg}^{*,(*,*)}(H^{**}(X), H^{**})$ into a tri-graded algebra. 

\begin{prop}\label{prop:MultASS} With the assumptions as above, the Adams spectral sequence \eqref{eqn:AdamsSS} has a multiplicative structure, compatible with the product on 
$(X^\wedge_{H\Z/l})^{*, *}$ induced by the product $\mu_X$. Moreover, the product on the $E_2$-terms
\[
E_2^{s,(t,u)}=\Ext^{s, (t-s, u)}_{A\bideg}(H^{**}(X), H^{**})
\]
induced by the spectral sequences agrees with the natural algebra structure on $\Ext_{A\bideg}^{*,(*,*)}(H^{**}(X), H^{**})$,  as defined in \S\ref{subsec:Ext}.
\end{prop}

\begin{proof} We follow the usual construction in the topological case. We first work in a somewhat more general setting, letting $X$  be a $H\bbZ/l$-cellular spectrum  such that each $W_s(X)$ is a motivically finite type wedge of copies of  $H\bbZ/l$, and $X'\in \SH(k)$ a second spectrum satisfying the same hypotheses. It follows that $X\wedge X'$ also is an $H\bbZ/l$-cellular spectrum  and  each $W_s(X\wedge X')$ is a motivically finite type wedge of copies of  $H\bbZ/l$. We will construct a pairing of spectral sequences
\[
E_r^{*,(*,*)}(X)\otimes E_r^{*,(*,*)}(X')\to E_r^{*,(*,*)}(X\wedge X')
\]
and show that the induced pairing on the $E_2$-terms agrees with the natural pairing on the $\Ext$-groups.

For this, we use the towers over $X$, $X'$ \eqref{eqn:AdamsTower0} to define the respective Adams spectral sequences for $X$ and $X'$. We replace $X_s$ with
$X_s^h:=\hocolim_{t\ge s}X_t$.
This gives us a new tower
\begin{equation}\label{eqn:AdamsTowerH}
 \xymatrix{
 \ldots\ar[r]&X^h_{s+1}\ar[d]\ar[r]&X^h_s\ar[r]\ar[d]&\ldots\ar[r]&X^h_1\ar[r]\ar[d]&X^h_0\ar[d]\ar[r]^\sim&X\\
 &W^h_{s+1}(X)&W^h_s(X)&&W^h_1(X)&W^h_0(X)
 }
 \end{equation}
which is weakly equivalent to the tower \eqref{eqn:AdamsTower0}, with each map $X^h_{s+1}\to X^h_s$ a cofibration; we do the same for $X'$. Taking the smash product of these two towers gives us the two-dimensional diagram with terms $X^h_s\wedge X^{\prime h}_t$. Defining
\[
(X\wedge X')^h_n:=\hocolim_{s+t\ge n}X^h_s\wedge X^{\prime h}_t
\]
gives us the tower of cofibrations
\begin{equation}\label{eqn:AdamsTowerProduct}
\ldots\to(X\wedge X')^h_{n+1}\to (X\wedge X')^h_n\to\ldots\to (X\wedge X')^h_0\sim X\wedge X'
\end{equation}
One computes that the cofiber of $(X\wedge X')^h_{n+1}\to (X\wedge X')^h_n$ is $\oplus_{s+t=n}W^h_s(X)\wedge W^h_t(X')$. 

Thus $\oplus_{s+t=n}W^h_s(X)\wedge W^h_t(X')$ is a motivically finite type wedge of copies of  $H\bbZ/l$, as is $(X\wedge X')^h_n$. The maps $X_s\to W_s(X)$, $X'_t\to W_t(X')$ induce surjections on $H^{**}$, hence the same holds for the map $(X\wedge X')^h_n\to \oplus_{s+t=n}W^h_s(X)\wedge W^h_t(X')$. By the arguments of \cite[6.14 Proposition]{DI}, the tower \eqref{eqn:AdamsTowerProduct} is weakly equivalent to the Adams tower \eqref{eqn:AdamsTower0} for $X\wedge X'$. 

The natural map of $X^h_s\wedge X^{\prime h}_t$ to $(X\wedge X')^h_n$ for $s+t=n$ and the inclusion $W_s(X)\wedge W_t(X')\to \cofib[(X\wedge X')^h_{n+1}\to (X\wedge X')^h_n]$  induces the pairing of rigid towers in the sense of \cite[Section 6, Appendix C]{Dugger}. By \cite[Theorem 6.1]{Dugger}, this gives us the desired pairing of spectral sequences. Note that the results of \cite{Dugger} are formulated and proven in the setting of usual spectra, but there is no essential change in checking that analogous results hold in the motivic setting.

It remains to check that the above pairing on the $E_2$-terms is the one given by the pairing of Ext-groups. The connection of the $E_2$-terms with the Ext-groups for the Adams tower \eqref{eqn:AdamsTower0}, as explained in the proof of Proposition~\ref{DI1},  follows from two facts:\\[3pt]
1. The complex
\[
0\leftarrow H^{**}(X)\leftarrow H^{**}(W_0(X))\leftarrow H^{**}(\Sigma^{1,0}W_1(X))\leftarrow \ldots\
\leftarrow H^{**}(\Sigma^sW_{s}(X))\leftarrow  \ldots
\]
induced from \eqref{eqn:AdamsTower0} is exact and $H^{**}(\Sigma^{s,0}W_s(X))$ is a free $A^{**}$-module for all $s$.\\[3pt]
2. The canonical map $\vartheta:\pi_{**}(W_s(X))\to \Hom_{A^{**}}(H^{**}(W_s(X)), H^{**})$ is an isomorphism for each $s$. 
\\[3pt]
In other words, the $E_1$-complex for the Adams tower \eqref{eqn:AdamsTower0}  is isomorphic to the complex
\[
0\to \Hom_{A^{**}}(H^{**}(W_0(X)), H^{**})\to \ldots\to \Hom_{A^{**}}(H^{**}(\Sigma^s W_s(X)), H^{**})\to\ldots
\]
whose cohomology in degree $s$ computes $\Ext^s_{A\bideg}(H^{**}(X), H^{**})$. The analogs of (1) and (2) carry over for the tower \eqref{eqn:AdamsTowerProduct}, due to the fact that $\oplus_{s+t=n}W^h_s(X)\wedge W^h_t(X')$ is a motivically finite type wedge of copies of  $H\bbZ/l$ and the map $(X\wedge X')^h_n\to \oplus_{s+t=n}W^h_s(X)\wedge W^h_t(X')$ induces a surjection on $H_{**}$. Thus the $E_1$-complex for the tower \eqref{eqn:AdamsTowerProduct} computes $\Ext_{A^{**}}(H^{**}(X\wedge X'), H^{**})$.

Moreover,  we have the commutative diagram
\[
\xymatrix{
\pi_{**}(W_s(X))\otimes \pi_{**}(W_t(X')\ar[r]\ar[d]_{\vartheta\otimes\vartheta}&\pi_{**}(W_s(X)\wedge W_t(X'))\ar[d]^\vartheta\\
\Hom_{A^{**}}(H^{**}(W_s(X)), H^{**})\otimes \Hom_{A^{**}}(H^{**}(W_t(X')), H^{**})\ar[r]&
\Hom_{A^{**}}(H^{**}(W_s(X)\wedge W_t(X')), H^{**})
}
\]
where the top row is the usual multiplication and the bottom row is the product 
\[
\Hom_{A^{**}}(H^{**}(W_s(X)), H^{**})\otimes \Hom_{A^{**}}(H^{**}(W_t(X')), H^{**})\to
\Hom_{A^{**}}(H^{**}(W_s(X))\otimes_{H^{**}}H^{**}( W_t(X')), H^{**})
\]
composed with the map induced by the isomorphism $H^{**}(W_s(X))\otimes_{H^{**}}(W_t(X')) \cong H^{**}(W_s(X)\wedge W_t(X'))$. Thus, our product structure induces a map of $E_1$-complexes
\[
E_1(X)\otimes E_1(X')\to E_1(X\wedge X')
\]
which after taking cohomology induces the the natural pairing on the $\Ext$-groups
\[
\Ext_{A^{**}}(H^{**}(X), H^{**})\otimes \Ext_{A^{**}}(H^{**}(X'), H^{**})\to 
\Ext_{A^{**}}(H^{**}(X\wedge X'), H^{**}).
\]
as defined in \S\ref{subsec:Ext}.

To conclude the proof, we use that fact that the Adams tower \eqref{eqn:AdamsTower0} is natural in $X$, so if we have a multiplication $\mu_X:X\wedge X\to X$, this together with the pairing we have just constructed in case $X'=X$ gives the multiplicative structure 
\[
E_r^{*,(*,*)}(X)\otimes E_r^{*,(*,*)}(X)\to E_r^{*,(*,*)}(X).
\]
The pairing on the $E_2$-terms, viewed as Ext-groups, is given by
\[
\Ext_{A^{**}}(H^{**}(X), H^{**})\otimes \Ext_{A^{**}}(H^{**}(X), H^{**})\to
\Ext_{A^{**}}(H^{**}(X\wedge X), H^{**})\xrightarrow{(\mu^*)^*}\Ext_{A^{**}}(H^{**}(X), H^{**})
\]
which is just the algebra structure on $\Ext_{A^{**}}(H^{**}(X), H^{**})$ as defined in \S\ref{subsec:Ext}.
\end{proof}

\subsection{The motivic Adams spectral sequence for $\MSL$}
\label{subsec:MASSMSL}

We first analyze the case $X=\mS_k$. Set $E=H\Z/l$. 

For each $s\ge1$ we have the distinguished triangle \eqref{eqn:triang_Es} in $\SH(k)$ 
\[
\overline E^{\wedge s} \to \overline E^{\wedge s-1} \to 
E\wedge \overline E^{\wedge s-1}=W_{s-1}.
\]
Smashing with $E$ gives us the distinguished triangle in $\Mod E$
\[
E\wedge \overline E^{\wedge s} \to E\wedge \overline E^{\wedge s-1} \to 
E\wedge E\wedge \overline E^{\wedge s-1};
\]
the multiplication map $E\wedge E\to E$ splits the map $E\wedge \overline E^{\wedge s-1} \to 
E\wedge E\wedge \overline E^{\wedge s-1}$ and shows that $W_s$ is a summand of $\Sigma^{-1,0}E\wedge W_{s-1}$. Inductively, $W_s$ is a summand of $\Sigma^{-s,0}E^{\wedge s+1}$. 

\begin{lemma}\label{lem:Ws} $W_s\cong \oplus_{p,q} \Sigma^{p, q}(H\Z/l)^{r_{p,q}}$, where the sum is over $(p,q)$ with  $p+s\ge 2q\ge0$ and only finitely many $r_{p,q}$ are non-zero for each $q$.
\end{lemma}

\begin{proof}  By Remark~\ref{rem:MFT}(3)
$H\Z/l\wedge H\Z/l= \oplus_{p,q} \Sigma^{p, q}(H\Z/l)^{n^{(2)}_{p,q}}$, where the sum is over $(p,q)$ with  $p\ge 2q\ge0$ and for each $q$ only finitely many $n^{(2)}_{p,q}$ are non-zero. By induction, $(H\Z/l)^{\wedge m}$ has the same description for all $m\ge2$
\[
(H\Z/l)^{\wedge m}=\oplus_{p,q} \Sigma^{p, q}(H\Z/l)^{n^{(m)}_{p,q}},\ p\ge 2q\ge0
\]
with only finitely many $n^{(m)}_{p,q}$ non-zero for each $q$.

Suppose $M$ is a summand of $(H\Z/l)^{\wedge m}$. Then there is an idempotent endomorphism $a:(H\Z/l)^{\wedge m}\to (H\Z/l)^{\wedge m}$ with $M\cong \im a$. As 
\[
\Hom_{H\Z/l\lMod}(\Sigma^{p, q}H\Z/l, \Sigma^{p', q'}H\Z/l)=
\begin{cases}
0&\text{ for }q>q'\\
0&\text{ for }q=q', p\neq p'\\
\Z/l\cdot \id&\text{ for }q=q', p=p'
\end{cases}
\]
and as each $\Sigma^{p, q}H\Z/l$ is compact, 
$a$ may be represented as a block lower triangular matrix 
\[
a=\begin{pmatrix} 
a_{0,0}&0&\cdots\\
*&a_{1,0}&0&\cdots\\
\vdots&&\ddots\\
*&\cdots&*&a_{p,q}&0\cdots\\
\vdots&&\vdots
\end{pmatrix},\ a_{p,q}\in \End_{\F_l}(\F_l^{n^{(m)}_{p,q}})
\]
with each column finite. Then each  $a_{p,q}$ is an idempotent in $\End_{\F_l}(\F_l^{n^{(m)}_{p,q}})$ and $a$ thus gives an isomorphism of $\im a$ with $\oplus_{p,q} \Sigma^{p, q}(H\Z/l)^{r_{p,q}}$, where $r_{p,q}=\text{rank}\,a_{p,q}\le n^{(m)}_{p,q}$ for all $p,q$. As $W_s$ is a summand of $\Sigma^{-s,0}E^{\wedge s+1}$, the result follows.
\end{proof}

\begin{lemma}\label{lem:Ws2} $H\Z/l\wedge \MGL\cong \oplus_{n\ge0}\Sigma^{2n, n}(H\Z/l)^{p_n}$
and
$H\Z/l\wedge \MSL\cong \oplus_{n\ge0}\Sigma^{2n, n}(H\Z/l)^{r_n}$ for suitable integers $p_n\ge r_n\ge0$.
\end{lemma}

\begin{proof} This follows from Theorem~\ref{thm:MGLMSLCohHom}, taking $\sE=H\Z/l$.
\end{proof} 

\begin{prop} \label{prop:MFTMSL} In the Adams tower \eqref{eqn:AdamsTower0} for $E=H\Z/l$ and $X=\MGL$ or $\MSL$, the cofiber $W_s(X)$  of $X_{s+1}\to X_s$ is of the form
\[
W_s(X)\cong \oplus_{p,q}\Sigma^{p,q}(H\Z/l)^{m_{p,q}}
\]
with the sum over $(p,q)$ with $p+s\ge 2q\ge0$ and only finitely many $m_{p,q}$ non-zero for each $q$. In particular, $W_s(\MGL)$ and $W_s(\MSL)$ are motivically finite type wedges of copies of $H\Z/l$.
\end{prop}

\begin{proof} This follows from Lemma~\ref{lem:Ws} and Lemma~\ref{lem:Ws2}, noting that  $W_s(X)=W_s\wedge X=W_s\wedge_{H\Z}H\Z\wedge X$.
\end{proof}

Now we apply Proposition \ref{DI1} to the case when $X$ is $\MGL$ or $\MSL$.

\begin{theorem}\label{thm:MGLMSLASS} For $X=\MGL$ or $\MSL$ and $l$ prime to the characteristic of $k$,  the mod $l$ Adams spectral sequence for $X$ is of the form
\[
E_2^{s,t,u}=\Ext_{A{\bideg}}^{s,(t-s,u)}(H^{**}(X), H^{**})\Rightarrow (X_{H\Z/l}^\wedge)^{t,u}
\]
with $d_r:E_r^{s,t,u}\to E_r^{s+r,t+1,u}$.
\end{theorem}

\begin{proof} We have already seen that the $W_s(X)$ are all motivically finite wedges of copies of $H\Z/l$ (Proposition~\ref{prop:MFTMSL}). We now show that $\MGL$ and $\MSL$ are $H\Z/l$-cellular spectra (in fact $\mS_k$-cellular). 

 Indeed, from the definition 
\[
\MSL=\colim_N\Sigma_{\P^1}^{-N}\MSL_N=\colim_N\Sigma_{\P^1}^{-N}\Th(\sE_N)=\colim_{N, m}\Sigma_{\P^1}^{-N}\P(\tilde\sE_{N, N+m}\oplus \sO_{\BSL_{N, N+m}})/\P(\tilde\sE_{N,N+m}).
\]
Here $\BSL_{N,N+m}$ is the $\G_m$-bundle over $\BGL_{N, N+m}=\Gr(N, N+m)$ corresponding to the line bundle $\det \sE_{N,N+m}$ and $\tilde\sE_{N, N+m}\to \BSL_{N,N+m}$ is the pull-back  of $\sE_{N, N+m}\to \Gr(N, N+m)$. The cellular structure on  $\Gr(N, N+m)$ given by the Schubert cell decomposition thus gives a cellular decomposition of $\BSL_{N,N+m}$ in terms of products $\G_m\times\A^r$ and this in turn induces a  cellular structure on $\Th(E_{N,N+m})$. This shows that the suspension spectrum $\Sigma_{\P^1}^\infty \Th(E_{N,N+m})$ is  $\mS_k$-cellular. As this subcategory by definition is closed under homotopy colimits, it follows that $\MSL$ is $\mS_k$-cellular and hence is $H\bb\Z/l$-cellular. The same argument shows that $\MGL$ is $H\bb\Z/l$-cellular. 
\end{proof}

\subsection{Vanishing of differentials}

By \cite[Theorem 5.11]{Hoy} and \cite[Corollary 4.9]{Spitz} , we know that $\MGL^{*}(k)[\frac{1}{p}]$ is isomorphic to the localized Lazard ring $\Laz[\frac{1}{p}]$, which is a  polynomial ring over $\bbZ[1/p]$ in generators $x_n$, $n=1, 2,\ldots$, with $\deg(x_n)=(-2n-n)$. 

Set $E_r^{t,u}(X):=\bigoplus_{s\geq 0} E_r^{s,t,u}(X)$.  
The differential $d_r$ of the motivic Adams spectral sequence sends $E_r^{t,u}(X)$ to $E_r^{t+1,u}(X)$. The restriction of $d_r$ to $E_r^{t,u}(X)$ is denoted by $d_r^{X,t,u}$, or simply by $d_r^{t,u}$ if $X$ is understood from the context.

\begin{prop}\label{diff_MGLMSL} 1. For $r\geq2$, the differentials $d_r^{\MGL,2u-1,u}$ and $d_r^{\MGL,2u,u}$ vanish; in consequence, $E_2^{2u,u}(\MGL)\cong E_\infty^{ 2u,u}(\MGL)$, and for each $u\in \Z$, the mod $l$ Adams spectral sequence for $\MGL$ converges completely to $(\MSL^\wedge_{H\Z/l})^{2u,u}\cong 
 (\MSL^\wedge_{l})^{2u,u}$.\\
 2.  For $r\geq2$, the differentials $d_r^{\MSL,2u-1,u}$ and $d_r^{\MSL,2u,u}$ vanish; in consequence, $E_2^{2u,u}(\MSL)\cong E_\infty^{ 2u,u}(\MSL)$, and for each $u\in \Z$, the mod $l$ Adams spectral sequence for $\MSL$ converges completely to $(\MSL^\wedge_{H\Z/l})^{2u,u}\cong 
 (\MSL^\wedge_{l,\eta})^{2u,u}$.
\end{prop}

\begin{proof} For (1), the isomorphism $(\MGL^\wedge_{H\Z/l})^{2u,u}\cong (\MGL^{2u,u})^\wedge_l$ follows from Proposition~\ref{prop:compl of MGL}. By Proposition~\ref{prop:ext_alg_MGL}, we have $E_2^{2u+1, u}(\MGL)=0$, so $d_r^{\MGL,2u,u}=0$. 

We use the multiplicative structure on the Adams spectral sequence given by Proposition~\ref{prop:MultASS}. From Proposition~\ref{prop:ext_alg_MGL}(1), we have $E^{1, (1,1)}_2=H^{1,1}$ and 
$E^{1+r, (2,1)}_2=0$ for all $r\ge2$.  Thus the differential leaving  $E^{1, (1,1)}_r$ is zero for all $r\ge2$. Also from Proposition~\ref{prop:ext_alg_MGL}(1), the product map
\[
E^{1, (1,1)}_2\otimes_{\Z/l} E^{s-1, (2u-2,u-1)}_2\to E^{s, (2u-1,u)}_2
\]
is surjective.  Using the multiplicative structure, the vanishing of the differentials  $d_r^{\MGL,2u,u}$ implies the vanishing of $d_r^{\MGL,2u-1,u}$ for all $u$. 

Since the differentials $d_r^{\MGL,2u,u}$ and $d_r^{\MGL,2u-1,u}$  are zero, we have
\[
E_{r+1}^{s, 2u, u}=E_{r}^{s, 2u, u},\ E_{r+1}^{s, 2u-1, u}=E_{r}^{s, 2u-1, u},
\]
for $r>s$, so $\varprojlim^1_rE_r^{s, 2u,u}(\MGL)=0=\varprojlim^1_rE_r^{s, 2u-1,u}(\MGL)$, which verifies the complete convergence.  

The proof of (2) is exactly the same as the proof of (1), using Proposition~\ref{prop:ext_alg_MGL}(2) instead of Proposition~\ref{prop:ext_alg_MGL}(1) for the vanishing of the differentials, and using Theorem \ref{thm:intro_MASS} for the identification of the $H\Z/l$-completion (we do not expect that
$(\MSL^\wedge_{l,\eta})^{2u,u}=(\MSL^\wedge_{l})^{2u,u}$ in general).
\end{proof}

\begin{lemma}\label{lem:fix1}
 For any $s\in\bbZ$, we have the following.
\begin{enumerate}
\item $(\MSL^\wedge_{\eta,l})^{2s,s}\inj (\MGL^\wedge_{l})^{2s,s}$; \label{lem:item1}
\item $ (\MSL^\wedge_{\eta,l})^{2s,s} \cong \varprojlim_n ((\MSL^{\wedge}_{\eta}/l^n)^{2s, s})$;  \label{lem:item2}
\item there is a natural injective map
$\Phi: ((\MSL_\eta^\wedge)^{2s,s})^{\wedge}_l \to (\MSL^{\wedge}_{\eta, l})^{2s, s}$.   \label{lem:item3}
\end{enumerate}

\end{lemma}
\begin{proof}
In the following commutative diagram
\[
\xymatrix @R=1.5em{
0\ar[r] &\varprojlim_n{}^1(\MSL^{\wedge}_{\eta}/l^n)^{2s-1, s} \ar[r] \ar[d]&(\varprojlim_n\MSL^{\wedge}_{\eta}/l^n)^{2s, s} \ar[r] \ar@{^{(}->}[d]&\varprojlim_n (\MSL^{\wedge}_{\eta}/l^n)^{2s, s}
\ar[r] \ar[d]&0\\
0\ar[r] &\varprojlim_n{}^1(\MGL/l^n)^{2s-1, s} \ar[r] &(\varprojlim_n\MGL/l^n)^{2s, s} \ar[r] &\varprojlim_n (\MGL/l^n)^{2s, s}
\ar[r] &0,
}\]
the middle map $(\MSL^\wedge_{\eta,l})^{2s,s}\inj (\MGL^\wedge_{l})^{2s,s}$ is injective 
by Propositions~\ref{diff_MGLMSL}  and Theorem~\ref{thm:intro_MASS}. This implies \eqref{lem:item1}.

Also, by the proof of Lemma \ref{lem:for MGL}, $\varprojlim_n{}^1(\MGL/l^n)^{2s-1, s}=0$. Using the conclusion of \eqref{lem:item1}, it implies $\varprojlim_n{}^1(\MSL^{\wedge}_{\eta}/l^n)^{2s-1, s}=0$. Therefore, we have $ (\MSL^\wedge_{\eta,l})^{2s,s} \cong \varprojlim_n ((\MSL^{\wedge}_{\eta}/l^n)^{2s, s})$. This implies \eqref{lem:item2}.

Using the diagram \eqref{equ:ses of mod p} with $\MGL$ replaced by $\MSL^{\wedge}_{\eta}$, we have the exact sequence
\begin{equation}\label{eqn:MSL_LES}
0\rightarrow \varprojlim_n ((\MSL^{\wedge}_{\eta})^{s, t}/l^n)\rightarrow 
\varprojlim_n ((\MSL^{\wedge}_{\eta}/l^n)^{s, t}) \rightarrow  \varprojlim_n ({}_{l^n-\tors}(\MSL^{\wedge}_{\eta})^{s+1, t})\rightarrow  \varprojlim {}^{1}((\MSL^{\wedge}_{\eta})^{s+1, t}/l^n)
\end{equation}
for any $s,t\in\bbZ$. 
In particular, we have an injection $\varprojlim_n((\MSL^\wedge_\eta)^{2s,s}/l^n)\inj \varprojlim_n((\MSL^\wedge_\eta/l^n)^{2s,s})$ for any $s\in\bbZ$. As we know from \eqref{lem:item2}, $\varprojlim_n ((\MSL^{\wedge}_{\eta}/l^n)^{2s, s})$ is isomorphic to  $(\MSL^\wedge_{\eta,l})^{2s,s}$. This implies  \eqref{lem:item3}. 
\end{proof}

%Without raising any confusion, we denote the composition of the two maps from Lemma~\ref{lem:fix1} $((\MSL_\eta^\wedge)^{*})^{\wedge}_l \to (\MSL^{\wedge}_{\eta, l})^{*}$ still by $\Phi$.

\subsection{The coefficient rings}

\begin{proof}[Proof of Theorem~\ref{thm:MSL}]  We first discuss the completion $(\MGL_{\eta,l}^\wedge)^*=(\MGL^*)^\wedge_l$. 

Since $\MGL^*\otimes\Z_{(l)}\cong \Z_{(l)}[x_1, x_2,\ldots]$ with $\deg x_i=-i$, $\MGL^n\otimes\Z_{(l)}$ is a finitely generated $\Z_{(l)}$-module for every $n$ and thus $(\MGL^*)^\wedge_l=\MGL^*\otimes{\ZZ_l}=\Z_l[x_1, x_2,\ldots]$.  

 Letting $Q(\MGL^*)^\wedge_l$ denote the $\Z_l$-module of indecomposables, we thus have the isomorphism of graded $Z_l$-modules
\[
Q(\MGL^*)^\wedge_l=\oplus_{n\ge1}\Z_l\cdot [x_n].
\]
We have the Newton class $c_{(n)}\in \Z[c_1, c_2,\ldots]$, corresponding to the partition $(n)$ of $n$, that is, to the symmetric function $\sum_i\xi_i^n$.   Note that $c_{(n)}(p_1^*V_1\otimes p_2^*V_2)=0$ for $V_i$  vector bundles on $X_i\in \Sm_k$, $p_i:X_1\times X_2\to X_i$ the projection, so by Theorem~\ref{thm:MSLLifitng}(4),  $c_{(n)}$ vanishes on decomposable elements in $\MGL^*\otimes{\ZZ_l}$ and 
\[
\nu_l(c_{(n)}(x_n))=\begin{cases}0&\text{ for }n\neq l^r-1\\ 1&\text{ for }n = l^r-1, r\ge1.
\end{cases}
\]
Thus a family of element $(y_n\in (\MGL^n)^\wedge_l)_{n\ge2}$ give polynomial generators of $(\MGL^*)^\wedge_l$ if and only if 
\[
\nu_l(c_{(n)}(y_n))=\begin{cases}0&\text{ for }n\neq l^r-1\\ 1&\text{ for }n = l^r-1, r\ge1.
\end{cases}
\] 

The spectral sequence filtration $F^*(\MGL^*)^\wedge_l$ has associated graded $\bigoplus_{u} E_\infty^{2u,u}(\MGL)=\bigoplus_{u} E_2^{2u,u}(\MGL)$, which is a graded $\Z/l$-algebra. Since $E_2^{0,(0,0)}(\MGL)=\Z/l$, it follows that $F^m((\MGL^0)^\wedge_l)=(l^m)\Z_l$, from which it  follows that the polynomial generators  $z_{(k)}\in E_2^{(0,(-2k,-k))}=E_\infty^{(0,(-2k,-k))}$, $k\neq l^r-1$, and $h_r'\in E_2^{(1,(1-2l^{r+1},1-l^{r+1})}=E_\infty^{(1,(1-2l^{r+1},1-l^{r+1})}$ lift to polynomial generators $\tilde{z}_{(k)}\in F^0(\MGL^{-2k, -k})_l^\wedge)$ and $\tilde{h}_r'\in F^1(\MGL^{-2(l^{r+1}-1), -(l^{r+1}-1)})_l^\wedge)$ for $(\MGL^*)^\wedge_l$ as $\Z_l$-algebra. The elements $\tilde{z}_{(k)}, \tilde{h}_r'$ are not uniquely defined, but are unique modulo decomposable elements plus elements in $l\cdot (\MGL^*)^\wedge_l$. Thus we have
$\nu_l(c_{(k)}(\tilde{z}_{(k)}))=0$ for all $k$ and $\nu_l(c_{(l^r-1)}(\tilde{h}_r))=1$.

We now consider the spectral sequence for $((\MSL^\wedge_{\eta,l})^*)$, with filtration $F^*\MSL^{\wedge *}_{\eta,l}$ on $\MSL^{\wedge *}_{\eta,l}\subset (\MGL^*)^\wedge_l$. As above, we find that the elements $z_{(k)}$, $k\neq l^i, l^i-1$,  $z_{\yu_r}$, $r\ge0$ and $h'_r$, $r\ge1$,  lift to elements $\tilde{z}'_{(k)}\in  \MSL^{\wedge k}_{\eta,l}$, $\tilde{z}_{\yu_r}\in  \MSL^{\wedge l^{r+1}}_{\eta,l}$ and $\tilde{h}'_r\in  F^1\MSL^{\wedge l^{r+1}}_{\eta,l}$ which give polynomial generators for 
$\MSL^{\wedge *}_{\eta,l}$ as $\Z_l$-algebra. 

The morphism $\MSL\to \MGL$ induces a map of spectral sequences, which is an isomorphism on the summands of the $E_2$-terms corresponding to the summands  $M_B\cdot u_{\yu}$, $\yu\in P_l$, of $H^{*, *}(\MSL)$ and $H^{*, *}(\MGL)$ described in Lemma~\ref{lem:MSL_Mbeta}. In particular, taking $\yu=(k)$, $k\neq l^i, l^i-1$, we see that the image of $\tilde{z}'_{(k)}$ in $(\MGL^*)^\wedge_l$ is equal to $\tilde{z}_{(k)}$ modulo decomposables and $l\cdot (\MGL^*)^\wedge_l$; taking $\yu=(0)$ shows that  the image of $\tilde{h}'_r$ in $(\MGL^*)^\wedge_l$ is equal to $\tilde{h}_r$ modulo decomposables and $l\cdot (\MGL^*)^\wedge_l$. Thus 
\[
\nu_l(c_{(k)}(\tilde{z}'_{(k)}))=0,\   \nu_l(c_{(l^r-1)}(\tilde{h}'_r))=1.
\]
For $\tilde{z}_{\yu_r}$,  let $c_{\yu_r}\in H^{2l^{r+1}, l^{r+1}}(\MGL,\Z_l)$ be the cohomology class corresponding to the virtual partition $\yu_r$. As the generator $z_{\yu_r}$ of the corresponding $E_2$-term is dual to $c_{\yu_r}$, it follows that $c_{\yu_r}(\tilde{z}_{\yu_r})\equiv 1\mod   l$. Since $\tilde{z}_{\yu_r}$ comes from $\MSL^{\wedge l^{r+1}}_{\eta,l}$, it follows that $c_1\cap \tilde{z}_{\yu_r}=0$,   so by definition of $\yu_r$ we have 
 \[
 c_{(l^{r+1})}(\tilde{z}_{\yu_r})=l\cdot c_{\yu_r}(\tilde{z}_{\yu_r})
\]
and thus $\nu_l(c_{(l^{r+1})}(\tilde{z}_{\yu_r}))=1$.

Since the Newton classes $c_{(n)}$ vanish on decomposables, this gives the following criterion for a family of  elements $(y_k'\in (\MSL^\wedge_{\eta,l})^{-2k, -k})_{k\ge2}$ to form a system of polynomial generators for $(\MSL^\wedge_{\eta,l})^*$, namely
\[
c_{(k)}(y_k')=\begin{cases}\lambda_k\in \Z_l^\times&\text{ for }k\neq l^r-1, l^r,\  r\ge1\\
\lambda_k'\cdot l,&\text{ with }\lambda_k' \in \Z_l^\times\text{ for }k=l^r, l^r-1,\  r\ge1.
\end{cases}
\]
This follows the chararcterization of polynomial generators for $\MSU^*$ given by Novikov  \cite[Theorem 8]{Nov}.

We now show that the map $\Phi: ((\MSL_\eta^\wedge)^{*})^{\wedge}_l \to (\MSL^{\wedge}_{\eta, l})^{*}$ of Lemma~\ref{lem:fix1} is an isomorphism; it suffices to show that $\Phi$ is surjective. 
We consider the following composition map 
\[
\xymatrix{
 (\MSL^{*})^\wedge_l\ar[r]^{\pi} &((\MSL_\eta^\wedge)^{*})^{\wedge}_l \ar[r]^{\Phi}  &(\MSL^{\wedge}_{\eta, l})^{*}.
}\]
To prove the surjectivity of $\Phi$, it suffices to show that $\Phi\circ\pi$ is surjective.

By  Theorem~\ref{thm:MSLLifitng}, there exists a class $[X,\theta]_{\MSL}$ in $\MSL^{-2n,-n}(k)$ associated to a dimension $n$ Calabi-Yau manifold $X$ together with an isomorphism $\theta:\omega_{X/k}\cong \calO_X$, with $[X,\theta]_{\MSL}$ lifting the class $[X]_{\MGL}\in \MGL^{-2d_X, -d_X}(k)$ under the projection $\MSL\to \MGL$. In particular, the image of $[X,\theta]_{\MSL}$ in $\MGL^{-2n,-n}(k)$ is the same as that given by the image of $X$ in the Lazard ring (after inverting $p=char(k)$). Thus, it suffices to show that $({\MSL}^{\wedge}_{\eta, l})^{*}\subset ({\MGL}^{\wedge}_{l})^{*}$ is generated (as $\Z_l$-module) by classes $[X]_{\MGL}$ of smooth projective Calabi-Yau manifolds $X$ defined over $k$. The necessary collection of smooth projective Calabi-Yau manifolds $X$ is furnished by Lemma~\ref{lem:Hypersurf} below.

As the image $\overline{(\MSL^\wedge_\eta)^*[1/2p]}$ of $(\MSL^\wedge_\eta)^*[1/2p]$ in $\MGL^*[1/2p]$ is degreewise a finitely generated $\Z[1/p]$ submodule, we have
\[
\overline{(\MSL^\wedge_\eta)^*[1/2p]}\otimes_{\Z[1/2p]}\Z_l\cong (\MSL^{\wedge *}_\eta)^\wedge_l\cong (\MSL^\wedge_{\eta,l})^*
\]
so we have the following criterion for a family of  elements $(y_k'\in \overline{(\MSL^\wedge_\eta)^{-2k, -k}[1/2p]})_{k\ge2}$ to form a system of polynomial generators for $\overline{(\MSL^\wedge_\eta)^*[1/2p]}$, namely
\begin{equation}\label{Chern number sn} 
c_{(k)}(y_k')=\begin{cases}\lambda_k\in \Z[1/2p]^\times&\text{ for }k\neq l^r-1, l^r, l\text{ a prime, } l\neq 2,p,  \text{ and } r>0\\
\lambda_k'\cdot l,&\text{ with }\lambda_k' \in \Z[1/2p]^\times\text{ for }k=l^r, l^r-1, l\text{ a prime, } l\neq 2,p, \text{ and } r>0.
\end{cases}
\end{equation}

\end{proof}

\begin{lemma}\label{lem:Hypersurf}
For given $n\ge2$ and this choice of $n_1,\ldots, n_r$, a  smooth hypersurface $H_n$ of $\bbP^{n_1} \times\cdots\times \bbP^{n_r}$ of multi-degree  $d_1,\dots, d_r$  satisfies Chern number condition \eqref{Chern number sn}. 
\end{lemma}

\begin{proof}

Fix $n\ge2$. For $n_1,\ldots, n_r$ positive integers, let $d_i=n_i+1$. Then each smooth hypersurface $H$ of multi-degree $d_1,\dots, d_r$ in $\bbP^{n_1} \times\cdots\times \bbP^{n_r}$ is a Calabi-Yau manifold.   

We consider the hypersurfaces constructed in 
 \cite[Page 241]{S68} for each of the three cases in \eqref{Chern number sn}. An elementary computation\footnote{One computes $c_{(n)}$ from the formula on \cite[Page 241]{S68} by taking the logarithm of the total Chern class of the tangent bundle of $H_n$, noting that the degree $n$ term in $\log(c(T_{H_n}))$ is $((-1)^n/n)\sum_i\xi_i^n$, where the $\xi_i$ are the Chern roots of $T_{H_n}$.} following the argument in \cite[\hbox{\it loc. cit.}]{S68} shows that for given $n\ge2$ and this choice of $n_1,\ldots, n_r$, a  smooth hypersurface $H_n$ of $\bbP^{n_1} \times\cdots\times \bbP^{n_r}$ of multi-degree  $d_1,\dots, d_r$  satisfies Chern number condition \eqref{Chern number sn}. 
 
For an arbitrary infinite field $k$, such a smooth $H_n$ always exists, by Bertini's theorem. Thus,  for each $n$, there is a smooth projective Calabi-Yau manifold $H_n$ defined over $k$ whose Chern numbers satisfy  \eqref{Chern number sn}. Thus the classes $[X]_{\MGL}$ of smooth projective Calabi-Yau manifolds $X$ defined over $k$ give polynomial generators (over $\Z^\wedge_l$) of ${\MSL}^{\wedge}_{\eta, l})^{*}$ and hence  $\Phi\circ\pi$ is surjective.

If $k$ is a finite field, there is for each prime $q$ (including $q=p$) a pro-$q$-power  infinite extension $L_q$ of $k$ and thus a smooth $H_{n,q}$ as above, defined over a finite extension $L_q'$ of $k$ of  degree $q^{\nu_q}$.  Taking norms from $L_q'$ down to $k$, we have the class 
$\pi_{L'_q/k*}[H_{n,q},\theta]_{\MSL}$ in $\MSL^{-2n,-n}(k)$ that maps to $q^{\nu_q}$ times a degree $n$ polynomial generator of   $({\MSL}^{\wedge}_{\eta, l})^{*}$; since $q$ was an arbitrary prime, this shows that $\Phi\circ\pi$ is surjective in this case as well.
\end{proof}

\begin{remark} Compare with  \cite[Theorem 8]{Nov}. 
\end{remark}

\section{Elliptic genus of $\MSL$-varieties}
\label{sec:EllGenus}

In this section, we use \cite{LYZ} and Theorem \ref{thm:MSL} to prove 
Theorem~\ref{thm:elliptic genus}. We recall from the introduction the ring $\tilde\Ell:=\bbZ[a_1,a_2,a_3, (1/2)a_4]$, the family of elliptic curves $\sE_{\tilde\Ell}$ over $\Spec \tilde\Ell$ with discriminant $\Delta\in \tilde\Ell$ and the localization $\Ell:=\tilde\Ell[\Delta^{-1}]$.
\subsection{Summary of previous work}
\label{subsec:summary_LYZ1}
In \cite{LYZ}, we studied the algebraic Krichever elliptic genus $\phi: \Laz\to \tilde\Ell$, and $\phi$ is given by the Baker-Akhiezer function \cite[(3.1)]{LYZ}.

\begin{theorem}[\cite{LYZ}]\label{thm: ellip_exact_intro}
Let $k$ be a perfect field of exponential characteristic $p$. The oriented cohomology theory on $\Sm_k$ in the sense of \cite{LM} sending 
\[
X\mapsto \MGL^*(X)\otimes_{\Laz}\Ell[1/2p],
\]
 is represented by a motivic oriented cohomology theory on $\Sm_k$ in the sense of \cite{PS}.
\end{theorem}
This theorem gives a well-defined notion of  Krichever's elliptic cohomology with coefficients $\ZZ[1/2p]$ of a variety $X$.

Let $\MGL^*_\QQ$ (resp. $\Ell^*_\QQ$) be the $\MGL$-cobordism theory (resp. elliptic cohomology theory) with rational coefficients.
The main focus of \cite{LYZ} is to study $\phi_\bbQ: \MGL^*_\QQ(k)\to \Ell^*_\QQ(k)$ when $k$ is an arbitrary perfect field. Recall that two smooth projective $n$-folds $X_1$ and $X_2$ are related by a flop if we have the following diagram of projective birational morphisms:
\begin{equation}
\xymatrix @R=0.5em{
& \widetilde{X}\ar[ld]\ar[rd] & \\
X_1\ar[rd]^{p_1}& & X_2\ar[ld]_{p_2}\\
& Y & 
}
\end{equation}
Here $Y$ is a singular projective $n$-fold with singular locus $Z$, such that $Z$ is smooth of dimension $n-2k+1$. We assume in addition that there exist rank $k$ vector bundles $A$ and $B$ on $Z$, such that the exceptional locus $F_1$ in $X_1$ is the $\PP^{k-1}$-bundle $\PP(A)$ over $Z$, with normal bundle $N_{F_1}X_1=B\otimes \sO_A(-1)$. Similarly, the exceptional locus $F_2$ in $X_2$ is $\PP(B)$, with normal bundle $N_{F_2}X_2=A\otimes \sO_B(-1)$.
Let $Q^3\subset \PP^4$ denote the 3-dimensional quadric with an ordinary double point $v$, defined by the equation $x_1x_2=x_3x_4$.  We say that $X_1$ and $X_2$ are related by a {\em classical} flop if in addition  $k=2$, and  along $Z$, $(Y,Z)$ is {\em Zariski} locally isomorphic to $(Q^3\times Z, v\times Z)$.

Let  $\calI_{fl}\subseteq \MGL[1/p]^*$ be the ideal generated by differences of flops.
\begin{theorem}[\cite{LYZ}]\label{thm:LYZ}
The kernel of the algebraic elliptic genus $\phi_\bbQ: \MGL^*_\QQ(k)\to \Ell^*_\QQ(k)$ is $\calI_{fl}\otimes_{\bbZ[1/p]}\bbQ$, and its image is the  polynomial ring $\bbQ[a_1,a_2,a_3,a_4]$.
\end{theorem}
In particular, $\calI_{fl}\subseteq\ker\phi$.
It is shown that the ideal $\calI_{fl}$ is also generated by the differences of classical flops.

\subsection{Proof of Theorem~\ref{thm:elliptic genus}}
Recall the restriction $\phi:\MGL[1/2p]^*\to\Ell[1/2p]$ to $\overline{\MSL}^\wedge_\eta[1/2p]^*\inj \MGL[1/2p]^*$ is denoted by $\overline{\phi}$. The ideal of $\SL$-flops $\calI^{\SL}_{fl}\subseteq \overline{\MSL}^\wedge_\eta[1/2p]^*$ is $\overline{\MSL}^\wedge_\eta[1/2p]^*\bigcap \calI_{fl}[1/2p]$. 
Then, Theorem~\ref{thm:LYZ} implies that $\calI^{\SL}_{fl}\subseteq \ker \overline{\phi}$.

Let $s_n$ be the Chern number
\[
s_n(X):=\deg_k(c_{(n)}(-T_X)
\]
where $c_{(n)}$ is as above   the $n$-th power sum polynomial in the Chern classes $c_1, \dots, c_n$; as a polynomial in the Chern roots $x_1, x_2, \dots, x_n$, $c_{(n)}:=x_1^n+x_2^n+\cdots +x_n^n$. Note that $s_n$ is an additive function on $\MGL^{-2n,-n}(k)$. For $X$ smooth and projective over $k$, we write $[X]$ for the $\MGL$-class  $[X]_{\MGL}$ as defined in \S\ref{sec:CY}.

Theorem \ref{thm:elliptic genus} is proved through the following lemmas. 

\begin{lemma}\label{prop:compute s_n}
Assume $X_1$ and $X_2$ are related by a classical flop. 
We denote the Chern roots (in $\CH^*$) of $A$ by $a_1, a_2$ and Chern roots of $B$ by $b_1, b_2$.
Then, we have
\begin{align*}
s_n([X_1-X_2])=\int_{Z}& \sum_{i_1+i_2+i_3+i_4=n-3, i_r\geq 0}
a_1^{i_1}a_2^{i_2}a_3^{i_3}a_4^{i_4} \\
&\left[ (-1)^{i_2} {n-1\choose i_1}+
(-1)^{i_1} {n-1\choose i_2}+
(-1)^{i_4+1} {n-1\choose i_3}+(-1)^{i_3+1} {n-1\choose i_4}
\right]. 
\end{align*}
\end{lemma}
\begin{proof}
By the double point relation of $\MGL$, we show in \cite[Lemmas 4.2 and 4.4]{LYZ}  that in $\MGL(k)$, 
\begin{equation}\label{X1-X2}
[X_1-X_2]=
\bbP_{\bbP(A)}(B\otimes \sO_{\bbP(A)}(-1)\oplus \sO)-
\bbP_{\bbP(B)}(A\otimes \sO_{\bbP(B)}(-1)\oplus \sO).
\end{equation}
Note that the right hand side of \eqref{X1-X2} is the difference of two systems of iterated projective bundles. 

For any $n$-dimensional vector bundle $V$   on a smooth quasi-projective variety $X$
with Chern roots $\{\lambda_i\}$,  let $\pi: \PP_X(V) \to X$ be the corresponding projective bundle. Take $f(t) \in \CH^*(X)[\![t]\!]$.  By \cite[Theorem 5.35, Lemma 5.36]{Vishik}, we have
\begin{equation}\label{Quill-formula}
\pi_*(f(c_1(\sO(1))))=
\sum_i\frac{f(-\lambda_i)}{\prod_{j\neq i}(\lambda_j-\lambda_i)}. 
\end{equation}

A direct computation using  \eqref{X1-X2} and  the formula \eqref{Quill-formula} shows the desired formula. 
\end{proof}

\begin{lemma}\label{lem:step2}
The polynomial generators (over $\Z[1/2p]$) of $\overline{\MSL}^{\wedge}_{\eta}[1/2p]$ of degrees greater than 4 lie in the ideal $\calI_{fl}$. 
\end{lemma}
\begin{proof}
This follows directly from Theorem \ref{thm:MSL} (2), Lemma~ \ref{prop:compute s_n}, and \cite[Lemma 6.2]{Totaro}. 
\end{proof}

\begin{lemma}
The polynomial generators of $\overline{\MSL}^{\wedge}_{\eta}[1/2p]^*$ with degree 2, 3, 4 have algebraically independent images  under the map $\phi$. 
\end{lemma}
\begin{proof}
As in \cite{H}, (see also \cite[Proposition 5.2]{LYZ}), we have the following generators $W_i$ of degree $i$ of $\MGL^*$,   characterized by their Chern numbers
\begin{align*}
&c_1^2[W_2]=0, c_2[W_2]=24;\\
&c_1^3[W_3]=0, c_1c_2[W_3]=0, c_3[W_3]=2;
\\
&c_1^4[W_4]=0, c_1^2c_2[W_4]=0, c_2^2[W_4]=2, c_1c_3[W_4]=0, c_4[W_4]=6.
\end{align*}
In particular, by Theorem~\ref{thm:MSL} (2), we know that $W_i$, $i=2,3,4$, are polynomial generators (over $\Z[1/2p]$) of $\overline{\MSL}^\wedge_\eta[1/2p]^*$ in the corresponding degrees. 

In \cite[\S~5]{LYZ}, we calculated that $\phi(W_2)=24a_2$, $\phi(W_3)=a_3$, and $\phi(W_4)= 6a_2^2-a_4$. Hence, they are algebraically independent. 
\end{proof}
This finishes the proof of Theorem~\ref{thm:elliptic genus}.

We conclude the main body of the paper with the following comment.
\begin{remark}
The proof of $\SU$-rigidity property in \cite{Kr} only uses properties of the Baker-Akhiezer function, hence this argument  also applies to the algebraic genus $\phi:\MGL[1/2p]^*\to\Ell[1/2p]$ studied here. 
Therefore, Theorem~\cite[Theorem~B]{LYZ} implies that any genus factoring through $\calI_{fl}$ is $\SL$-rigid. However, the argument in \cite{Totaro} shows that the converse is also true in the topological setting, although at the present we do not know this in the algebraic setting. 
\end{remark}

\appendix
\section{Convergence of the motivic Adams spectral sequence}
\label{app}
In the topological setting, in \cite{Bous}, Bousfield proved that the 
nilpotent completion of any connective spectrum at the Eilenberg-MacLane spectrum of $\bbZ/l$ is isomorphic to the Bousfield localization  at the Moore spectrum. The goal of this appendix is to study the motivic analogues of Bousfield's result. 

Let  $k$ be an arbitrary perfect field and $\SH(k)$ the stable motivic homotopy category over $k$. For a prime number  $l$  and a motivic spectrum $Y$, $Y^\wedge_l$ is the completion of $Y$ at $l$ as in \S~\ref{subsec:completion}. We set $E:=H\Z/l$, giving us 
$Y^{\wedge}_{E}$, the nilpotent completion of $Y$ at $E$, i.e., the homotopy inverse limit of the Adams tower \eqref{eqn:AdamsTower}   for $H\Z/l$. 

We recall some facts about Voevodsky's slice tower; for details we refer the reader to \cite{RSO, Voev_open}. For each $q\in \Z$, let  $\Sigma^{2q, q}\SH^{\eff}(k)\subset \SH(k)$ be the localizing subcategory generated by the spectra $\Sigma^{2q, q}\Sigma^\infty_TX_+$ for $X\in \Sm_k$. The inclusion $i_q: \Sigma^{2q, q}\SH^{\eff}(k)\subset \SH(k)$ admits the right adjoint $r_q:\SH(k)\to \Sigma^{2q, q}\SH^{\eff}(k)$ and the endofunctor $f_q:\SH(k)\to \SH(k)$ is defined to be the composition $i_q\circ r_q$. For each $M\in \SH(k)$, the co-unit $\eta_M:f_qM\to M$ is thus universal for maps  $N\to M$ in $\SH(k)$, $N$ in $\Sigma^{2q, q}\SH^{\eff}(k)$. The $f_q$ fit together to form the {\em slice tower}
\[
\ldots\to f_{q+1}\to f_q\to\ldots\to \id_{\SH(k)}. 
\]
Moreover, $f_nf_m=f_{\max(m,n)}$.
The $q$th slice functor $s_q:\SH(k)\to \SH(k)$ is characterized as fitting into a natural distinguished triangle
\[
f_{q+1}\to f_q\to s_q\to \Sigma^{1,0}f_{q+1}\ .
\]

Let  $\Sigma^{2q, q}\SH^{\eff}(k)^\perp\subset \SH(k)$ denote the right perpendicular to $ \Sigma^{2q, q}\SH^{\eff}(k)$, that is the full subcategory of objects $N$ with $[M, N]_{\SH(k)}=0$ for all $M\in \Sigma^{2q, q}\SH^{\eff}(k)$. The inclusion $i^q:\Sigma^{2q, q}\SH^{\eff}(k)^\perp\to \SH(k)$ admits the left adjoint $\ell^q:\SH(k)\to \Sigma^{2q, q}\SH^{\eff}(k)^\perp$ and the functor $f^{q-1}:\SH(k)\to \SH(k)$ is defined as the composition $i^q\circ\ell^q$. There is a natural distinguished triangle
\[
f_q\xrightarrow{\eta} \id_{\SH(k)}\xrightarrow{u} f^{q-1}\to \Sigma^{1,0}f_q
\]
with $u:\id_{\SH(k)}\to f^{q-1}$ the unit of the adjunction. For $M\in \SH(k)$, the map $u_M:M\to f^{q-1}M$ is universal for maps $M\to N$ with $N$ in $\Sigma^{2q, q}\SH^{\eff}(k)^\perp$. The relation $f_nf_m=f_{\max(m,n)}$ translates to the relation $f^nf^m=f^{\min(m,n)}$ and we have the distinguished triangle
\[
s_q\to f^q\to f^{q-1}\to \Sigma^{1,0}s_q.
\]

A spectrum $Y\in \SH(k)$ is {\em slice connective} if $f_NY\to Y$ is an isomorphism for some $N$; in other words, $Y\in \SH(k)$ is slice connective if and only if $Y$ is in $\Sigma^N_{\P^1}\SH^{\eff}(k)$ for some $N$. If this is the case, then $f_nY\to Y$ is an isomorphism for all $n\le N$, since $f_nF_N=f_{\max(n,N)}=f_N$. Equivalently, $f^nM=0$ for $n\le N-1$. 

\begin{remarks}\label{rem:Slice} 1. As $f_n\circ f_{n+1}=f_{n+1}f_n=f_{\max(n,n+1)}=f_{n+1}$, and similarly $f_{n+1}f_{n+1}=f_{n+1}$,  applying $f_n$ to the distinguished triangle $f_{n+1}\to f_n\to s_n\to \Sigma^{1,0}f_{n+1}$ shows that  $f_{n+1}s_n=0$ and $f_ns_n=s_ns_n=s_n$. Similarly, $s_ms_n=0$ for $m\neq n$. \\[2pt]
2. By \cite{Voev0Slice} in characteristic zero and \cite[\S 10]{LevineHCT} in arbitrary characteristic, $s_0(\mS_k)\cong H\Z$. Thus $f_0H\Z=s_0H\Z=H\Z$.
\end{remarks}

Following \cite{RSO}, the slice completion endofunctor on $\SH(k)$ $sc$ is defined as
 $sc:=\holim_q f^{q-1}$. 

We say a set of bi-degrees
$\{(p_i, q_i)\}_{i\in I}$  satisfies condition \eqref{eqn:fin} if
\begin{equation}\label{eqn:fin}\tag{Fin}
\hbox{ there exists some $s\in \Z$, such that 
$p_i-2q_i\geq s$, for all $i\in I$.}
\end{equation}
We say a spectrum $Y\in \SH(k)$   satisfies condition \eqref{eqn:fin}, if $H\Z/l\wedge Y=\oplus_{i\in I} \Sigma^{(p_i, q_i)}H\Z/l$, where the index set $\{(p_i, q_i)\}_{i\in I}$ satisfies condition \eqref{eqn:fin}. 

We will need one more finiteness condition, taken from \cite{RSO}. To describe this, fix a pair of integers $(p,q)$, let $D^{p, q}\bbS_k$ be the simplicial mapping cylinder of $\Sigma^{p-1,q}\bbS_k\to 0$ and let $\alpha_{p,q}:\Sigma^{p-1,q}\bbS_k\to D^{p, q}\bbS_k$ be the canonical map.  Roughly speaking, a  motivic spectrum $X$ has  a  {\em cell presentation of finite type} if $X$ is equivalent to the colimit $X_\infty$ of a sequence 
\[
0=X_0\to X_1\to \ldots \to X_{i-1}\to X_i\to \ldots 
\]
with each $X_{i-1}\to X_i$  fitting into a co-cartesian diagram of the form
\[
\xymatrix{
\bigvee_j\Sigma^{p_{ij}-1, q_{ij}}\bbS_k\ar[r]\ar[d]_{\bigvee_j\alpha_{p_{ij}, q_{ij}}}&X_{i-1}\ar[d]\\
\bigvee_jD^{p_{ij}, q_{ij}}\bbS_k\ar[r]&X_i,
}
\]
such that\\[5pt]
i. there is an integer $k$ such that $p_{ij}-q_{ij}\ge k$ for all $i,j$,\\
ii. for each integer $n$, there are only finitely many indices $i,j$ with $p_{ij}-q_{ij}=n$. 
\\[5pt]
For details, we refer the reader to \cite[\S~3.3]{RSO}.

We have the following motivic analogue of the Bousfield isomorphism, which is a more precise version of  Theorem~\ref{thm:intro_MASS}.
\begin{theorem}\label{thm:app}\label{thm:eta completion} Let $l$ be an odd prime.
\begin{enumerate}
\item
Let $Y\in \SH(k)$ be a  slice connective  motivic spectrum and let $Y^{\wedge}_{l, \eta}$ be the completion of $Y$ at $l, \eta$. 
Then we have a weak equivalence
\[
sc(Y^{\wedge}_{l})\cong sc(Y^{\wedge}_{H\Z/l}).
\]
Moreover, if $Y$ has a cell presentation of finite type, then $sc(Y^{\wedge}_{H\Z/l})\cong Y^{\wedge}_{l, \eta}$. 
\item
Let $Y\in \SH$ be a motivic spectrum satisfying condition \eqref{eqn:fin}, 
then $Y^{\wedge}_{H\Z/l}$ is slice complete. 
Moreover, if  $Y$ has a cell presentation of finite type, 
then there is a weak equivalence $Y^{\wedge}_{H\Z/l}\cong Y^{\wedge}_{l, \eta}$. 
\end{enumerate}
\end{theorem}

\begin{remark}\begin{enumerate}
\item Mantovani \cite[Theorems 1.0.1, 1.0.3]{Mantovani} has proven similar comparison properties under a different connectivity hypotheses. In the special case of $E=H\Z/l$, $l\neq \Char k$, his results show the following: Suppose that $Y\in \SH(k)$ is connective, that is, there is an $n_0$ such that $\pi_n(Y)_*=0$ for $n<n_0$. Then $Y^\wedge_{H\Z/l}\cong Y^\wedge_{l,\eta}$.  
\item If the field $k$ has characteristic $0$, similar convergence properties of the motivic Adams spectral sequence have been studied by Hu-Kriz-Ormsby in  \cite{HKO}.  The proof in the present paper has no restriction on the characteristic of the field,   except that we assume $l$ is different from the characteristic.
\item Theorem~\ref{thm:app} is false without any finiteness assumption. For an $H\Z/l$-module $X$, we have $X^{\wedge}_{H\Z/l}\cong X$ (see e.g., \cite[Remark~6.9]{DI}). Without condition \eqref{eqn:fin}, there are examples of non-zero $H\Z/l$-modules whose slice completion is zero, one example being the \'etale cohomology spectrum $H^{\acute{e}t}\Z/l$. 
\item Both $\MGL$ and  $\MSL$ have  cell presentations of finite type and satisfy condition \eqref{eqn:fin}. Indeed, \cite[Proposition 3.31]{RSO} says that $\MGL$ has a  cell presentations of finite type. The proof goes by noting that the cell decomposition of $\BGL_n$ discussed here in the proof of Theorem~\ref{thm:MGLMSLASS} gives the suspension spectrum $\Sigma^\infty_{\P^1}\MGL_n$ a cell presentation of finite type with cells of type $(2i,i)$, then applying  \cite[Lemma 3.35]{RSO} shows that $\MGL$ has a cell presentation of finite type. Using the description of $\BSL_n$ as a $\G_m$-bundle over $\BGL_n$ shows as in the proof of Theorem~\ref{thm:MGLMSLASS} that $\Sigma^\infty_{\P^1}\MSL_n$ has a  cell presentation of finite type
with cells of type $(2i,i)$, $(2i-1,i)$, and then the same argument as above shows that $\MSL$ has a cell presentation of finite type. The condition \eqref{eqn:fin} for $\MGL$ and $\MSL$ follows from  Theorem~\ref{thm:MGLMSLCohHom} and Theorem~\ref{Thm19.3}.
\end{enumerate}
\end{remark}

\subsection{Proof of Theorem \ref{thm:eta completion} (1)}
In this subsection, we follow the approach of \cite{Bous} to prove Theorem \ref{thm:eta completion} (1). 

We adapt the same notations as in \cite{Bous}. As above, we set $E:=H\Z/l$.
For each $s\geq 0$, let $\overline{E}^s$ be as in \S~\ref{subsec:completion}; Define $\overline{E}_{s-1}$ by the triangle $
\overline{E}^s\to  {\mS_k}\to \overline{E}_{s-1}\to \Sigma^{1, 0} \overline{E}^s$; in particular, $\bar{E}_0=E$. As in  \cite[(5.1)]{Bous}, we have the distinguished triangle
\begin{equation}\label{(5.1)B}
E\wedge \overline{E}^s \to \overline{E}_s \to \overline{E}_{s-1} \to \Sigma^{1, 0}(E\wedge \overline{E}^s). 
\end{equation}
For a spectrum $Y\in \SH(k)$, the tower  $Y\to \{\overline{E}_s\wedge Y \}$ under $Y$ has homotopy inverse limit $Y^{\wedge}_{E}$. 

We say that $W\in \SH(k)$ is a {\em finite extension of $E$-modules} if there is a tower
\[
W=W_r\to W_{r-1}\to \ldots\to W_1\to W_0=0
\]
in $\SH(k)$, such that the homotopy fiber $\bar{W}_i$ of $W_i\to W_{i-1}$ admits the structure of an $E$-module, for each $i=1,\ldots, r$. We let $M(E)$ be the collection of  finite extensions $W$ of $E$-modules,  such that $W=f^n W$ for all large enough $n\in \N$. An $N\in M(E)$ is called an {\em $E$-nilpotent} object of $\SH(k)$. It is easy to see that $M(E)$ forms a full triangulated subcategory of $\SH(k)$.
 
\begin{definition}
An $E$-nilpotent resolution of $Y$ is a tower $Y\to \{W_s\}$ under $Y$, such that 
\begin{enumerate}
\item {$W_s$ is $E$-nilpotent for all $s$}. 
\item For each $E$-nilpotent $W$, the canonical map $\colim_{s}[W_s, W]\to [Y, W]$ is an isomorphism.  
\end{enumerate}
\end{definition}

{For $M, N\in \SH(k)$ and $t\in \Z$,  we write $[M, N]_t$ for $[\Sigma^{t,0}M, N]_{\SH(k)}$.}

\begin{lemma}\label{lem1:E-nilp res}
The tower $\{f^s(\overline{E}_s \wedge Y)\}$ is an $E$-nilpotent resolution of $Y$. 
\end{lemma}
\begin{proof}  We note that $\overline{E}_s \wedge Y$ is a finite extension of $E$-modules by induction on $s\ge0$. Indeed, 
$E\wedge \overline{E}^s$ is an $E$-module, and $\overline{E}_{s-1}$ is a finite extension of $E$-modules by the induction hypothesis,  { starting with  $\overline{E}_0=E$}. 
Hence, $\overline{E}_s \wedge Y$ is a finite extension of $E$-modules using \eqref{(5.1)B}. 

{As noted above, we have $f^n\circ f^s=f^s$ for all $n\ge s$.
We claim that if $M$ is an $E$-module, then so is $f_n M$. Indeed, by  \cite{GRSO}, the $E$-module structure on $M$ induces an $f_0E$-module structure on $f_nM$ for every $n$. As $f_0E\to E$ is an isomorphism (Remark~\ref{rem:Slice}(2)), $f_nM$ is thus an $E$-module. Next, as $f_n$ is exact, if $M$ is a finite extension of $E$-modules, then so is $f_nM$  for every $n$. Therefore, $f^nM:=\text{cofiber}(f_{n+1}M\to M)$ is also a finite extension of $E$-modules  for every $n$. Taking $M=\overline{E}_s \wedge Y$, we conclude that $f^s(\overline{E}_s \wedge Y)$ is a finite extension of $E$-modules and therefore $f^s(\overline{E}_s \wedge Y)$ is in $M(E)$. }

{
For $N\in M(E)$, we have $f^nN=N$, for all large enough $n\in \N$. Therefore
\[\colim_s[f^s(\overline{E}_{s}\wedge Y), N ]\cong
\colim_s[f^s(\overline{E}_{s}\wedge Y), f^sN ]
\cong\colim_s[(\overline{E}_{s}\wedge Y), f^s N]\cong\colim_s[(\overline{E}_{s}\wedge Y), N],
\] 
the second isomorphism coming from the universal property of $f^s$. 
 To complete the proof of the lemma, it suffices to show that the natural map $\colim_s[\overline{E}_{s}\wedge Y, N ]\to [Y, N]$ is an isomorphism for all $N \in M(E)$. }
We have the following diagram {({\it cf.} \cite[(5.2)]{Bous})}, 
\[
\xymatrix @R=1em 
{
\overline{E}^{s+1} \ar[r] \ar[d]& {\mS_k} \ar[r]\ar@{=}[d]& \overline{E}_{s} \ar[r] \ar[d]& \Sigma^{1,0} \overline{E}^{s+1} \ar[d]\\
\overline{E}^{s} \ar[r]&  {\mS_k} \ar[r]& \overline{E}_{s-1} \ar[r] & \Sigma^{1,0} \overline{E}^{s} 
}\]
Applying $[-\wedge Y, N]_*$ and taking colimit, we get a long exact sequence
\[
\cdots \to \colim_{s}[ \overline{E}^{s+1}\wedge Y, N ]_{*+1} \to 
 \colim_{s}[ \overline{E}_{s}\wedge Y, N ]_* \to  [Y, N]_*\to 
  \colim_{s}[ \overline{E}^{s+1}\wedge Y, N ]_*\to \cdots. 
\]
It suffices to show that $ \colim_{s}[ \overline{E}^{s+1}\wedge Y, N ]_* =0$.
%%These are not equivalent statements!
%%, or equivalently,  the map $\pi_s^*:[\overline{E}^{s}\wedge Y, \Sigma^{*,0}N] \to [\overline{E}^{s+1}\wedge Y, \Sigma^{*,0}N]$ is the zero map. 
%%

{We consider the map  $\pi_s^*:[\overline{E}^{s}\wedge Y, N]_* \to [\overline{E}^{s+1}\wedge Y, N]_*$, which fits into the following long exact sequence 
\[
\cdots \to [E\wedge \overline{E}^{s} \wedge Y, N ]_*
\xrightarrow{u_s^*}  [\overline{E}^{s} \wedge Y, N ]_* \xrightarrow{\pi_{s,N}^*}  [\overline{E}^{s+1} \wedge Y, N ]_*\to \cdots,
\]
arising from the sequence $\overline{E}^{s+1}\xrightarrow{\pi_s} \overline{E}^s\xrightarrow{u_s}E\wedge \overline{E}^{s}$.}

{If $N$ is an $E$-module, then given $\phi: \overline{E}^{s} \wedge Y \to N$, we can take the $E$-linear extension $ \tilde{\phi}: E\wedge \overline{E}^{s} \wedge Y \to N$. As  
$\phi=\tilde{\phi}\circ u_s^*$, this   shows that $\pi_{s,N}^*=0$ when $N$ is an $E$-module.}

{For general $N\in M(E)$, $N$ is a finite extension of $E$-modules. In other words, there is a tower
\[
N=N_r\to N_{r-1}\to \ldots\to N_1\to N_0=0
\]
such that  the fiber $\bar{N}_i$ of each map $N_i\to N_{i-1}$ is an $E$-module. By induction on $r$, we may assume that the composition $\pi^*_{s+{r-1},N_{r-1}}\circ\ldots\circ\pi^*_{s, N_{r-1}}=0$, which implies that the image of $\pi^*_{s+{r-1},N}\circ\ldots\circ\pi^*_{s, N}$ factors through 
$[\overline{E}^{s+r-1} \wedge Y,\bar{N}]_*$, and thus $\pi^*_{s+r,N}\circ\ldots\circ\pi^*_{s, N}=0$, and thus  $ \colim_{s}[ \overline{E}^s\wedge Y, N ]_*=0$.   }
\end{proof}

Similar to Lemma \ref{lem1:E-nilp res}, we have the following. 
\begin{lemma}\label{lem2:E-nilp res}
{Let $Y$ be a slice connective spectrum}. Then
the tower $Y\to \{f^s(S\Z/l^s \wedge Y)\}$ is an $E$-nilpotent resolution of $Y$. 
\end{lemma}
\begin{proof}  We first show that $f^m(S\Z/l^s\wedge Y)$ is in $M(E)$ for all $m\in \Z$ and $s\ge1$.   Since $f^m(f^m(S\Z/l^s\wedge Y))=f^m(S\Z/l^s\wedge Y)$, it suffices to show that $f^m(S\Z/l^s\wedge Y)$ a finite extension of $E$-modules.   Using the distinguished triangle
\[
f^m(S\Z/l^{s-1}\wedge Y)\to f^m(S\Z/l^s\wedge Y)\to f^m(S\Z/l\wedge Y)\to
\Sigma^{1,0}f^m(S\Z/l^{s-1}\wedge Y)
\]
we reduce to the case $s=1$.

By assumption, $Y$ is slice connective, so there is a $c\in \Z$ with $f^nY=0$ for $n<c$.  As $f_n$ is an exact functor,  $f^n(S\Z/l\wedge Y)=0$  for $n<c$. In particular $f^{c-1}(S\Z/l\wedge Y)=0$ is a finite extension of $E$-modules. Using the distinguished triangle
\[
s_n(S\Z/l\wedge Y)\to f^n(S\Z/l\wedge Y)\to f^{n-1}(S\Z/l \wedge Y)\to
\Sigma^{1,0} s_n(S\Z/l\wedge Y)
\]
and induction in $n$, starting with $n=c-1$,  we reduce to showing that $s_n(S\Z/l\wedge Y)$ is  an $E$-module for every $n$. For this, 
we have noted in Remark~\ref{rem:Slice} that $s_0(\mS_k)\cong H\Z$. Letting $X$ be an arbitrary object of $\SH(k)$, $X$ has a canonical structure of an $\bbS_k$-module. By \cite{GRSO, P11}  $s_n(X)$ is an $s_0(\mS_k)\cong H\Z$-module  and thus
\[
s_n(S\Z/l\wedge Y)\cong S\Z/l\wedge s_n(Y)  \cong  H\Z/l \wedge_{H\Z}s_n(Y)
\]
is a module over $E=H\Z/l$. 

Take $N\in M(E)$. By definition there is a $c(N)\in \N$ such that  $N=f^nN$, for all $n>c(N)$. 
Therefore, $[f^s(S\Z/l^s\wedge Y), N]=[S\Z/l^s\wedge Y, N]$ for $s>c(N)$, by the universal property of $f^s$. This implies that the natural map  $\colim_{s}[f^s(S\Z/l^s\wedge Y), N]\to \colim_s[S\Z/l^s\wedge Y, N]$ is an isomorphism. To complete the proof, it suffices to verify that the natural map $\colim_s[S\Z/l^s\wedge Y , N ]\to [Y, N]$ is an isomorphism.   
We have the diagram
\[
\xymatrix @R=1.2em
{
Y\ar[r]^{\times  l^n} \ar[d]_{\times l} &Y\ar[r] \ar@{=}[d] &S\Z/l^n\wedge Y \ar[d]\\
Y\ar[r]_{\times  l^{n-1}} &Y\ar[r] &S\Z/l^{n-1}\wedge Y
}\]
Applying the functor $[-, N]$, we get
\[\xymatrix@R=1.2em{
\cdots \ar[r]&[Y, N]_{*+1} \ar[r]& [S\Z/l^n\wedge Y, N]_* \ar[r]& [Y, N]_* \ar[r]^{\times l^n} & [Y, N]_* \ar[r]& \cdots\\
\cdots \ar[r]&[Y, N]_{*+1} \ar[r] \ar[u]^{\times l}& [S\Z/l^{n-1}\wedge Y, N]_* \ar[r]\ar[u]
& [Y, N]_* \ar[r]_{\times l^{n-1}} \ar@{=}[u]& [Y, N]_* \ar[r] \ar[u]_{\times l} & \cdots
}
\]
Taking the colimit of the above system, we get a long exact sequence. In order to show that $\colim_s[S\Z/l^s\wedge Y , N ]\to [Y, N]$ is an isomorphism, it suffices to show the colimit of the system $\{[Y, N]_*,\times l\}$ is zero. This follows from the fact that for all  $N\in M(E)$,  the multiplication map $\times l: N\to N$ is  a nilpotent endomorphism. 
\end{proof}

Essentially the same argument as \cite[5.9, 5.10. 5.11]{Bous} shows the following. 
\begin{lemma}\label{lem3:E-nilp res} Take $Y\in \SH(k)$ and let $Y\to \{W_s\}_s$, $Y\to  \{W_r'\}_r$ be $E$-nilpotent resolutions. Then there is a canonical isomorphism $\holim_s\{W_s\}_s\cong \holim_r\{W_r'\}_r$ in $\SH(k)$. In particular, if $Y$ is slice connective, then there is a canonical isomorphism $\holim_s f^s(\overline{E}_s \wedge Y)\cong \holim_q f^q( S\Z/l^q \wedge Y)$ in $\SH(k)$.
\end{lemma}

\begin{proof} We give a sketch, indicating the necessary changes. We have the category $\Tow_{\SH(k)}$ of towers in $\SH(k)$, with $\Hom_{\Tow_{\SH(k)}}(\{W_s\}_s, \{W'_r\}_r):=\lim_r\colim_s[W_s, W'_r]$. We consider a tower $Y\to\{W_s\}_s$ under $Y$  as a map of the constant tower $\{Y\}$ to $\{W_s\}_s$. It follows directly from the definitions that if $Y\to \{W_s\}_s$, $Y\to  \{W_r'\}_r$ are $E$-nilpotent resolutions, there is a unique map $\phi: \{W_s\}_s\to  \{W_r'\}_r$  in $\Tow_{\SH(k)}$ making the diagram
\[
\xymatrix{
Y\ar[d]\ar@{=}[r]&Y\ar[d]\\
 \{W_s\}_s\ar[r]_\phi&  \{W_r'\}_r
 }
 \]
commute, in particular, $\phi$ is an isomorphism in $\Tow_{\SH(k)}$. Replacing the stable homotopy group $\pi_*$ with the bi-graded stable homotopy sheaf $\pi_{*,*}$ in the proof of \cite[5.11]{Bous}, one sees that the isomorphism $\phi$ induces an isomorphism $\holim\phi:\holim_s\{W_s\}_s\to \holim_r\{W_r'\}_r$ in $\SH(k)$. 
  
{The second assertion follows from the first together with Lemmas \ref{lem1:E-nilp res} and \ref{lem2:E-nilp res}.}
\end{proof}
 
A diagonal argument as in \cite[Theorem 24.9]{CS}
shows that
\[
\holim_s f^s( \overline{E}_s \wedge Y)
\cong \holim_s f^s \holim_q( \overline{E}_q \wedge Y)
=sc(Y^{\wedge}_E). 
\]
Using the diagonal argument again, we have
\[
\holim_q f^q( S\Z/l^q \wedge Y)
=\holim_{q} f^q( \holim_{s} S\Z/l^s \wedge Y)
=sc(Y^{\wedge}_{l} ). \]
{Applying Lemmas~\ref{lem1:E-nilp res}, \ref{lem2:E-nilp res} and \ref{lem3:E-nilp res} thus gives a canonical isomorphism $sc(Y^{\wedge}_{l})\cong sc(Y^{\wedge}_E)$.}

Recall the following  result of R\"{o}ndigs-Spitzweck-\O stv\ae r. 
\begin{theorem}[\cite{RSO}, Theorem 3.50]
Suppose $Y$ has a cell presentation of finite type. There is a canonical weak equivalence between  $sc(Y)$ and $Y^{\wedge}_{\eta}$. 
\end{theorem}
Hence, if $Y$ has a cell presentation of finite type, our isomorphism $sc(Y^{\wedge}_{l})\cong sc(Y^{\wedge}_E)$ gives rise to an isomorphism  $sc(Y^{\wedge}_{l} )\cong Y^{\wedge}_{l,\eta}$. This implies Theorem \ref{thm:eta completion} (1).

\subsection{Proof of Theorem \ref{thm:eta completion}(2)}

\begin{lemma}\label{lem:1}
{Suppose $M\in \SH(k)$ satisfies condition \eqref{eqn:fin}. Then $Y:=E\wedge M$ is slice complete, that is, the natural map $Y\to \holim_qf^q(Y)$ is an isomorphism in $\SH(k)$.}
\end{lemma}
\begin{proof}
{It suffices to show that for  $X\in \Sm_k$ and $a,b\in \Z$, the map $Y\to \holim_qf^q(Y)$ induces an isomorphism}
\[
[\Sigma^{a,b}\Sigma_{T}^{\infty} X, Y] \cong 
[\Sigma^{a,b}\Sigma_{T}^{\infty} X, \holim_qf^qY]. 
\]
By assumption, we have the decomposition $Y=\bigoplus_{\{(p_i, q_i)\}_{i\in I} } \Sigma^{(p_i, q_i)} E$, and there is an integer $s$ such that $p_i-2q_i\ge s$ for all $i\in I$. Since $E=f_0E=s_0E$ (Remark~\ref{rem:Slice}(2)), it follows that $f_{q}(Y)=\bigoplus_{q_i \ge q, i\in I} \Sigma^{ p_i, q_i} E$.   The exact triangle $f_qY\to Y\to f^{q-1} Y$ induces a long exact sequence
\[
\cdots\to [\Sigma^{a,b}\Sigma_{T}^{\infty} X, f_qY ]
\to  [\Sigma^{a,b}\Sigma_{T}^{\infty} X, Y ]
\to  [\Sigma^{a,b}\Sigma_{T}^{\infty} X, f^{q-1} Y ]
\to  [\Sigma^{a+1, b}\Sigma_{T}^{\infty} X, f_qY ]\to \cdots
\]
{Since  $\Sigma^{a, b}\Sigma_{T}^{\infty} X$ is a compact object of $\SH(k)$, we have  
\begin{multline}\label{equ:prod vanish}
[\Sigma^{a, b}\Sigma_{T}^{\infty} X, f_qY ]=
[\Sigma^{a, b}\Sigma_{T}^{\infty} X, \oplus_{q_i >q, i\in I} \Sigma^{ p_i, q_i} E ]\\
=\oplus_{ q_i \ge q, i\in I} [\Sigma_{T}^{\infty} X, \Sigma^{p_i-a, q_i-b} E ]
=\oplus_{ q_i \ge q, i\in I} H^{ p_i-a, q_i-b}(X, \Z/l).
\end{multline}
By the vanishing Theorem \ref{Thm19.3}, $[\Sigma^{a, b}\Sigma_{T}^{\infty} X, f_qY ]$ vanishes if $q>a-b-s+\dim(X)$. 
Consequently, for $q\gg 0$, we have the isomorphism $ 
[\Sigma^{a, b}\Sigma_{T}^{\infty} X, Y] \cong 
[\Sigma^{a, b}\Sigma_{T}^{\infty} X, f^{q-1}Y]$. 
In particular, for all $a,b\in \mathbb{N}$, 
$R^1 \lim_q[\Sigma^{a, b}\Sigma_{T}^{\infty} X, f^qY]=0$ and the maps $Y\to  f^qY$ induce an isomorphism $[\Sigma^{a, b}\Sigma_{T}^{\infty} X, Y] \cong 
\lim_{q}[\Sigma^{a, b}\Sigma_{T}^{\infty} X, f^qY]$. 
The desired isomorphism now follows from the short exact sequence 
\[
0\to R^1\lim{}_q[ \Sigma^{a+1, b}\Sigma_{T}^{\infty} X, f^qY]
\to [\Sigma^{a,b}\Sigma_{T}^{\infty} X, \holim_q f^qY]
\to \lim_q[\Sigma^{a,b}\Sigma_{T}^{\infty} X, f^qY]\to 0.
\]}
\end{proof}
\begin{lemma}\label{lem:2}
If $M\in \SH(k)$ satisfies condition \eqref{eqn:fin}, then $E\wedge M$ satisfies condition \eqref{eqn:fin}. 
\end{lemma}
\begin{proof}
We have 
\begin{align*}
E\wedge E\wedge M=
&E\wedge ( \bigoplus_{i\in \Lambda} \Sigma^{p_i, q_i}E)
\cong  \bigoplus_{i\in \Lambda} \Sigma^{p_i, q_i}E\wedge E\\
\cong &\bigoplus_{i\in \Lambda} \Sigma^{ p_i, q_i} ( \bigoplus_{I\in B} \Sigma^{p(I), q(I)}E)
\cong \bigoplus_{i\in \Lambda} \bigoplus_{I\in B} 
\Sigma^{ p_i+p(I), q_i+q(I)} E
\end{align*}
Here the isomorphism  $E\wedge E\cong \bigoplus_{I\in B} \Sigma^{p(I), q(I)}E$, with  $p(I)\geq 2q(I)$ for $I\in B$,  is essentially due to Voevodsky \cite[Theorem 4.46]{VoevEM}  (in characteristic zero); see  \cite[Corollary 3.4]{HSO} or \cite[Theorem 11.24]{S} for this result in arbitrary characteristic prime to $l$.    This completes the proof.
\end{proof}
\begin{lemma}\label{lem:3}
If $M\in \SH(k)$ satisfies condition \eqref{eqn:fin}, then  for any $s\in \N$, $\overline{E}_s\wedge M$ is { in the full triangulated subcategory $\calC$ of $\SH(k)$ generated by objects of} the form $E\wedge N$ with $N$ satisfying condition \eqref{eqn:fin}. 
\end{lemma}
\begin{proof}
We prove this by induction.   For $s=0$, $\overline{E}_0\wedge M=E\wedge M$, so the assertion is trivially true.  For general $s>0$, the induction hypothesis and the exact triangle $E\wedge \overline{E}^s\wedge M \to \overline{E}_s\wedge M \to \overline{E}_{s-1}\wedge M$  reduces us to showing that $E\wedge \overline{E}^s\wedge M$ is   in $\calC$. 

This  is also proved by induction in $s$. For $s=1$, we have the exact triangle 
$E\wedge \overline{E }\wedge M\to E\wedge M\to  E\wedge E\wedge M$ and both  $M$ and $E\wedge M$ satisfy condition \eqref{eqn:fin}, the latter by Lemma \ref{lem:2}. 
In general, consider the distinguished triangle
\[
E\wedge \overline{E}^{\wedge s}\wedge M \to E\wedge \overline{E}^{\wedge s-1}\wedge M \to E\wedge E\wedge\overline{E}^{\wedge s-1}\wedge M.\] 
By the induction hypothesis  $E\wedge \overline{E}^{\wedge s-1}\wedge M$ is in $\calC$. By Lemma \ref{lem:2}, $\calC$ is closed under the operation $M\mapsto E\wedge M$, hence 
$E\wedge E\wedge\overline{E}^{\wedge s-1}\wedge M$ is also in $\calC$ and thus 
$E\wedge \overline{E}^{\wedge s}\wedge M$ is in $\calC$.
\end{proof}

We note that the slice complete objects in $\SH(k)$ form the objects in a full triangulated subcategory of $\SH(k)$. Thus, by Lemma \ref{lem:1}, an object in the category $\calC$ of Lemma \ref{lem:3} is slice complete and thus by Lemma \ref{lem:3}, $\overline{E}_s\wedge Y$ is slice complete for every $s\ge0$.  
As taking slice completion commutes with homotopy limits, it follows  by a diagonalization argument that $Y^{\wedge}_E$ is slice complete, completing the proof of Theorem \ref{thm:eta completion}(2).

\section{A proof of Novikov's lemma}\label{App:Novikov}
Novikov \cite[Lemma 16]{Nov} states without proof a description of $H^*(\MSU, \mathbb{Z}/l)$ as a module over the mod $l$ Steenrod algebra. This result plays a central role in our work, and we were unable to find a proof in the literature, so we include a proof here. We also use a description of an additive basis of 
$H^*(\MSU, \mathbb{Z}/l)$ in terms of ``$l$-admissible partitions'' in defining the virtual partition $\yu_r$ (Definition~\ref{def:NovikovClass}), and this basis is needed in our proof of Novikov's lemma, so we give a proof of this fact as well.

We retain the notations for the classical Steenrod algebra $A^{\topo}$ and its quotient $M_B^{\topo}$ from \S~\ref{sec:SteenMod}.

\subsection{The $\mathbb{Z}/l$ basis of $H^*(\MSU, \mathbb{Z}/l)$}
\begin{definition}
Let $l$ be a prime number.  A partition $\omega = (a_1, \ldots , a_s)$ is called $l$-admissible if the number 
 \[
\# \{i\mid a_i=l^r\} 
 \]
 is a multiple of $l$, for any $r\geq 0$. 
In particular, when $r=0$, the number $\# \{i\mid a_i=1\}$  is a multiple of $l$. 
\end{definition}
\begin{example}\label{ex:p large}
When $l> \sum_{i=1}^s a_i$, the condition of $l$-admissibility is equivalent to the condition that
\[
a_i\neq 1, \text{for any $i=1, \ldots, s$.}
\]
When a field is of characteristic zero, we use the convention that $l>>0$. 
This means that $a_i \neq 1$ for $i = 1, \ldots, s$ for a field of characteristic zero. 
\end{example}

\begin{prop}\label{prop:Admissible}
$H^*(\MSU, \mathbb{Z}/l)= \lim_k \mathbb{Z}/l [t_1, t_2, \ldots, t_k]^{S_k} /(t_1+t_2+\cdots +t_k)= \mathbb{Z}/l[c_2, c_3, \ldots]$ has a $\mathbb{Z}/l$ basis given by monomial symmetric functions corresponding to $l$-admissible partitions. 
\end{prop}
We first show that the monomial symmetric functions associated to $l$-admissible partitions span $H^*(\MSU, \mathbb{Z}/l)$.  

\begin{lemma}\label{lem:division}
Let $\omega$ be a partition of length at most $k$ and let  $u_{\omega}$ be the corresponding monomial symmetric function in variables $t_1, \ldots, t_k$. Then we have a decomposition
\[
u_{\omega}=\sum_{\omega'} u_{\omega'} + (t_1+t_2+\cdot +t_k) f(t_1, \ldots, t_k), 
\]
where each $\omega'$ is an $l$-admissible partition, and $f(t_1, \ldots, t_k)\in \mathbb{Z}/l[t_1, t_2, \ldots t_k]^{S_k}$. 
\end{lemma}
\begin{proof}
If $\omega$ is already an $l$-admissible partition, we can choose $f$ to be zero. This gives the above decomposition. 

Assume $\omega$ is not $l$-admissible, then start with the smallest $r_1$, such that 
\[
\# \{i\mid a_i=l^{r_1}\}
\] is coprime to $l$. 
The  monomial symmetric function $u_{\omega}$ is of the form
\[
u_{\omega}=\sum  t_1^{l^{r_1}} t_2^{l^{r_1}} \cdots t_b ^{l^{r_1}} t_{b+1} ^{a_{b+1}}\cdots t_s^{a_s}  
\]
Working over $\mathbb{Z}/l$, we have the identity
\[
(t_1+\cdots+t_k)^{l^{r}}=t_1^{l^r}+t_2^{l^r}+\cdots+t_k^{l^r}. 
\]
We now compute the following difference (here the $\Sigma$ is the symmetrization of a monomial to the corresponding  monomial symmetric function)
\begin{align}
&(t_1+\cdots+t_k)^{l^{r_1}} \cdot ( \Sigma t_2^{l^{r_1}} \cdots t_b ^{l^{r_1}} t_{b+1} ^{a_{b+1}}\cdots t_s^{a_s} )-u_{\omega} \notag\\
=&(t_1^{l^{r_1}}+t_2^{l^{r_1}}+\cdots+t_k^{l^{r_1}}) \cdot ( \Sigma t_2^{l^{r_1}} \cdots t_b ^{l^{r_1}} t_{b+1} ^{a_{b+1}}\cdots t_s^{a_s} )-u_{\omega} \notag\\
=&\Sigma t_2^{2l^{r_1}} \cdots t_b ^{l^{r_1}} t_{b+1} ^{a_{b+1}}\cdots t_s^{a_s}+
\Sigma t_2^{l^{r_1}} \cdots t_b ^{l^{r_1}} t_{b+1} ^{a_{b+1}+l^{r_1}}\cdots t_s^{a_s}+
\cdots +\Sigma t_2^{l^{r_1}} \cdots t_b ^{l^{r_1}} t_{b+1} ^{a_{b+1}}\cdots t_s^{a_s+l^{r_1}} \label{division}
\end{align}
In the above linear combinations, we have the following partitions
\begin{align}
&(2\cdot l^{r_1}, l^{r_1}, \ldots, l^{r_1}, a_{b+1}, \ldots, a_{s}) \,\  \,\   \,\ (\text{$b-2$ copies of $l^{r_1}$}), \\
&(l^{r_1}, \ldots, l^{r_1}, a_{b+1}+l^{r_1}, \ldots, a_{s}) \,\  \,\   \,\ (\text{$b-1$ copies of $l^{r_1}$})\\
&\ldots \\
&(l^{r_1}, \ldots, l^{r_1}, a_{b+1}, \ldots, a_{s}+l^{r_1}) \,\  \,\   \,\ (\text{$b-1$ copies of $l^{r_1}$})
\end{align}
To summerise, applying the process \eqref{division} to $\omega=(l^{r_1}, l^{r_1}, \ldots, l^{r_1}, a_{b+1}, \ldots, a_{s})$, the number of copies of $l^{r_1}$ in the resulting partitions  decreases (by either one or two), and  none of these new partitions involve $l^r$ with $r<r_1$. 
We repeat the process \eqref{division} to each of the new partitions, until the number of copies of $l^{r_1}$ in each is a multiple of $l$ or is zero.   

After that, we repeat the process \eqref{division} to the next minimal $r_2$ (necessarily $r_2>r_1$), such that, $\# \{i\mid a_i=l^{r_2}\}$ is coprime to $l$. 
We repeat the process until   $\# \{i\mid a_i=l^{r}\}$ is a multiple of $l$, for all $r \geq 0$. The resulting partitions are thus all $l$-admissible. 
\end{proof}

We now show that the monomial symmetric functions for $l$-admissible partitions form a basis of  $H^*(\MSU, \mathbb{Z}/l)$ by computing the respective Hilbert series (as all the cohomology rings we are considering are concentrated in even degree, we will use half the cohomology degree as the grading degree). 
Note that the Hilbert series of $\mathbb{Z}/l[c_2, c_3, \ldots]$ is $\prod_{i=2}^{\infty} (1-t^i)^{-1}$. 
 
\begin{prop}\label{prop:admissible} Let $l$ be a prime.\\[2pt]
1. We have the identity
\[
\prod_{i=2}^{\infty} (1-t^i)^{-1}
=\sum_{n\ge0} \#\{\text{ $l$-admissible partitions of $n$} \}t^n. 
\]
In particular, the right hand side is independent of $l$. \\[2pt]
2. For $l\neq 2$, and for all $n\ge0$,
\begin{multline}\label{eqn:ind}
\#\{\omega=(a_1,\ldots, a_s) \vdash n\mid a_i\neq 1, i=1, \ldots, s, \text{and $\omega$ is non-$l$ adic}\}\\
=
\#\{\omega \vdash n\mid \text{$\omega$ is non-$l$ adic and is $l$-admissible}\}.
\end{multline}

\end{prop}
\begin{proof} 1. In the product $\prod_{i=2}^\infty(1-t^i)^{-1}$, the term $t^{im}$ in the factor $(1-t^i)^{-1}=\sum_{n=0}^\infty t^{in}$ contributes to the count of partitions containing $i$ exactly $m$ times.  Thus, we can write the generating function $\sum_{n\ge0} \#\{\text{ $l$-admissible partitions of $n$} \}t^n$ as
\begin{align*}
\sum_{n\ge0} \#\{\text{ $l$-admissible partitions of $n$} \}t^n&=\prod_{i\ge 2, i\neq l^r, r\ge1}(1-t^i)^{-1}\cdot \prod_{r=0}^\infty(\sum_{m=0}^\infty t^{l^r\cdot l\cdot m})\\
&=\prod_{i\ge 2, i\neq l^r, r\ge1}(1-t^i)^{-1}\cdot \prod_{r=1}^\infty(\sum_{m=0}^\infty t^{l^r\cdot m})\\
&=\prod_{i\ge 2}(1-t^i)^{-1}
\end{align*}

The proof of (2), using generating functions, is the same as for (1), we just delete from the product description of the respective generating functions the terms $(1-t^i)^{-1}$ for $i=l^r-1$, $r\ge1$.
\end{proof}

\subsection{The $A^{\topo}$-module structure of  $H^*(\MSU, \mathbb{Z}/l)$}\label{AppSubsec:Novikov}
We recall the statement of Novikov's lemma. We have the quotient $M_B^{\topo}:=A^{\topo}/(Q_0, Q_1,\ldots)$ of the classical mod $l$ Steenrod algebra $A^{\topo}$.

\begin{lemma*}[Novikov \hbox{\cite[Lemma 16]{Nov}}] For each non $l$-adic, $l$-admissible partition $\omega$, the map $M_B^{\topo}\to H^*(\MSU, \mathbb{Z}/l)$ sending $P$ to $P(u_\omega)$ is a well-defined injective $A^{\topo}$-module map and induces a decomposition as $A^{\topo}$-module 
\[
H^*(\MSU, \mathbb{Z}/l)=\oplus_{\omega\mid  \omega\text{ non $l$-adic, $l$-admissible}}M_B^{\topo}u_\omega.
\]
\end{lemma*}
The proof follows in a number of steps.

\begin{prop}\label{prop:identity} Let $n\ge0$ be an integer. We have the identity
\[
\sum_{i=0}^n \#\{\omega \vdash i\mid \text{$\omega$ is non $l$-adic and is $l$-admissible}\}
=\#\{\omega \vdash n \mid \text{$\omega$ is non $l$-adic}\}
\]
\end{prop}
\begin{proof}
Let $F(t)$, $G(t)$, $H(t)$ be the generating functions
\begin{align*}
F(t)&=\sum_{n\ge0} \#\{\omega \vdash n\mid \text{$\omega$ is non $l$-adic and is $l$-admissible}\}\cdot t^n\\
G(t)&=\sum_{n\ge0}(\sum_{i=0}^n \#\{\omega \vdash i\mid \text{$\omega$ is non $l$-adic and is $l$-admissible}\})\cdot t^n\\
H(t)&=\sum_{n\ge 0} \#\{\omega \vdash n \mid \text{$\omega$ is non $l$-adic}\}\cdot t^n
\end{align*}
By Proposition~\ref{prop:admissible}, $F(t)=\prod_{i\ge 2, i\neq l^r-1, r\ge1}(1-t^i)^{-1}$. But $G(t)=(1-t)^{-1}F(t)$, which is equal to $H(t)$. 
\end{proof}

For a   graded $\Z/l$-module $M=\oplus_{n\ge0}M_n$  with finite dimensional summands $M_n$,   we have the Hilbert series $P_t(M):=\sum_{n\ge0}\dim_{\Z/l} M_n\cdot t^n$. 
\begin{prop}\label{prop:HilbertSeries}
For each partition $\omega$, we have the submodule $M^{\topo}_B u_{\omega}\subset H^*(\MU,\Z/l)$, with grading induced by the grading  in $H^*(\MU,\Z/l)$: $(M^{\topo}_B u_{\omega})_n=M^{\topo}_B u_{\omega}\cap H^{2n}(\MU,\Z/l)$. Then, 
\[
P_t( \oplus_{\omega\mid \text{$\omega$ is non $l$-adic, $l$-admissible}} M^{\topo}_Bu_{\omega})=\prod_{i=2}^{\infty} (1-t^i)^{-1}
\]
\end{prop}
\begin{proof} We note that for $\omega\vdash n$ non $l$-adic, $P_t(M^{\topo}_Bu_\omega)=t^n\cdot P_t(M^{\topo}_B)$, as the multiplication map $M^{\topo}_B\to M^{\topo}_Bu_\omega$ is an isomorphism (\cite[Lemma 4]{Nov} or \cite[Theorem 2]{Mil2}). 
By Proposition \ref{prop:identity} we have
\begin{align}
(1-t)^{-1}P_t( \oplus_{\omega\mid \text{$\omega$ is non $l$-adic, $l$-admissible}} M^{\topo}_B u_{\omega})
=P_t( \oplus_{\omega\mid \text{$\omega$ is non $l$-adic}} M^{\topo}_B u_{\omega}). \label{eq:1/1-t}
\end{align}
Since $H^*(\MU)=\oplus_{\omega\text{ non $l$-adic}}M^{\topo}_B u_{\omega}$ (again, \cite[Lemma 4]{Nov} or \cite[Theorem 2]{Mil2}), we have
\[
P_t(\oplus_{\omega\mid \text{$\omega$ is non $l$-adic}} M^{\topo}_B u_{\omega})= \prod_{i=1}^{\infty} (1-t^i)^{-1}. 
\]
The conclusion now follows by cancelling the factor $(1-t)^{-1}$ from both sides of \eqref{eq:1/1-t}. 
\end{proof}

The natural map  $H^*(\MU,\Z/l)\to H^*(\MSU,\Z/l)$ gives us the well-defined $A^{\topo}$-module homomorphism
\[
 \oplus_{\omega\mid \text{$\omega$ is non $p$-adic, $p$-admissible}} M^{\topo}_Bu_{\omega}
 \to H^*(\MSU,\Z/l). 
\] 
By Proposition~\ref{prop:HilbertSeries}, Novikov's lemma will follow once we show that this map is surjective, which we now proceed to do 

We have the following commutative diagram
\[
\xymatrix{
A^{\topo}\times H^*(\MU, \mathbb{Z}/l) \ar[r] \ar@{->>}[d]&H^*(\MU, \mathbb{Z}/l) \ar@{->>}[d]\\
A^{\topo}\times H^*(\MSU, \mathbb{Z}/l) \ar[r]  &H^*(\MSU, \mathbb{Z}/l)\\
}
\]
For each $u_{\omega}$ such that $\omega$ is $l$-admissible,
we lift $u_{\omega}$ to $H^*(\MU, \Z/l)$; we continue to denote this lifting by $u_{\omega}$. 
We have the decomposition
\begin{equation}\label{eq:non p admissible}
u_{\omega}=\sum_{\omega'} P^{R_{\omega'}} u_{\omega'}, 
\end{equation}
where $\omega'$ is non-$l$-adic, but could be non-$l$-admissible. 
The degree of $u_{\omega}$ is strictly larger than the degree of $u_{\omega'}$, unless $\omega$ is already non-$l$-adic. 
Indeed, all $R_{\omega'}$ from the right hand side of \eqref{eq:non p admissible} is nonzero. 
Otherwise, we have the decomposition
\begin{align}\label{eq:omega_1}
u_{\omega}= u_{\omega_1} + \sum_{\omega''} P^{R_{\omega''}} u_{\omega''}, 
\end{align}
where $\omega_1, \omega''$ are non-$l$-adic partitions, and $\omega\neq \omega_1$. 
Note that $u_{\omega_1}$ and $ \sum_{\omega''} P^{R_{\omega''}} u_{\omega''}$ are in two different summands  of
\[
 \oplus_{\omega\mid \text{$\omega$ is non $p$-adic}} M_B^{\topo} u_{\omega}=H^*(\MU,\Z/l). 
\]
Thus, $u_{\omega_1}$ and $ \sum_{\omega''} P^{R_{\omega''}} u_{\omega''}$ are $\Z/l$ linearly independent.  Since $\{u_{\omega'}\}$, as $\omega'$ runs over all partitions,  is a $\Z/l$ basis of $H^*(\MU,\Z/l)$, the decomposition \eqref{eq:omega_1} forces $\omega=\omega_1$,  a contradiction. 
The equality \eqref{eq:non p admissible} is homogenous, and each $R_{\omega'}$ is non-zero. 
Therefore
\[
\deg(u_{\omega})> \deg(u_{\omega'}), \text{unless $\omega$ is non $l$-adic}. 
\]
For those non $l$-admissible $u_{\omega'}$, we apply Lemma \ref{lem:division} to change $u_{\omega'}$ to admissible partitions $u_{\omega''}$ ( $u_{\omega''}$ could be $l$-adic), where 
\[
\deg(u_{\omega})> \deg( u_{\omega'})=
\deg( u_{\omega''}). 
\]
Now we repeat the above process to $u_{\omega''}$, and each time the degree is decreasing. 
When $\deg(u_{\omega})=1$, we have $u_{\omega}=t_1+t_2+\cdots+t_k$,  which is zero. 
This finishes the proof. 

\newcommand{\arxiv}[1]
{\texttt{\href{http://arxiv.org/abs/#1}{arXiv:#1}}}
\newcommand{\doi}[1]
{\texttt{\href{http://dx.doi.org/#1}{doi:#1}}}
\renewcommand{\MR}[1]
{\href{http://www.ams.org/mathscinet-getitem?mr=#1}{MR#1}}

\end{document}